\documentclass[reqno, 11pt, a4paper]{amsart}

\usepackage{
a4wide, 
amssymb,
bbm,
centernot,
mathtools,
calrsfs,
xcolor
} 
\usepackage{graphicx}
\graphicspath{{./}}
\usepackage[cal=euler, scr=boondoxo]{mathalfa}
\usepackage[colorlinks=true, linkcolor=blue,citecolor=blue]{hyperref}
\usepackage{float}
\usepackage{enumitem} 
\usepackage{empheq} 
\newtheorem{theorem}{Theorem}[section]
\newtheorem{lemma}[theorem]{Lemma}
\newtheorem{prop}[theorem]{Proposition}
\newtheorem{corollary}[theorem]{Corollary}
\newtheorem{proposition}[theorem]{Proposition}

\theoremstyle{definition}

\newtheorem*{assumption}{Main Assumption}

\theoremstyle{remark}
\newtheorem{remark}[theorem]{Remark}

\usepackage{xspace} 
\newcommand{\refPone}{\textbf{\hyperref[P1]{P1}}\xspace}
\newcommand{\refPtwo}{\textbf{\hyperref[P2]{P2}}\xspace}
\newcommand{\refPthree}{\textbf{\hyperref[P3]{P3}}\xspace}
\newcommand{\refD}{\textbf{\hyperref[D]{D}}\xspace}
\newcommand{\refSone}{\textbf{\hyperref[S1]{S1}}\xspace}
\newcommand{\refStwo}{\textbf{\hyperref[S2]{S2}}\xspace}
\newcommand{\refA}{\hyperref[A]{Main Assumption}\xspace}

\numberwithin{equation}{section}

\newcommand{\IND}{\mathbbm{1}}

\newcommand{\capa}{\mathrm{Cap}}

\newcommand{\De}{\mathrm{d}}

\newcommand{\Res}{\mathrm{Res}}
\newcommand{\DeltaS}{\Delta_{\mathrm{S}}}
\newcommand{\Rden}{R_{\mathrm{den}}}
\newcommand{\Rhk}{R_{\mathrm{hk}}}
\newcommand{\Rkhk}{R_{\mathrm{khk}}}
\newcommand{\cA}{\ensuremath{\mathcal A}}

\newcommand{\cC}{\ensuremath{\mathcal C}}

\newcommand{\cF}{\ensuremath{\mathcal F}}

\newcommand{\cH}{\ensuremath{\mathcal H}}

\newcommand{\cL}{\ensuremath{\mathcal L}}

\newcommand{\cQ}{\ensuremath{\mathcal Q}}

\newcommand{\cS}{\ensuremath{\mathcal S}}
\newcommand{\cT}{\ensuremath{\mathcal T}}
\newcommand{\cU}{\ensuremath{\mathcal U}}

\newcommand{\bbG}{\ensuremath{\mathbb G}}

\newcommand{\bbN}{\ensuremath{\mathbb N}}

\newcommand{\bbP}{\ensuremath{\mathbb P}}

\newcommand{\bbR}{\ensuremath{\mathbb R}}

\newcommand{\bbZ}{\ensuremath{\mathbb Z}}
\begin{document}

\definecolor{airforceblue}{RGB}{204, 0, 102}
\newenvironment{draft}
{\par\medskip
\color{airforceblue}%
\medskip}

\title[Solidification estimates for random walks on percolation clusters]{Solidification estimates for random walks on supercritical percolation clusters}


\author{Alberto Chiarini}
\address{Universit\`a degli Studi di Padova}
\curraddr{Department of Mathematics ``Tullio Levi-Civita'', via Trieste 63, 35121, Padova}
\email{chiarini@math.unipd.it}
\thanks{}


\author{Zhizhou Liu}
\address{Department of Mathematics, The Hong Kong University of Science and Technology}
\curraddr{Clear Water Bay, Kowloon, Hong Kong}
\email{zliugm@connect.ust.hk}
\thanks{}


\author{Maximilian Nitzschner}
\address{Department of Mathematics, The Hong Kong University of Science and Technology}
\curraddr{Clear Water Bay, Kowloon, Hong Kong}
\email{mnitzschner@ust.hk}
\thanks{}

\begin{abstract}
We consider the simple random walk on the infinite cluster of a general class of percolation models on $\mathbb{Z}^d$, $d\geq 3$, including Bernoulli percolation as well as models with strong, algebraically decaying correlations. For almost every realization of the percolation configuration, we obtain uniform controls on the absorption probability of a random walk by certain ``porous interfaces'' surrounding the discrete blow-up of a compact set $A$. These controls substantially generalize previous results obtained in~\cite{nitzschner2017solidification} for Brownian motion in $\mathbb{R}^d$ and in~\cite{CN2021disconnection} for random walks on $\mathbb{Z}^d$ equipped with uniformly elliptic edge weights to a manifestly non-elliptic framework. 
\end{abstract}

\subjclass[2010]{}
\dedicatory{}
\maketitle
\tableofcontents

\section{Introduction}
\label{sec:introduction}

In this article, we study the simple random walk on the infinite connected component $\cS_\infty$ of a typical realization of a supercritical percolation configuration on $\mathbb{Z}^d$, $d \geq 3$, for a large class of percolation models. The latter contains Bernoulli bond- or site-percolation, as well as models with strong, algebraically decaying correlations fulfilling certain structural assumptions taken from~\cite{alves2019,drewitz2014chemical}, including the vacant set of random interlacements, the level sets of the Gaussian free field, and the range or vacant set of a random walk loop soup. \medskip

As our main result, we consider the trace on $\mathcal{S}_\infty$ of the discrete blow-up $A_N$ of a compact set $A \subseteq \mathbb{R}^d$ with non-empty interior, $A_N \cap \cS_\infty$, and prove \textit{uniform absorption estimates} for a random walk started in $A_N \cap \cS_\infty$ by a class of \textit{porous interfaces} $\Sigma \subseteq \mathcal{S}_\infty$ surrounding $A_N$. Estimates of this type were first developed in~\cite{nitzschner2017solidification} in the context of Brownian motion and later extended in~\cite{CN2021disconnection} to random walks on $\mathbb{Z}^d$, $d \geq 3$, with uniformly elliptic weights $\lambda \in [\lambda_{\min},1]^{\mathbb{E}^d}$ attached to the edges $\mathbb{E}^d$ (here $\lambda_{\min} \in (0,1)$). Together with related lower bounds on the capacity of the porous interfaces $\Sigma$, these \textit{solidification estimates} have been instrumental as a substitute for a Wiener-type criterion in deriving large deviation upper bounds for the probability of various events characterized by atypical
isolation-type phenomena of macroscopic bodies, see in particular~\cite{chiarini2020GFF,chiarini2020entropic,
CN2021disconnection,nitzschner2018disconnection,
nitzschner2017solidification, sznitman2019macroscopic} (but see also Remark~\ref{rem:Applications-generalizations} for more on these applications).  While in~\cite{CN2021disconnection} the aforementioned solidification estimates could be obtained for \textit{every} fixed choice of edge weights $\lambda \in [\lambda_{\min},1]^{\mathbb{E}^d}$, leveraging quenched heat kernel bounds for random walks on $(\mathbb{Z}^d,\mathbb{E}^d,\lambda)$, the present set-up is markedly more delicate. This is due to an effect of \textit{local degeneracies} of the percolation cluster, rendering the use of uniform controls even in the simpler case of Bernoulli bond percolation infeasible. Informally, while many \textit{global properties} of $\mathcal{S}_\infty$ exhibit the same behavior as on $\mathbb{Z}^d$ (including chemical distances, volume regularity, quenched invariance principles, and Gaussian heat kernel bounds for the random walk, see, e.g.,~\cite{antal1996chemical,armstrong2018elliptic,
barlow2004RWpercolation, berger2007quenched,dario-gu2021percolation, mathieu2007quenched, sidoravicius2004quenched}, as well as~\cite{andres2025scaling} for the scaling limit of the Gaussian free field on a large class of percolation clusters), the presence of regions of small volume or irregular heat kernel qualitatively changes certain characteristics of the cluster. These include, in the case of Bernoulli percolation, the maximal effective resistance between points of the largest connected component of the cluster in large boxes, see~\cite{abe2015effective}, or the phases of the polymer model on $\mathcal{S}_\infty$, see~\cite{cosco2021directed,
nitzschner2022polymer}. In some cases, a behavior comparable to the full lattice is retained for $p$ close enough to one, see for instance~\cite{schweiger2024maximum} for the maximum of the Gaussian free field on a two-dimensional Bernoulli bond percolation cluster. However, removing the effect of locally bad behavior would require a very fine understanding of the local geometry of $\mathcal{S}_\infty$ (see Remark 4.15 of the same reference). Our main contribution may therefore be seen as a \textit{structural property} of $\mathcal{S}_\infty$ for a large class of models. Notably, our methods are applicable throughout the entire supercritical phase of Bernoulli bond percolation, as well as the (strongly) supercritical phase of a wide class of (possibly correlated) site percolation models. A major challenge inherent in our present set-up is to verify quantitatively that the influence of atypically bad regions that are present in $\mathcal{S}_\infty$ is \textit{not} felt by the random walk on sufficiently many scales. Crucially, our controls must be strong enough to deploy the construction of a certain generic \textit{resonance set}, which is hard to avoid for a random walk starting in $A_N \cap \cS_\infty$. \smallskip

We now state our results in more detail. Consider $\mathbb{Z}^d$, $d\geq 3$, and a one-parameter family of probability measures $\mathbb{P}^u$, $u \in (a,b)$, on $\{0,1\}^{\mathbb{Z}^d}$ governing a site percolation model on $\mathbb{Z}^d$ fulfilling the regularity properties  \refPone--\refPtwo, \refD, and \refSone--\refStwo, where we declare sites with label one to be open (see Subsection~\ref{subsec:ModelAssumptions} for a precise definition). We will turn to bond percolation models later in the introduction. The parameters $a,b$ are chosen to guarantee that the model is in a (strongly) supercritical phase. We denote the random set of open vertices by $\mathcal{S} \subseteq \mathbb{Z}^d$ and write 
\begin{equation}
\mathcal{S}_\infty = \text{ the unique infinite cluster of the subgraph of $\mathbb{Z}^d$ induced by $\mathcal{S}$,}
\end{equation}
which exists with $\mathbb{P}^u$-probability one under our standing assumptions. The set of assumptions is taken from~\cite{alves2019}, which itself is adapted from a slightly more restrictive set of assumptions that appeared in~\cite{drewitz2014chemical}, and includes a large class of pertinent site percolation models, see Remark~\ref{rmk:examples}.  Informally, assumptions \refPone, \refPtwo, and \refD correspond to ergodicity of the lattice shift operator, monotonicity in the percolation parameter $u$, and the availability of certain sprinkled decoupling inequalities (see~\eqref{eq:Decoup-A}, \eqref{eq:Decoup-B}), respectively. The properties \refSone and \refStwo guarantee the local uniqueness of the infinite cluster in large boxes, and the continuity of the density of the infinite cluster $u \mapsto \mathbb{P}^u[0 \in \mathcal{S}_\infty]$. For illustrative purposes of this introduction, one may restrict the attention to the family $(\mathbb{P}^u)_{u \in (p_c,1)}$ governing supercritical Bernoulli site percolation at varying levels $u$, with $p_c$ denoting the corresponding percolation threshold. Let $A \subseteq \mathbb{R}^d$ be compact with non-empty interior and denote its discrete blow-up by 
\begin{equation}
\label{eq:Discrete-Blow-up}
A_N = (NA) \cap \mathbb{Z}^d.
\end{equation}
We consider below a class of bounded sets $U_0 \subseteq \mathcal{S}_\infty$ with the exterior boundary $S=\partial_{\cS_\infty} U_0$ (taken with respect to $\mathcal{S}_\infty$) of $U_0$ being a sort of ``hard'' interface at a certain distance from $A_N$. Informally, $S$ can be thought of as being located at sup-norm distance $2^{\ell_\ast}$ of the set $A_N\cap\cS_\infty$ for an integer $\ell_\ast \geq 0$. To deal with the presence of ``sizeable, far-away bad regions'' of $\mathcal{S}_\infty$, we introduce a ``growth condition'' limiting the possible values of the distance parameter $2^{\ell_\ast}$ and the size of $U_0$. More precisely, we will assume the following:
\begin{equation}
\begin{minipage}{0.8\linewidth}
 \label{eq:a_N-bound}
Let $(a_N)_{N\geq 0}$ and $(b_N)_{N\geq 0}$ be two sequences of positive real numbers with $a_N,b_N\nearrow \infty$, $b_N/a_N \nearrow \infty$, and
  \[
  a_N / \exp\left\{\log (b_N)^{\frac{1}{1+\DeltaS/2}}\right\}  \nearrow \infty,
  \]
  as $N\to\infty$,
\end{minipage}
\end{equation}
where $\DeltaS = \DeltaS(u)$ is a parameter, possibly depending on $u$, entering the definition of \refSone, see~\eqref{eq:funcS}. In fact, we can assume without loss of generality that
\begin{equation}
\label{eq:b_N-at-least-linear}
b_N \geq \widetilde{c}(A) N, \qquad \text{for all }N \geq 1,
\end{equation} 
with a constant $\widetilde{c}(A)$ depending only on the set $A$. Heuristically, $a_N$ and $b_N$ should be viewed as lower and upper bounds for the distance $2^{\ell_\ast}$, respectively, and the condition in~\eqref{eq:a_N-bound} is chosen in such a way that good controls on the heat kernel of the random walk and volume regularity hold uniformly across all boxes of sizes essentially between $a_N$ and $b_N$ and whose centers lie in a sup-norm ball  of radius $b_N$, centred at the origin, 
for almost every realization of $\omega$, when $N$ is sufficiently large. 
With these preparations we formally introduce for any given ``typical'' percolation configuration $\omega\in \{0,1\}^{\mathbb{Z}^d}$ (so that $\cS_\infty$ is non-empty and unique) and an integer $\ell_\ast \geq 0$ the class 
\begin{equation}\label{eq:cal_U-intro}
\begin{split}
\mathcal{U}_{\ell_\ast,N}^\omega = & \;\mbox{the collection of sets $U_0 \subseteq B(0,b_N) \cap \mathcal{S}_\infty$ for which the}
\\[-0.5ex]
&\; \mbox{local density of $U_1=\cS_\infty \setminus U_0$ for any box $B(x,2^\ell) \cap \cS_\infty$ }
\\[-1ex]
&\; \mbox{centered at a point $x \in A_N\cap\cS_\infty$ is at most $\frac{1}{2}$ when $\ell \leq \ell_\ast$}
\end{split}
\end{equation}
(here, $B(z,R)$ stands for the closed sup-norm ball in $\mathbb{Z}^d$ centered at $z \in \mathbb{Z}^d$ with radius $R \geq 0$).  We then define the porous interfaces attached to a given $U_0 \in \mathcal{U}^\omega_{\ell_\ast,N}$, which may be viewed as deformations of the hard interface $S$. The former come with a certain integer ``distance'' parameter $\epsilon$ (limiting the sup-norm distance to $S$) and ``strength'' $\chi \in (0,1)$ (measuring how much the deformation is felt for a walk starting in $S$), and the set of all such interfaces is given by
\begin{equation}
\begin{split}
\mathcal{S}^\omega_{U_0,\epsilon,\chi} = & \;\mbox{the collection of bounded subsets $\Sigma \subseteq \mathcal{S}_\infty$ such that for all $x \in S$}
\\
&\; \mbox{$P_x^\omega[\text{Random walks enters $\Sigma$ before going to distance $\epsilon$}] \geq \chi$},
\end{split}
\end{equation}
where $P^\omega_x$ is the canonical measure governing a (continuous-time, constant-speed) simple random walk on $\mathcal{S}_\infty$, starting from $x \in \mathcal{S}_\infty$ (see Subsection~\ref{subsec:random-walks} for a precise definition). \medskip
 
\begin{figure}[htbp]
\begin{center}
\includegraphics[scale=0.65]{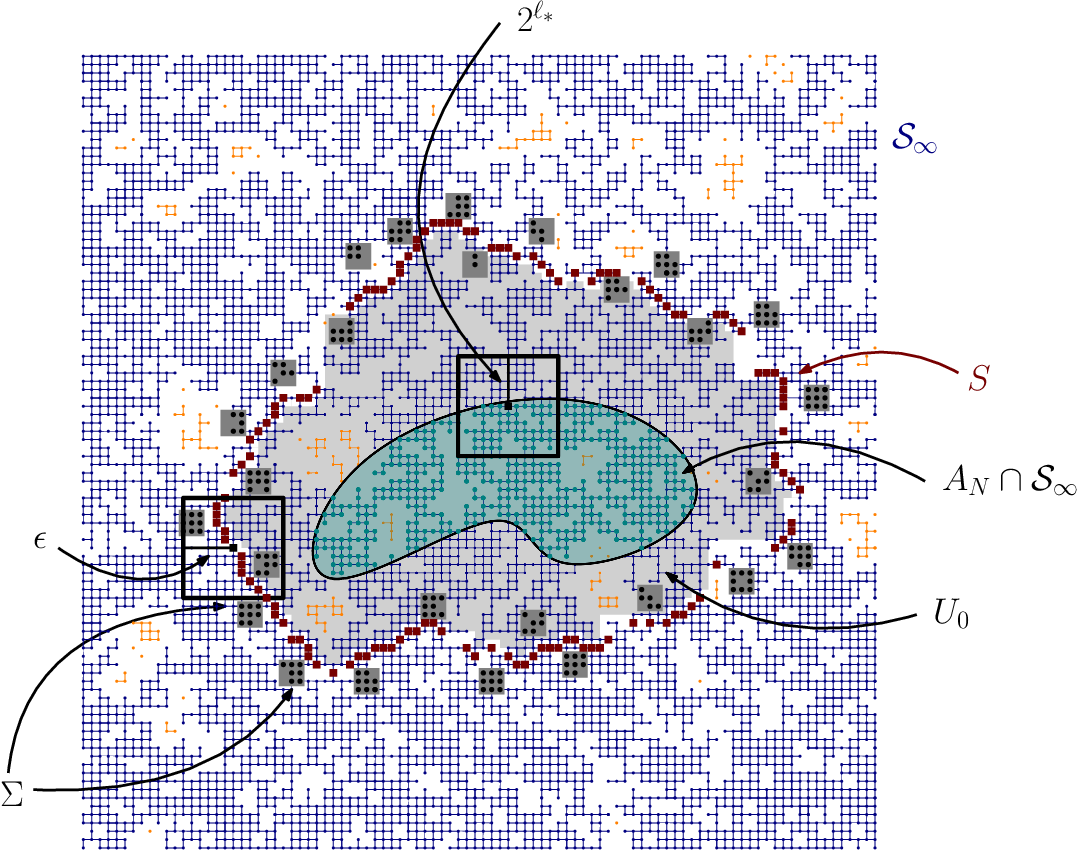}
\end{center}
\caption{Schematic illustration of a possible choice of $U_0$ with boundary $S = \partial_{\cS_\infty} U_0$ and $\Sigma$ in $\cS^\omega_{U_0,\epsilon,\chi}$.  In the picture, $U_0$ appears as the intersection between $\mathcal{S}_\infty$ and the shaded area, and $S$ consists of the squares at the boundary of the shaded area.}
\label{fig:interface}
\end{figure}

Our main result appears in Theorem~\ref{thm:solidification}. It states that for a given $A \subseteq \mathbb{R}^d$ as above~\eqref{eq:Discrete-Blow-up} and $\chi \in (0,1)$, uniformly over pairs of integers $(\epsilon,\ell_\ast)$ with $\epsilon/2^{\ell_\ast}$ tending to zero with $N$, $2^{\ell_{\ast}} \in [a_N,b_N]$ with $a_N,b_N$ fulfilling~\eqref{eq:a_N-bound} (and~\eqref{eq:b_N-at-least-linear}), $U_0 \in \cU^\omega_{\ell_\ast,N}$, and $\Sigma \in \mathcal{S}^\omega_{U_0,\epsilon,\chi}$, a random walk starting in $A_N \cap \cS_\infty$ asymptotically cannot avoid $\Sigma$, for almost every realization of the percolation configuration $\omega$. More precisely, we have
\begin{equation}
\label{eq:Solidification-intro}
\lim_{N \to \infty} \widetilde{\sup} \sup_{x \in A_N \, \cap \, \mathcal{S}_\infty} P_x^\omega[H_\Sigma  = \infty] = 0, \qquad \text{for $\mathbb{P}^u$-a.e.~$\omega \in \{0,1\}^{\mathbb{Z}^d}$}
\end{equation}
(and $\widetilde{\sup}$ stands for the supremum over all $(\epsilon,\ell_\ast)$ with $\epsilon/2^{\ell_\ast}$ tending to zero with $N$, $2^{\ell_{\ast}} \in [a_N,b_N]$ with $a_N,b_N$ fulfilling~\eqref{eq:a_N-bound} (and~\eqref{eq:b_N-at-least-linear}), $U_0 \in \mathcal{U}^\omega_{\ell_\ast,N}$, and $\Sigma \in \mathcal{S}^\omega_{U_0,\epsilon,\chi}$). We show in Corollary~\ref{cor:Capacity-Lower-Bound} that~\eqref{eq:Solidification-intro} implies the capacity lower bound
\begin{equation}
\label{eq:Capacity-lower-bound-intro}
\liminf_{N \to \infty} \widetilde{\inf} \frac{\capa^\omega(\Sigma)}{\capa^\omega(A_N \cap \cS_\infty)}  \geq 1, \qquad \text{for $\mathbb{P}^u$-a.e.~$\omega \in \{0,1\}^{\mathbb{Z}^d}$},
\end{equation}
where $\widetilde{\inf}$ is over the same configuration as $\widetilde{\sup}$ in~\eqref{eq:Solidification-intro}, and $\capa^\omega(\cdot)$ denotes the random walk capacity on $\mathcal{S}_\infty$, see~\eqref{eq:capacity}. On the full lattice $\mathbb{Z}^d$, $d\geq 3$ or on $\mathbb{R}^d$, $d \geq 3$, uniform lower bounds of the type considered in~\eqref{eq:Capacity-lower-bound-intro} are typically easier to obtain in the context of a convex set $A$ (or its discrete blow-up~$A_N$), by considering certain a projection attached to $A$ resp.~$A_N$ (we refer to the discussion in~\cite[Remark 3.6 (3)]{nitzschner2017solidification} for more on this in the continuum set-up). However, even if $A$ is convex, such a projection is typically ill-defined for $A_N \cap \mathcal{S}_\infty \subseteq \mathcal{S}_\infty$, which reinforces the generality of the methods presented here.
\medskip

We now briefly comment on the idea of the proof and the challenges compared to~\cite{nitzschner2017solidification} and~\cite{CN2021disconnection}. The main difficulty in the proof of~\eqref{eq:Solidification-intro} is the construction of a certain \textit{resonance set} which is attached to $\mathcal{U}^\omega_{\ell_\ast,N}$, and is hard to avoid for a random walk starting in $A_N \cap \mathcal{S}_\infty$. In essence, the former set is characterized by the non-degeneracy of the local density of both $U_0$ and $U_1=\cS_\infty \setminus U_0$ relative to the infinite cluster, on many properly separated scales. More precisely, for a given (typical) percolation configuration $\omega \in \{0,1\}^{\mathbb{Z}^d}$, integers $I, J, L \geq 1$ (which fulfill a certain compatibility condition, see~\eqref{eq:compatible}) representing the total number, strength, and separation of inspected scales, respectively, and $U_0 \subseteq \mathcal{S}_\infty$, we define
\begin{equation}
  \Res = \Res^\omega (U_0,I,J,L,\ell_*) = \Big\{x\in \cS_\infty\, : \, \sum_{\ell\in \cA_*} \IND_{\{\widetilde{\sigma}^\omega_\ell (x) \in [\widetilde{\alpha},1-\widetilde{\alpha}]\}} \geq J\Big\},
\end{equation}
where $\widetilde{\alpha}$ depends only on the dimension $d$, $\widetilde{\sigma}^\omega_\ell(x)$ measures the volume of $U_1$ relative to 
the box $B(x,4\cdot 2^\ell) \cap \mathcal{S}_\infty$, and $\ell$ is chosen among a subset $\mathcal{A}_\ast$ of integers $L\mathbb{N}$ (see~\eqref{eq:A-ast_def} for the precise definition). 
A key obstacle to directly utilizing previous results such as~\cite{CN2021disconnection} is that
due to the inherent irregularity of the percolation cluster, even at sufficiently large scales, one can only control the behavior of local densities up to a certain degree of precision, quantified by an (error) parameter $\alpha \in (0,1)$. By decreasing the size of $\alpha$, thereby aiming at a more precise control on the relative volume at a certain scale, the size of the region where the density is uniformly well-behaved at that scale is in turn reduced (see~\eqref{eq:r_N}). As $N$ tends to infinity, we require $\alpha$ to depend on $N$ in such a way that this well-behaved region remains sufficiently large (see \eqref{eq:alpha_N}). Furthermore, to uniformly control the accumulated error as the ``resonance strength'' $J$ increases, the precision parameter $\alpha$ must also depend on $J$.
For that reason, a more quantitative approach in which $J = J_N$ slowly tends to infinity with $N$ is then necessary. The key insight is that for any $\omega \in \Omega_{\mathrm{reg}} \subseteq \{0,1\}^{\mathbb{Z}^d}$, a set of full $\mathbb{P}^u$-measure, and any $(J_N)_{N \geq 0}$ with $J_N \nearrow \infty$, one has
\begin{equation}
\lim_{N \rightarrow \infty} \sup_{\ell_*} \sup_{U_0\in \cU_{\ell_*,N}^\omega} \sup_{x\in A_N \cap \cS_\infty} P_x^\omega \left[H_{\Res^\omega(U_0,I(J_N),J_N,L(J_N),\ell_\ast)}=\infty\right] = 0,
\end{equation}
where the supremum over $\ell_\ast$ is taken subject to the compatibility condition~\eqref{eq:compatible}, which depends also on $\omega \in \Omega_{\mathrm{reg}}$. 
 \medskip

Let us emphasize that we have chosen to formulate the results above for (strongly) supercritical \textit{site} percolation models under regularity assumptions \refPone--\refPtwo, \refD, and \refSone--\refStwo, however our approach can be generalized straightforwardly to (strongly) supercritical \textit{bond} percolation models, see Remark~\ref{rem:Applications-generalizations} (1). The most important case is Bernoulli bond percolation on $\mathbb{Z}^d$, $d \geq 3$, to which our methods apply as well. Moreover, we highlight that our approach does \textit{not} rely on precise quantitative homogenization techniques, but only requires a well-behaved \textit{geometry} of the underlying graph in terms of heat kernel bounds and volume regularity. In particular, the bounds~\eqref{eq:Solidification-intro} and~\eqref{eq:Capacity-lower-bound-intro} remain valid if the edges of the subgraph of $(\mathbb{Z}^d,\mathbb{E}^d)$ induced by $\mathcal{S}_\infty$ are equipped with \textit{any} choice of uniformly elliptic weights, see Remark~\ref{rem:Applications-generalizations} (2) for more on this aspect. We also stress that the bounds derived in this work ought to be pertinent to obtain upper bounds on several atypical disconnection- or isolation-type events for analogues of the models considered in~\cite{chiarini2020GFF,chiarini2020entropic,
CN2021disconnection,nitzschner2018disconnection,
nitzschner2017solidification, sznitman2019macroscopic}, in the non-elliptic random environments considered here, and we return to questions of this type elsewhere~\cite{chiarini2025-in-prep}. \medskip

The organization of this article is as follows. In Section~\ref{ref:Notation-Assumptions}, we introduce further notation and state the formal definitions of the assumptions \refPone--\refPtwo, \refD, and \refSone--\refStwo posed for the underlying percolation models, as well as quenched Gaussian heat kernel upper and lower bounds from~\cite{alves2019} or~\cite{sapozhnikov2017long-range}, which will be instrumental for the development of the multi-scale techniques in subsequent sections. In Section~\ref{sec:large-scale} we quantify the (uniform) regularity properties of the percolation cluster. In particular, we single out a set $\Omega_{\mathrm{reg}}^{\alpha,\vartheta}$ (depending on parameters $\alpha,\vartheta \in (0,1)$ characterizing the degree of irregularity) of ``good percolation configurations'', cf.~\eqref{eq:Omega_reg}, of full $\mathbb{P}^u$-measure which will guarantee the required large-scale uniformity to implement the multi-scale argument in the following section. The main quenched asymptotic absorption estimates~\eqref{eq:Solidification-intro} are proved in Section~\ref{sec:solidification}, see Theorem~\ref{thm:solidification}. In the Appendix~\ref{sec:proof-of-QCV} we give a proof of a volume regularity property for percolation clusters with long-range correlations, slightly adapted from~\cite{sapozhnikov2017long-range}. \medskip

Finally, we state our convention concerning constants that we use throughout the remainder of the text. We will denote by $C$, $c$, $c'$, $\ldots$ positive constants with values that can change from place to place. Numbered constants $c_0$, $c_1$, $\ldots$ retain their value assigned to them at the place of their first occurrence. All constants may implicitly depend on the dimension. Dependence of constants on further parameters will be made explicit in the notation. In Sections~\ref{sec:large-scale} and ~\ref{sec:solidification}, we will fix the parameter $u$ governing the site percolation measure $\mathbb{P}^u$, and allow all constants to implicitly depend on $u$.

\section{Notation and model assumptions on percolation clusters}
\label{ref:Notation-Assumptions}

We begin this section by introducing further notation that will be used throughout the article. 
Following this, we present the main assumptions \refPone--\refPtwo, \refD, and \refSone--\refStwo, which characterize the probability measures $(\mathbb{P}^u)_{u \in (a,b)}$ for some $0 < a< b< \infty$ governing the relevant percolation models and the structure of their respective (almost surely unique) infinite connected component. 
Our set of assumptions corresponds to the generic framework of \cite{alves2019}, see also Remark~\ref{rmk:on-assumpt} concerning the generality of the latter. Importantly, a broad class of percolation models on $\bbZ^d$, $d\geq 3$, with strong, algebraically decaying correlations, satisfies these assumptions. We provide some pertinent examples in Remark~\ref{rmk:examples}. 
Finally, we conclude this section with a brief discussion of some properties of random walks on infinite percolation clusters.

\subsection{Notation}\label{subsec:notation}

We let $\bbN  = \{0,1,2,...\}$ stand for the set of natural numbers, and $\mathbb{Z}_+ = \{1,2,...\}$ for the set of positive integers. 
For real numbers $s, t$, we let $s \wedge t$ and $s \vee t$ stand for the minimum and maximum of $s$ and $t$, respectively, and we denote by $\lfloor s \rfloor$ the integer part of $s$ and by $\lceil s \rceil$ the next integer greater or equal to $s$, when $s$ is non-negative. For two sequences $(a_n)_{n\in \bbN}$ and $(b_n)_{n\in \bbN}$ of positive real numbers, we write $a_n=o(b_n)$ if $\lim_{n\to\infty}|a_n/b_n|=0$, $a_n\gg b_n$ if $\lim_{n\to\infty}|a_n/b_n|=\infty$, and $a_n=O(b_n)$ if $\limsup_{n\to\infty}|a_n/b_n| <\infty$. We write $a_n \nearrow a$ if $a_n \leq a_{n+1}$ for every $n \in \mathbb{N}$ and $a_n\to a$ as $n\to\infty$ (with $a \in \mathbb{R} \cup \{+\infty\}$), and similarly  $b_n \searrow b$ if $b_n \geq b_{n+1}$ for every $n \in \mathbb{N}$ and $b_n \to b$ as $n \to \infty$ (with $b \in \mathbb{R} \cup \{-\infty\}$). \smallskip

Unless otherwise stated, we assume $d\geq 3$ throughout the article. Occasionally we may also allow $d = 2$, especially when we provide examples satisfying the assumptions (see Examples \ref{it:bernoulli}, \ref{it:membrane}, \ref{it:FF}).
In fact, part of our approach remains applicable when $d=2$, but substantial modifications to the set-up would be necessary due to the recurrence of the underlying random walk; see also the discussion preceding  \cite[(2.1)]{nitzschner2017solidification}. \smallskip

We denote by $| \cdot |$, $| \cdot |_1$, and $| \cdot |_\infty$ the Euclidean, $\ell^1$-, and $\ell^\infty$-norms on $\bbR^d$, respectively. For $x\in \bbR^d$ and $r \geq 0$, we write $B_{\bbR^d}(x,r)$ for the closed $\ell^\infty$-ball in $\mathbb{R}^d$ with radius $r$, centred at $x$. When $A$ is a subset of $\bbR^d$, we write $\mathring{A}$ for its interior. For $x \in \bbZ^d$ and $r \geq 0$, we furthermore write $B(x,r) = \{y \in \bbZ^d \, : \, |x-y|_\infty \leq r\} \subseteq \bbZ^d$ for the (closed) $\ell^\infty$-ball of radius $r$ and center $x$. 
If $x,y \in \bbZ^d$ fulfill $|x - y| = 1$, we call them neighbors and write $x \sim y$. A function $\pi : \{ 0, ...,n \} \rightarrow \bbZ^d$ is called a nearest-neighbor path (of length $n \geq 1$) if $\pi_i \sim \pi_{i+1}$ for all $0 \leq i \leq n -1$. For $K \subseteq \bbZ^d$, we let $|K|$ stand for the cardinality of $K$, and we write $K \subset \subset L$ if $K\subseteq L$ and $|K| < \infty$. 
Moreover, if $K \subseteq L$ we write $\partial_L K = \{ y \in L \setminus K \, : \, y \sim x \text{ for some } x \in K \}$ for the external boundary of $K$ relative to $L$, and $\partial_{L,\text{in}}K = \{y \in K \, : \, y \sim x \text{ for some } x \in L \setminus K \}$ for the internal boundary of $K$ relative to $L$. For $A \subseteq \bbZ^d$, we denote the set of edges in $A$ by $E(A) = \{\{x,y\} \, : \, x,y \in A, x \sim y \}$ and use $\mathbb{E}^d = E(\bbZ^d)$  as a shorthand notation. 
For a set $D \subseteq \bbR^d$ we denote by $D_N = (ND) \cap \bbZ^d$ its discrete blow-up.

  \medskip

  We now introduce some notation pertaining to the site-percolation models we study below. Consider the set $\Omega = \{0,1\}^{\bbZ^d}$ equipped with the $\sigma$-algebra $\cF$ generated by the coordinate maps $\{\omega \mapsto \omega(x)\}_{x\in \bbZ^d}$. Let $(\Psi_x)_{x\in \bbZ^d}$ be the group of lattice shifts on $\Omega$, defined by 
  \begin{equation}
    \Psi_x: \Omega\to \Omega, \qquad (\Psi_x \omega) (y) = \omega(y+x) \text{ for all $\omega\in \Omega$ and $x,y\in \bbZ^d$}.
  \end{equation}
For a given $\omega \in \Omega$, sites $x \in \mathbb{Z}^d$ with $\omega(x)= 1$ are declared open, and we denote the induced subset of open sites of $\bbZ^d$ by 
  \begin{equation}
  \cS = \cS(\omega) = \left\{x\in \bbZ^d \,:\, \omega(x)=1\right\}\subseteq \bbZ^d.
  \end{equation}
  We view $(\cS,E(\cS))$ as a subgraph of $(\bbZ^d,\mathbb{E}^d)$, retaining all edges connecting two vertices of $\cS$ of $\ell^1$-distance one. For $r\in [0,\infty]$, we denote by $\cS_r$ the set of vertices of $\cS$ which are in connected components of $(\cS,E(\cS))$ of $\ell^1$-diameter greater or equal to $r$. In particular, $\cS_\infty$ is the set of vertices in infinite connected component(s) of $\cS$. 
  We also view $(\cS_r,E(\cS_r))$, $r\in [0,\infty]$, as a subgraph of $(\bbZ^d,E(\bbZ^d))$. For $x,y\in \cS_\infty$, we define the weights
  \begin{equation}
  \label{eq:Weights-given-mu}
      \mu_{\{x,y\}} = \begin{cases}
          1, & \{x,y\}\in E(\cS_\infty),\\
          0, & \text{otherwise},
      \end{cases}
      \qquad
      \mu_{x} = \sum_{y \in \mathcal{S}_\infty} \mu_{\{x,y\}},  \ x \in \cS_\infty,
  \end{equation}
  and extend them to measures on $E(\cS_\infty)$ and $\cS_\infty$, respectively.  \smallskip

For $\omega, \omega' \in \Omega$ we write $\omega \leq \omega'$ if $\omega(x) \leq \omega'(x)$ for all $x \in \bbZ^d$. An $\cF$-measurable function $f : \Omega \rightarrow \bbR$ is called increasing if $\omega \leq \omega'$ implies $f(\omega) \leq f(\omega')$, and decreasing if the function $-f$ is increasing. An event $G\in \cF$ is called increasing (resp., decreasing) if $\IND_G$ is increasing (resp., decreasing).

\subsection{Model assumptions and examples}
\label{subsec:ModelAssumptions}

In this subsection, we state the fundamental assumptions \refPone--\refPtwo, \refD, and \refSone--\refStwo on the underlying percolation models of interest. The first instance of assumptions of this type goes back to~\cite{drewitz2014chemical}.
Several geometric and probabilistic properties, such as a shape theorem, quenched invariance principle, and Gaussian bounds for heat kernels have been established for models fulfilling these assumptions in, e.g., \cite{drewitz2014chemical,procaccia2016,sapozhnikov2017long-range}, see also~\cite{andres2024first} for results on first-passage percolation (albeit under a slightly different assumption replacing~\refD concerning the decoupling). We refer to the introduction of~\cite{sapozhnikov2017long-range} for a summary of results. In fact, the statements in~\cite{drewitz2014chemical,procaccia2016,sapozhnikov2017long-range} were first established under a stronger decoupling assumption~\refPthree, but these results were proved in~\cite{alves2019} to remain valid when the latter assumption is relaxed to the slightly weaker version~\refD (see also Remark~\ref{rmk:on-assumpt}). These assumptions capture several essential properties of (correlated) percolation models under which the geometry of the infinite cluster resembles that of the integer lattice on large scales.  \smallskip

On $(\Omega, \cF)$ we consider a family of probability measures $\{\bbP^u\,:\,u\in (a,b)\}$ with $0<a<b<\infty$ satisfying:
\begin{itemize}
\item[\textbf{P1}] (\textit{Ergodicity})\phantomsection\label{P1}
For each $u\in(a,b)$ and $x\in \bbZ^d$, the family of lattice shifts $(\Psi_x)_{x \in \mathbb{Z}^d}$ on $\Omega$ is measure-preserving and ergodic with respect to $\bbP^u$.
\item[\textbf{P2}] (\textit{Monotonicity})\phantomsection\label{P2}
For $a<u<u'<b$, and any increasing event $G\in\cF$, 
$\bbP^u[G] \leq \bbP^{u'}[G]$.
\item[\textbf{D}] (\textit{Decoupling}) \phantomsection\label{D}
There exist $\beta,\gamma,\zeta>0$ such that for $i\in \{1,2\}$, $L,s\in \bbZ_+$, $x_1,x_2\in \bbZ^d$ with $|x_1-x_2|_\infty =sL$, and $a<u<u'<b$, the following hold:
  \begin{itemize}
    \item[(a)] if $A_i\in \sigma(\{\omega\mapsto\omega(y)\} \,:\, y\in B(x_i,L))$ are increasing events, then 
    \begin{equation}
  \label{eq:Decoup-A}
    \bbP^u [A_1\cap A_2] \leq \bbP^{u'}[A_1]\bbP^{u'}[A_2] + C \exp\left(-c\min \left\{(u'-u)^\beta s^\gamma, e^{(\log L)^\zeta}\right\}\right),
    \end{equation}
    \item[(b)] if $B_i\in \sigma(\{\omega\mapsto\omega(y)\} \,:\, y\in B(x_i,L))$ are decreasing events, then 
    \begin{equation}
    \label{eq:Decoup-B}
    \bbP^{u'} [B_1\cap B_2] \leq \bbP^u[B_1]\bbP^{u}[B_2] + C \exp\left(-c\min \left\{(u'-u)^\beta s^\gamma, e^{(\log L)^\zeta}\right\}\right).
    \end{equation}
  \end{itemize}
\item[\textbf{S1}] (\textit{Local uniqueness})\phantomsection\label{S1}
There exists a function $f_{\text{S}}:(a,b)\times\bbZ_+\to \bbR$ such that for any $u\in(a,b)$,
\begin{equation}\label{eq:funcS}
\begin{minipage}{0.8\linewidth}
there exists $\DeltaS = \DeltaS(u)>0$ and $R_{\text{S}} = R_{\text{S}}(u)<\infty$ 
such that $f_{\text{S}}(u,R) \geq (\log R)^{1+\DeltaS}$ for all $R\geq R_{\text{S}}$,
and for any $u\in(a,b)$ and $R\geq 1$, one has:
$$
\mathbb P^u\left[ \, 
\cS_R\cap B(0,R) \neq \varnothing \,
\right]
\geq 
1 - e^{-f_{\text{S}}(u,R)}, \text{ and} 
$$

$$
\bbP^u\left[
\begin{array}{c}
\text{for all $x,y\in\cS_{\scriptscriptstyle{R/10}}\cap B(0,R)$,}\\
\text{$x$ is connected to $y$ in $\cS\cap B(0,2R)$}
\end{array}
\right]
\geq 1 - e^{-f_{\text{S}}(u,R)}.
$$ 
\end{minipage}
\end{equation}
\item[\textbf{S2}] (\textit{Continuity})\phantomsection\label{S2}
The percolation density $\eta : (a,b) \rightarrow [0,1]$, defined by
\begin{equation}\label{eq:eta}
\eta(u) \stackrel{\mathrm{def}}{=} \bbP^u\left[0\in\cS_\infty\right],
\end{equation}
is strictly positive and continuous. 
\end{itemize}
It is a standard consequence of \refSone via a union bound argument that the set $\cS_\infty$ is $\bbP^u$-a.s.~non-empty and connected (see, e.g.,~\cite[(2.8)]{drewitz2014chemical} or~\cite[Remark 1.9(2)]{sapozhnikov2017long-range}), and we refer to it as the infinite cluster of the site percolation induced by $\omega$. We denote 
\begin{equation}
\label{eq:Omega-0}
  \Omega_0 \stackrel{\mathrm{def}}{=} \{\omega\in \Omega \,:\, \mbox{$\cS_\infty(\omega)$ is nonempty and connected}\},
\end{equation}
so $\mathbb{P}^u[\Omega_0] = 1$ for any $u \in (a,b)$.
\begin{remark}\label{rmk:on-assumpt}
  Before the assumption \refD was introduced in \cite{alves2019}, 
  several results had been obtained with \refD replaced by the following stronger decoupling condition:
  \begin{itemize}
  \item[\textbf{P3}] (\textit{Stronger decoupling})\phantomsection\label{P3}
    For $i\in\{1,2\}$, let $A_i\in\sigma(\{\omega\mapsto\omega(y)\} \, : \, {y\in B(x_i,10L)})$ be decreasing events, and let 
  $B_i\in\sigma(\{\omega\mapsto\omega(y)\} \,:\, {y\in B(x_i,10L)})$ be increasing events, where $L \in \mathbb{Z}_+$ and $x_1,x_2\in\bbZ^d$.
  There exist $R_{\text{P}},L_{\text{P}} <\infty$ and $\epsilon_{\text{P}},\chi_{\text{P}}>0$ such that for any integer $R\geq R_{\text{P}}$ and $a<\widehat u<u<b$ with 
  \[
  u\geq \left(1 + R^{-\chi_{\text{P}}}\right)\cdot \widehat u,
  \]
  and assuming $|x_1 - x_2|_\infty \geq R L$, one has
  \[
  \bbP^u\left[A_1\cap A_2\right] \leq 
  \bbP^{\widehat u}\left[A_1\right] \cdot
  \bbP^{\widehat u}\left[A_2\right] 
  + e^{-f_{\text{P}}(L)} ,\
  \]
  and
  \[
  \bbP^{\widehat u}\left[B_1\cap B_2\right] \leq 
  \bbP^u\left[B_1\right] \cdot
  \bbP^u\left[B_2\right] 
  + e^{-f_{\text{P}}(L)} ,\
  \]
  where $f_{\text{P}} : \mathbb{Z}_+ \rightarrow \mathbb{R}$ fulfills $f_{\text{P}}(L) \geq e^{(\log L)^{\epsilon_{\text{P}}}}$ for all $L\geq L_{\text{P}}$.
  \end{itemize}
In fact, \refPthree implies \refD by \cite[Remark 6.1(1)]{alves2019}, and the motivation to introduce \refD in \cite{alves2019} is due to the fact that the models considered in the latter satisfy \refD but fail to satisfy \refPthree. We refer to Remarks 6.1--6.3 in \cite{alves2019} for more discussions on the assumption.
\end{remark}

\begin{remark}[Examples]\label{rmk:examples}
    We first collect some examples that satisfy \refPone--\refPthree and \refSone--\refStwo, which can be found in \cite[Section III]{drewitz2014chemical}. 
    \begin{enumerate}[label=(\arabic*),leftmargin=*]
        \item \label{it:bernoulli}Bernoulli site percolation in $d\geq 2$. In this case, $\bbP^u$ is the product measure of a Bernoulli distribution with parameter $u$ standing for the probability of any given site to be open. Formally, one has $\bbP^u[\omega(x)=1]=1-\bbP^u[\omega(x)=0]=u$ with i.i.d.~$\{\omega \mapsto \omega(x)\}_{x \in \mathbb{Z}^d}$. The assumptions \refPone--\refPthree and \refSone--\refStwo hold for $u\in(a,b)$ with $p_c(d)<a<b<1$, where $p_c(d)$ is the critical parameter for site percolation. We refer to \cite{grimmett1999} for a thorough treatment of Bernoulli percolation.
        \item \label{it:RI}Random interlacements in $d\geq 3$. This model was introduced in \cite{sznitman2010}. In this case, $\bbP^u$ is characterized by the equation 
        \[
        \bbP^u[\cS\cap K =\varnothing] = e^{-u \capa(K)}, \quad \mbox{for all $\varnothing \neq K\subset\subset \bbZ^d$},
        \]
        where $\capa(K)$ denotes the capacity of $K$ (see \eqref{eq:capacity} and below).
        The assumptions \refPone--\refPthree and \refSone--\refStwo hold for $u\in(a,b)$ for any $0<a<b<\infty$, by~\cite{rath2011transience,sznitman2010, sznitman2012decoupling}. 
        \item \label{it:vacant-RI}Vacant set of random interlacements in $d\geq 3$. In this case, $\bbP^u=\mathrm{P}^{1/u}$ (the transformation $u\mapsto 1/u$ is implemented to align with the monotonicity in \refPtwo) and $\mathrm{P}^{v}$ is characterized by the equation
        \[
        \mathrm{P}^{v}[K\subseteq \cS] = e^{-v  \capa(K)}, \quad \mbox{for all $\varnothing \neq K\subset\subset \bbZ^d$}.
        \]
        Owing to the series of papers \cite{hugo2024,
        duminilcopin2023finiterangeinterlacementscouplings,
        duminilcopin2023phasetransitionvacantset,goswami2025stronglocaluniquenessvacant}, the equality of certain critical parameters governing the percolation phase transition has been established in~\cite[Theorem 1.1]{goswami2025stronglocaluniquenessvacant}. This shows in combination with~\cite{drewitz2014local,sznitman2010,
        sznitman2012decoupling,teixeira2009uniqueness}
         that \refSone holds for the entire supercritical phase, and the assumptions \refPone--\refPthree and \refSone--\refStwo are satisfied for $u\in (a,b)$ with $\frac{1}{u_{*}(d)}<a<b<\infty$, where $u_*(d)$ is the critical parameter for the percolation phase transition of the vacant set of random interlacements. 
        \item \label{it:GFF}Level-sets of the Gaussian Free Field in $d\geq 3$. In this case, $\bbP^u=\mathrm{P}^{h_*(d)-u}$, where $\mathrm{P}^h$ for $h\in \bbR$ is the law of the upper level-set $E^{\geq h}_{\mathrm{GFF}}$ under the measure governing a Gaussian Free Field and $h_*(d)$ is the critical threshold for the percolation phase transition (see~\cite{bricmont1987percolation,rodriguez2013phase}).  Using the equality of the critical parameters characterizing the percolation phase transition of $E^{\geq h}_{\mathrm{GFF}}$, see~\cite{hugo2023equality}, \refSone holds for the entire supercritical phase and assumptions \refPone--\refPthree and \refSone--\refStwo are satisfied for $u\in (a,b)$ with $0 < a <b< \infty$.
    \end{enumerate}
    We also collect some further models in the literature satisfying \refPone--\refPthree and \refSone--\refStwo.
    \begin{enumerate}[resume,label=(\arabic*),leftmargin=*]
        \item\label{it:membrane} Level-sets of the Gaussian membrane model in $d\geq 5$. Here, $\bbP^u=\mathrm{P}^{\overline{h}(d)-u}$, where $\mathrm{P}^h$ for $h\in \bbR$ is the law of the upper level-set $E^{\geq h}_{\mathrm{GMM}}$ under the measure governs a Gaussian membrane model (see \cite{CN23GMM} for its detailed description) and $\overline{h}(d)$ is a critical parameter such that $E^{\geq h}_{\mathrm{GMM}}$ is in a strongly percolative regime for $h < \overline{h}(d)$ (see~\cite[(5.4)]{CN23GMM} for its precise definition). It is known by the proof of Theorem 5.1 in \cite{CN23GMM} that $-\infty< \overline{h}(d)\leq h_*(d)$ (with $h_*(d)$ the percolation threshold for $E^{\geq h}_{\mathrm{GMM}}$) and the assumptions \refPone--\refPthree and \refSone--\refStwo hold for $u\in (a,b)$ with $0 <a<b< \infty$. 
        \item \label{it:FF} Level-sets of fractional Gaussian fields in $d \geq 2$.  As in the previous example, here one has $\mathbb{P}^u = \mathrm{P}^{\overline{h}(d) - u}$, where $\mathrm{P}^h$ governs for $h \in \mathbb{R}$ the law of the upper level-set $E^{\geq h}_{\mathrm{FF}}$ of a fractional Gaussian field in $d \geq 2$, with covariances given by the Green function of an isotropic, $\alpha$-stable random walk, with $\alpha \in (0, 2 \wedge d)$, see~\cite{bolthausen1995entropic,
        chiarini2016extremes} for a precise definition of these Gaussian fields. As before, $\overline{h}(d)$ is the critical parameter such that $E^{\geq h}_{\mathrm{FF}}$ is in a strongly percolative regime for $h < \overline{h}(d)$. As outlined in~\cite[Remark 3.4(2)]{CN23GMM}, the phase transition for $E^{\geq h}_{\mathrm{FF}}$ follows analogously to that of the membrane model, and \refPone--\refPthree and \refSone--\refStwo hold for $u\in (a,b)$ with $0 <a<b< \infty$.
        \item\label{it:GL} Level-sets of the Ginzburg--Landau $\nabla \phi$-model in $d\geq 3$. Here, $\bbP^u=\mathrm{P}^{\overline{h}(d)-u}$, where $\mathrm{P}^h$ for $h\in \bbR$ is the law of the level-set $E^{\geq h}_{\mathrm{GL}}$ under the measure governing a Ginzburg--Landau $\nabla \phi$-model (see \cite{rodriguez2016nabla} for its detailed description) and $\overline{h}(d)$ is the threshold for the strongly percolative regime (see (4.27) in \cite{rodriguez2016nabla} for its precise definition). It is known by Remark 4.7 and the proof of Theorems 4.8 and 4.9 in \cite{rodriguez2016nabla} that $-\infty< \overline{h}(d)\leq h_*(d)$ (with $h_*(d)$ again denoting the corresponding percolation threshold) and the assumptions \refPone--\refPthree and \refSone--\refStwo hold for $u\in (a,b)$ with $0<a<b<\infty$. 
    \end{enumerate}
    It is plausible but still open at the moment that $\overline{h}(d)=h_*(d)$ for each of the models in Examples \ref{it:membrane}--\ref{it:GL} (incidentally, we refer to~\cite{muirhead2024sharpness} for a result on the subcritical sharpness for a general class of Gaussian models). If the two critical parameters $h_\ast(d)$ and $\overline{h}(d)$ are proved to coincide for any specific model in Examples \ref{it:membrane}--\ref{it:GL}, our results will thus apply to the entire supercritical phase of the respective model. 
    
    The subsequent examples satisfy \refPone--\refPtwo, \refD, and \refSone--\refStwo as established in~\cite{alves2019}, and we refer to the latter for a more detailed descriptions of the models.
    \begin{enumerate}[resume,label=(\arabic*),leftmargin=*]
        \item\label{it:loop-soup} Range of the random walk loop soup in $d\geq 3$. In this case, $\bbP^u$ is the law of the trace on $\mathbb{Z}^d$ of a Poisson point process of loops $\mathscr{L}^u$ with intensity measure $u\mu$, and $\mu$ is a certain loop measure on the set of equivalence classes of based loops $\cL$. It follows from~\cite[Theorem 1.1]{alves2019} and~\cite{chang2017} that the assumptions \refPone--\refPtwo, \refD, and \refSone--\refStwo hold for $u\in (a,b)$ with $\alpha_c<a<b<\infty$, where $\alpha_c$ is the critical parameter for percolation phase transition.
        \item\label{it:vacant-loop-soup}Vacant set of the random walk loop soup in $d\geq 3$. In this case, $\bbP^u=\mathrm{P}^{\overline\alpha-u}$, where $\mathrm{P}^{\alpha}$ for $\alpha>0$ is the law of the vacant set $\mathscr{V}^\alpha$ of random walk loop soup and $\overline\alpha(d)$ is the critical parameter for the strong percolating regime, which can be defined similarly as in Examples \ref{it:membrane}--\ref{it:GL}; see also \cite[Theorem 1.3]{alves2019} for its positivity.
        By Remark 6.6 in \cite{alves2019}, the assumptions \refPone--\refPtwo, \refD, and \refSone--\refStwo are satisfied for $u\in (a,b)$ with $0<a<b<\overline\alpha$. It is plausible but open at the moment that $\overline\alpha(d)=\alpha_*(d)$.
    \end{enumerate}
\end{remark}

Although we have formulated our assumptions for site percolation models, our main result could similarly be stated for appropriate bond percolation models under suitable modifications of the assumptions above. In particular, our analysis extends directly to supercritical Bernoulli bond percolation on $\mathbb{Z}^d$, $d \geq 3$, see Remark~\ref{rem:Applications-generalizations}~(1) for further discussion.

\subsection{Random walks on the infinite percolation cluster} \label{subsec:random-walks}
 
Unless otherwise stated, we consider a fixed configuration $\omega\in \Omega_0$ throughout this entire subsection (recall~\eqref{eq:Omega-0}). We write $\Gamma(\mathbb{R}_{\geq 0},\mathbb{Z}^d)$ for the space of right-continuous, piecewise constant functions from $\mathbb{R}_{\geq 0} = [0,\infty)$ to $\mathbb{Z}^d$ that have finitely many jumps on any bounded interval. We let $X=(X_t)_{t\geq 0}$ be the canonical process on $\Gamma(\mathbb{R}_{\geq 0},\mathbb{Z}^d)$, which is distributed under $P_x^\omega$, $x \in \mathcal{S}_\infty$, as a continuous-time, constant-speed simple random walk on $\cS_\infty(\omega)$ starting from $x$, where $\cS_\infty(\omega)\subseteq \bbZ^d$ is the infinite percolation cluster introduced in Subsection \ref{subsec:notation}. Formally, $X$ can be viewed as a Markov process with generator 
\begin{equation}
\cL_{\cS_\infty} f(x)=\frac{1}{\mu_x}\sum_{y\in \cS_\infty}\mu_{\{x,y\}}(f(y)-f(x)), \quad x\in \cS_\infty,
\end{equation}
with $\mu_x, \mu_{\{x,y\}}$ as in~\eqref{eq:Weights-given-mu}, which waits an exponential time with mean one at each vertex and jumps to a uniformly chosen nearest neighbor. \smallskip

Let $(\theta_t)_{t\geq 0}$ stand for the family of canonical time-shifts on $\Gamma(\mathbb{R}_{\geq 0}, \mathbb{Z}^d)$, i.e., $\theta_t w(\cdot)=w(t+\cdot)$ for $w\in \Gamma(\bbR_{\geq 0}, \bbZ^d)$ and $t \geq 0$. Let $\zeta_1 = \inf\{t \geq 0 \, : \, X_t \neq X_0\}$ denote the time of the first jump of $X$ (with the convention $\inf \varnothing=\infty$). Given $U \subseteq \mathcal{S}_\infty$, we introduce stopping times (with respect to the canonical right-continuous filtration $(\cF_t)_{t \geq 0}$ generated by $(X_t)_{t \geq 0}$) $H_U = \inf\{ t \geq 0 \, : \, X_t \in U\}$, $\widetilde{H}_U = \inf\{ t > \zeta_1 \, : \, X_t \in U\}$, and $T_U = \inf \{ t \geq 0 \, : \, X_t \notin U \}$, which are the entrance, hitting, and exit times of $U$. We also introduce for $r\geq 0$ the stopping time 
\begin{equation}\label{eq:tau_r}
  \tau_r = \inf\{ t \geq 0 \, : \, |X_t-X_0|_\infty \geq r\}.
\end{equation}

We now introduce the heat kernel $q^\omega_t$ and the heat kernel $q^{\omega}_{t,U}$ killed upon leaving $U \subseteq \mathcal{S}_\infty$ of the  random walk $X$ as
\begin{align}
\label{eq:HeatKernelDef}
  q^\omega_t(x,y) & = \frac{P^\omega_x[X_t = y]}{\mu_y},\quad  t \geq 0,\, x,y \in \mathcal{S}_\infty, \text{ and} \\
  q^{\omega}_{t,U}(x,y) & = \frac{P^\omega_x[X_t = y, T_U > t]}{\mu_y},\quad  t \geq 0,\, x,y \in \mathcal{S}_\infty.\label{eq:KilledHeatKernel}
\end{align}
Both $q_t^\omega(\cdot,\cdot)$ and $q_{t,U}^\omega(\cdot,\cdot)$ are symmetric in the exchange of their arguments. \medskip 

It is proved in~\cite[Corollary 6.5]{alves2019} (see also~\cite[Theorem 1.15]{sapozhnikov2017long-range}) that if the assumptions \refPone--\refPtwo, \refD, and \refSone--\refStwo hold, then for any $\varepsilon>0$ there exists a family of random variables $\{\cT(x,\varepsilon,\cdot)\}_{x\in \bbZ^d}$, and constants $c_{\text{hk}i}=c_{\text{hk}i}(\varepsilon)$, $i=1,\ldots,6$, such that for $\bbP^u$-almost every $\omega$, $x\in \cS_\infty(\omega)$, $\cT(x,\varepsilon,\omega)$ is finite, and for $t\geq \cT(x,\varepsilon,\omega)\vee |x-y|_1^{1+\varepsilon}$ and $y\in \cS_\infty(\omega)$, 
\begin{equation}\label{eq:HK}
    c_{\text{hk}1} t^{-d/2} \exp\left(-c_{\text{hk}2} \frac{|x-y|_1^2}{t}\right) \leq q_t^\omega(x,y) \leq c_{\text{hk}3} t^{-d/2} \exp\left(-c_{\text{hk}4} \frac{|x-y|_1^2}{t}\right)
  \end{equation} 
holds. Moreover, for all $z\in \bbZ^d$ and $r\geq 1$,
\begin{equation}\label{eq:HK-time}
  \bbP^u [\cT(z,\varepsilon,\cdot) \geq r] \leq c_{\text{hk}5}e^{-c_{\text{hk}6}(\log r)^{1+\DeltaS}},
\end{equation}
where $\Delta_\text{S}$ is defined in \eqref{eq:funcS}. We will apply the result above with $\varepsilon=\frac{1}{2}$ and abbreviate $c_{\text{hk}i}(\frac{1}{2})$, $\cT(\cdot,\frac{1}{2},\omega)$ by $c_{\text{hk}i}$, $\cT(\cdot,\omega)$ for $i=1,\ldots,6$, respectively. \smallskip

We conclude this section by introducing some potential-theoretic notions for the random walk on $\mathcal{S}_\infty$. For the remainder of this section, we assume that $\omega$ is in a subset of $\Omega_0$ of full $\mathbb{P}^u$-measure under which~\eqref{eq:HK} holds, and that $d \geq 3$. We define for $x,y \in \mathcal{S}_\infty(\omega)$ the Green function of the random walk on $\mathcal{S}_\infty$ by
\begin{equation}
g^\omega(x,y) = \frac{1}{\mu_y}\left[\int_0^\infty \IND_{\{X_t = y \}} \mathrm{d}t \right] = \int_0^\infty q_t^\omega(x,y) \mathrm{d}t,
\end{equation} 
which is $\bbP^u$-a.s.~finite by~\eqref{eq:HK} and symmetric (in particular, $X$ is transient on $\mathcal{S}_\infty$). We note in passing that one has in fact quantitative bounds on the decay of $g^\omega$ for $\bbP^u$-a.e.~$\omega$, see~\cite[Theorem 1.17]{sapozhnikov2017long-range}, although we will not need such strong controls. For $\varnothing \neq A\subset\subset \mathcal{S}_\infty$, we define the equilibrium measure of $A$ on $\cS_\infty$ by
\begin{equation}\label{eq:equilibrium}
  e_A^\omega(x) = P_x^\omega [\widetilde{H}_A=\infty]\mu_x \IND_A(x), \quad x\in \mathcal{S}_\infty,
\end{equation}
and denote its (finite) total mass, the capacity of $A$, by
\begin{equation}\label{eq:capacity}
  \capa^\omega(A) = \sum_{x\in A} e_A^\omega(x)
\end{equation}
(when $\omega = \boldsymbol{1}$, the constant function assigning $1$ to every $x \in \mathbb{Z}^d$, we drop $\omega$ from the notation, and note that this coincides with the classical simple random walk capacity on $\bbZ^d$). We record for later use that one has for $\varnothing \neq A \subset \subset \mathcal{S}_\infty$ the last-exit decomposition
\begin{equation}
\label{eq:Last-exit}
P^\omega_x[H_A < \infty] = \sum_{y \in A } g^\omega(x,y) e_A^\omega(y), \qquad x \in \cS_\infty,
\end{equation}
see, e.g.,~\cite[Proposition 7.10]{barlow2017heatkernel}. 

\section{Uniform large-scale behavior over very large regions}\label{sec:large-scale}

We will demonstrate in this section that both the heat kernel and the relative volume density of the cluster exhibit controlled behavior over many scales, uniformly over sizeable boxes. This will be instrumental in Section \ref{sec:solidification} for the development of a ``one-step estimate'' (see Proposition \ref{prop:positive-probability-2}).
We begin by quantifying the mixing property of the relative volume density, in which a convergence rate which is faster than any polynomial is obtained. The proof builds on the framework developed in \cite{sapozhnikov2017long-range} and is detailed in Appendix~\ref{sec:proof-of-QCV}. 
The spatial uniformity we need will then follow readily from the quantitative control on the volume regularity and the previously stated heat kernel estimates \eqref{eq:HK}--\eqref{eq:HK-time}.
Finally, we conclude the section with a heuristic overview of the results and explain how they will be applied in the proof of the solidification estimates.

\medskip

We assume $d\geq 2$ throughout this section and work under the following assumption throughout the remainder of the present article, unless stated otherwise.
\begin{assumption}\phantomsection\label{A}
  The family of measures $\{\bbP^u\,:\,u\in (a,b)\}$ on $(\Omega, \cF)$ satisfies assumptions \textnormal{\refPone--\refPtwo}, \textnormal{\refD} and \textnormal{\refSone--\refStwo}.
\end{assumption}
For $x\in \bbZ^d$, we denote $\bbP^u[~\cdot \mid x\in \cS_\infty]$ by $\bbP^u_x[~\cdot~]$ (by~\refPone and~\refStwo, $\bbP^u[x\in \cS_\infty] = \eta(u) > 0$).
\begin{prop}[$d\geq 2$, \refA] \label{prop:Quantitative-control-volume}
  For $\alpha \in (0,1)$, there exist positive constants $\kappa_{\mathrm{d}i}(\alpha)$, $i=1,2$, such that for $R\geq 1$, 
  \begin{equation}\label{eq:QCV}
    \bbP_0^u \left[ \frac{|\cS_\infty \cap B(0,R)|}{|B(0,R)|} \in [(1-\alpha)\eta(u), (1+\alpha)\eta(u)]^c\right]\leq \kappa_{\mathrm{d}1}(\alpha) e^{-\kappa_{\mathrm{d}2}(\alpha) (\log R)^{1+\DeltaS}}
  \end{equation}
  (with $\DeltaS$ defined in~\eqref{eq:funcS}). Moreover, the function $\alpha \mapsto \kappa_{\mathrm{d}2}(\alpha)$ is increasing on $(0,1)$.
\end{prop}  

The proof of the proposition is a rather straightforward modification of the proof of Lemma 3.3 in \cite{sapozhnikov2017long-range}; the necessary modifications are given in Appendix \ref{sec:proof-of-QCV} for the convenience of the reader. \medskip

We now define for a given integer $R \geq 1$ the radius of a sup-norm ball on which the behavior of the relative volume density of the graph $\mathcal{S}_\infty$ and the heat kernel are uniformly controlled for all sub-balls at scale $R$ (when $R$ is large enough).
For each positive integer $R\geq 1$ and $\alpha \in (0,1)$, we introduce 
\begin{equation}\label{eq:r_N}
  r_{\alpha, R} = \exp(\kappa_{\text{reg}}(\alpha)(\log R)^{1+\DeltaS}),\quad \mbox{with }\kappa_{\text{reg}}(\alpha)=\frac{1}{2d}(\kappa_{\text{d}2}(\alpha)\wedge c_{\text{hk}6}).
\end{equation}
Here, the constants $\kappa_{\text{d}2}(\alpha)$ and $c_{\text{hk}6}$ are from \eqref{eq:QCV} and \eqref{eq:HK-time}, respectively, and in particular $\alpha \mapsto \kappa_{\mathrm{reg}}(\alpha)$ is increasing on $(0,1)$. 
The scale $r_{\alpha, R}$ will be used to characterize a (very) large sup-norm ball centered at the origin, over which certain uniform controls (see Propositions \ref{prop:uni-den} and \ref{prop:uni-HK} below) can be obtained. In the proof of our main result, Theorem \ref{thm:solidification}, we will actually require $\alpha$ to tend to zero slowly with $N$, i.e.~by considering a sequence $(\alpha_N)_{N \geq 0}$ with $\alpha_N \searrow 0$, which allows the inspection of many properly separated scales, see \eqref{eq:alpha_N}. \medskip

We turn to the main results of this section. The following proposition  states that for any small approximation error $\alpha \in (0,1)$, and $\mathbb{P}^u$-almost every realization of $\omega$, there exists a minimal scale $R_{\mathrm{den}}(\omega,\alpha)$ such that for all $R$ above the scale, the relative volume densities of the cluster $\mathcal{S}_\infty$ measured within boxes of radius $R$ are deviate from $\eta(u)$ by at most $\alpha$, uniformly over $B(0,r_{\alpha, R})$. 
\begin{proposition}[$d\geq 2$, \refA]\label{prop:uni-den}
  For $0<\alpha<1$ and $\omega\in \Omega$, consider 
  \begin{equation}\label{eq:minimal-scale-ball}
    \Rden(\omega, \alpha) \stackrel{\mathrm{def}}{=} \inf
    \left\{R_0 \geq 1\,:\,
    \begin{aligned}
      &\mbox{for all $R\geq R_0$ and $x\in \cS_\infty \cap B(0,r_{\alpha, R})$,}\\
      & (1-\alpha)\eta(u) \leq \frac{|\cS_\infty \cap B(x,R)|}{|B(x,R)|} \leq (1+\alpha)\eta(u)
    \end{aligned}
    \right\}
  \end{equation}
  (with $\inf \varnothing =\infty$). One has, for every $n \in \mathbb{Z}_+$, the upper bound  
  \begin{equation}
  \label{eq:R-den-upper-bound}
    \bbP^u [\Rden(\,\cdot\,, \alpha) > n] \leq C(\alpha)e^{-c(\alpha) (\log n)^{1+\DeltaS}}.
  \end{equation}
  In particular, $\Rden(\,\cdot\,,\alpha)$ is a $\bbP^u$-a.s.~finite random variable, and for all $R\geq \Rden(\omega, \alpha)$, one has
  \begin{equation}\label{eq:volume-concentration}
    (1-\alpha) \eta(u) \leq \frac{|\cS_\infty \cap B(x,R)|}{|B(x,R)|} \leq (1+\alpha) \eta(u),
  \end{equation}
  for all $x\in \cS_\infty\cap B(0,r_{\alpha, R})$.
\end{proposition}
\begin{proof}
  For $n\geq 1$, 
  \begin{align}
    \{\Rden > n\} &= \left\{
    \begin{aligned}
      &\mbox{there exists $R_0\geq n$ and $x\in \cS_\infty \cap B(0,r_{\alpha, R_0})$ such that}\\
      &\frac{|\cS_\infty \cap B(x,R_0)|}{|B(x,R_0)|} \in [(1-\alpha)\eta(u), (1+\alpha)\eta(u)]^c
    \end{aligned}\right\}  \label{eq:N_den-2}
    \\
    &= \bigcup_{R_0=n}^\infty \bigcup_{x\in B(0,r_{\alpha, R_0})} \left\{ \frac{|\cS_\infty \cap B(x,R_0)|}{|B(x,R_0)|} \in [(1-\alpha)\eta(u), (1+\alpha)\eta(u)]^c , x\in \cS_\infty\right\},\notag
  \end{align}
  and by applying a union bound and using the definition of $r_{\alpha, R}$ in \eqref{eq:r_N}, we have 
  \begin{align*}
    \bbP^u [\Rden > n] 
    & \leq \sum_{R_0=n}^\infty \sum_{x\in\cap B(0,r_{\alpha, R_0})} \bbP^u \left[ \frac{|\cS_\infty \cap B(x,R_0)|}{|B(x,R_0)|} \in [(1-\alpha)\eta(u), (1+\alpha)\eta(u)]^c,x\in \cS_\infty\right]\\
    & \stackrel{\eqref{eq:QCV}, \eqref{eq:r_N}}{\leq} C(\alpha)\eta(u) e^{-c(\alpha) (\log n)^{1+\DeltaS}}, 
  \end{align*}
  proving~\eqref{eq:R-den-upper-bound}.
  The claim~\eqref{eq:volume-concentration} follows directly from the definition of $R_{\mathrm{den}}(\omega,\alpha)$.
\end{proof}
For $\alpha\in (0,1)$, we denote 
\begin{equation}\label{eq:Omega_den}
  \Omega_{\mathrm{den}}^\alpha \stackrel{\mathrm{def}}{=} \{\omega\in \Omega \,:\, \Rden(\omega, \alpha)<\infty\} \stackrel{\eqref{eq:N_den-2}}{\in} \cF.
\end{equation}
It is clear from Proposition \ref{prop:uni-den} that $\Omega_{\mathrm{den}}^\alpha$ has full $\bbP^u$-measure, for every $\alpha \in (0,1)$.

\medskip

We now derive an analogous statement concerning uniform controls on the heat kernel of the random walk on $\mathcal{S}_\infty$.

\begin{lemma}[$d\geq 2$, \refA]\label{lem:HK}
  Let $0 < \alpha < 1$ and $\omega\in \Omega$. The random variable defined by
  \begin{equation}\label{eq:minimal-scale-HK}
    \Rhk(\omega, \alpha) \stackrel{\mathrm{def}}{=} \inf
    \left\{
    \begin{array}{l}
      R_0 \geq 1\,:\,\mbox{for all $R\geq R_0$, $q_t^\omega(x,y)$ satisfies \eqref{eq:HK}}\\
      \mbox{for all $x\in \cS_\infty \cap B(0,2r_{\alpha, R})$, $y\in \cS_\infty$, and $t\geq R \vee |x-y|^{3/2}_1$}
    \end{array}
    \right\}
  \end{equation}
  (with $\inf \varnothing =\infty$) satisfies that 
  \begin{equation}
    \bbP^u \left[\Rhk(\,\cdot\,, \alpha) > n \right] \leq C(\alpha) e^{-c(\alpha) (\log n)^{1+\DeltaS}}.
  \end{equation}
  In particular, $\Rhk(\,\cdot\,, \alpha)$ is a $\bbP^u$-a.s.~finite random variable.
\end{lemma}

\begin{proof} Recall the definition of the random variables $\cT(\cdot)$ above~\eqref{eq:HK}, and the convention following~\eqref{eq:HK-time}.
  For $n\geq 1$ 
  \begin{align}
      \{\Rhk&(\,\cdot\,, \alpha) > n\} \notag\\
      &=\left\{
    \begin{array}{l}
      \mbox{there exists $R_0 \geq n$ such that $q_{t_0}^\omega(x_0,y_0)$ fails to satisfy \eqref{eq:HK}}\\
    \mbox{for some $x_0\in \cS_\infty \cap B(0,2r_{\alpha, R_0})$, $y_0\in \cS_\infty$ and $t_0 \geq R_0\vee |x_0-y_0|_1^{3/2}$}
    \end{array}
    \right\} \label{eq:N_hk-2}
    \\
    & \subseteq \{ \mbox{there exists $R_0 \geq n$ such that $R_0 \leq \max_{x\in \cS_\infty \cap B(0,2r_{\alpha, R_0})} \cT(x)$}\} \notag \\
    & = \bigcup_{R_0=n}^\infty \bigcup_{x\in B(0,2r_{\alpha, R_0})} \{\cT(x) \geq R_0, x\in \cS_\infty\}. \notag
  \end{align}
  Then we can apply a union bound to obtain
  \begin{equation}
  \begin{split}
  \bbP^u[\Rhk(\,\cdot\,, \alpha) > n] & \leq \sum_{R_0=n}^\infty \sum_{x\in B(0,2r_{\alpha,R_0})} \bbP^u [\cT(x) \geq R_0,x\in \cS_\infty] \\
  & \stackrel{\eqref{eq:HK-time}, \eqref{eq:r_N}}{\leq} C(\alpha)e^{-c(\alpha)(\log n)^{1+\DeltaS}},
  \end{split}
  \end{equation}
  which concludes the proof. 
\end{proof}

The first part of following proposition follows from the previous lemma by the definition of $\Rhk$ in \eqref{eq:minimal-scale-HK}. The second part is a consequence of the first part via a standard argument; see \cite[Proposition~5.26]{barlow2017heatkernel} or \cite[Lemma~5.8]{barlow2004RWpercolation}.

\begin{proposition}[$d\geq 2$, \refA]\label{prop:uni-HK} Let $\alpha\in (0,1)$.
  \begin{enumerate}[label=(\roman*)]
    \item For $\bbP^u$-almost every $\omega$, $\Rhk(\omega,\alpha)$ defined in \eqref{eq:minimal-scale-HK} is finite, and for all $R\geq \Rhk(\omega, \alpha)$, the heat kernel satisfies ($c_{\text{hk}i}$ for $i=1,\ldots,4$ defined at the end of Section \ref{subsec:random-walks})
  \begin{equation}\label{eq:HK2}
    c_{\text{hk}1} t^{-d/2} \exp\left(-c_{\text{hk}2} \frac{|x-y|_1^2}{t}\right) \leq q_t^\omega(x,y) \leq c_{\text{hk}3} t^{-d/2} \exp\left(-c_{\text{hk}4} \frac{|x-y|_1^2}{t}\right),
  \end{equation} 
  for any $x\in \cS_\infty \cap B(0,2r_{\alpha, R})$, $y\in \cS_\infty$, and $t\geq R \vee |x-y|^{3/2}_1$.
  \item {For any $\vartheta\in (0,1)$, there exist constants $c_{\text{khk}1}(\vartheta)>0$, $0<c_{\text{khk}2}(\vartheta)<1$ and a random variable $\Rkhk(\cdot,\alpha, \vartheta)$ such that for $\bbP^u$-almost  every $\omega \in \Omega$, $ \Rkhk(\omega, \alpha, \vartheta) < \infty$}, and for all integers $R\geq \Rkhk(\omega, \alpha, \vartheta)$, the heat kernel $q_{t,B(x_0,R) \, \cap \, \mathcal{S}_\infty}^\omega(\cdot,\cdot)$  of the random walk killed upon exiting $B(x_0,R) \cap \mathcal{S}_\infty$ (recall~\eqref{eq:KilledHeatKernel}) satisfies
  \begin{equation}\label{eq:killed-heat-kernel}
  q_{t,B(x_0,R)\, \cap \,\mathcal{S}_\infty}^\omega(x,y) \geq c_{\text{khk}1}(\vartheta) t^{-d/2},
  \end{equation}
  where $x_0\in \cS_\infty\cap B(0,r_{\alpha, R})$, $x,y\in \cS_\infty\cap B(x_0,(1-\vartheta)R)$ and $c_{\text{khk}2}(\vartheta)R^2 \leq t \leq R^2$.
  \end{enumerate}
\end{proposition}

For $\alpha,\vartheta\in (0,1)$, let
\begin{equation}
  \Omega_{\text{hk}}^{\alpha, \vartheta} \stackrel{\mathrm{def}}{=} \{\omega\in \Omega\,:\, \Rhk(\omega, \alpha)<\infty, \Rkhk(\omega, \alpha, \vartheta)<\infty\} \in \cF,
\end{equation}
and, with $\Omega_0$ and $\Omega_{\mathrm{den}}^\alpha$ as in \eqref{eq:Omega-0} and \eqref{eq:Omega_den} respectively, 
\begin{equation}\label{eq:Omega_reg}
  \Omega_{\mathrm{reg}}^{\alpha, \vartheta} \stackrel{\mathrm{def}}{=} \Omega_0 \cap \Omega_{\mathrm{den}}^\alpha \cap \Omega_{\text{hk}}^{\alpha, \vartheta}.
\end{equation}
It is clear that $\Omega_{\mathrm{reg}}^{\alpha, \vartheta}\in \cF$ and $\bbP^u\bigl[\Omega_{\mathrm{reg}}^{\alpha, \vartheta}\bigr]=1$.

\medskip

Propositions~\ref{prop:uni-den} and \ref{prop:uni-HK} show that for $\omega\in \Omega_{\mathrm{reg}}^{\alpha, \vartheta}$ the relative volume density of $\mathcal{S}_\infty$ and heat kernels exhibit well-controlled behavior at any scale $R\geq\widetilde{R}=\Rden\vee \Rhk\vee \Rkhk$, uniformly over the ball $B(0,r_{\alpha, R})$. This is similar in spirit to the notion of ``good'' and ``very good'' boxes introduced in~\cite[Definition~1.7]{barlow2004RWpercolation} to control degeneracies of the supercritical cluster of Bernoulli (bond) percolation, and developed in~\cite{sapozhnikov2017long-range} in the set-up of assumptions \refPone--\refPthree and \refSone--\refStwo (see also \cite[Definition~4.1]{sapozhnikov2017long-range} for the notion of ``regular'' and ``very regular'' boxes). For the application we have in mind, it is expedient to directly use the controls on the heat kernel of~\cite{sapozhnikov2017long-range} and to require quite precise volume concentration. \medskip

We will typically use $\Omega_{\mathrm{reg}}^{\alpha,\vartheta}$ to argue as follows. When the ``lowest accessible scale'' satisfies $R_{\min}\geq \widetilde{R}$, we obtain uniform estimates over the ball $B(0,r_{\alpha,R_{\min}})$. We choose $\alpha$ to depend on $N$ so that $r_{\alpha_N, R_{\min}}$ is comparable to or larger than the macroscopic scale.
In other words, we can examine the macroscopic sup-norm ball $B(0,r_{\alpha_N,R_{\min}})$ at any location down to the microscopic scale $R_{\min}$, while retaining good control over the underlying graph structure (see Fig.~\ref{fig:bad-box} for an 
illustration). Remarkably, the superpolynomial separation between the microscopic and macroscopic scales is sufficient to implement the multi-scale analysis in the proof of the solidification estimates.

\begin{figure}[htbp]
\begin{center}
\includegraphics[scale=0.8]{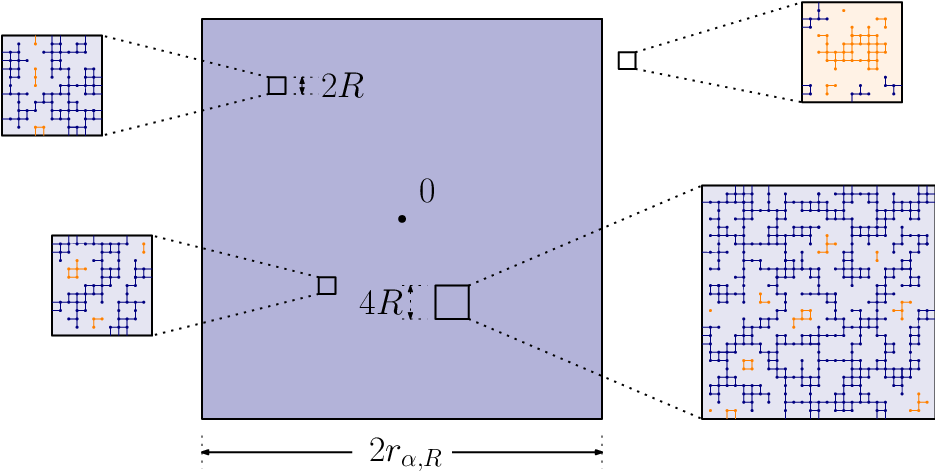}
\end{center}
\caption{Schematic illustration of the role of $r_{\alpha,R}$, for $R \geq R_{\min}$. All boxes with center inside $B(0,r_{\alpha,R}) \cap \mathcal{S}_\infty$ and radius greater or equal to $R$ are well-behaved, but degeneracies (depicted here for the relative volume) occur outside of $B(0,r_{\alpha,R})$.}
\label{fig:bad-box}
\end{figure}

\section{Solidification of porous interfaces on percolation clusters}\label{sec:solidification}

In this section, we state and prove in Theorem~\ref{thm:solidification} the solidification estimates of porous interfaces on typical realizations of supercritical site percolation clusters under our standing~\refA. We begin by stating our main result in a precise way, after providing some additional notation, and then present the proof. The latter consists of three steps, which are similar to \cite{nitzschner2017solidification} in the context of Brownian motion and to \cite{CN2021disconnection} in the context of random walks on $(\bbZ^d,\mathbb{E}^d)$ equipped with uniformly elliptic edge weights. In the first step, we study how a random walk on the supercritical cluster enters sets of well-balanced local density, relative to a given set $U_0 \subseteq \mathcal{S}_\infty$ (more precisely, the density will be defined in terms of the complement $U_1 = \mathcal{S}_\infty \setminus U_0$). Secondly, we introduce a certain ``resonance set'' which is naturally attached to $U_0$, and show that it is atypical for a random walk to avoid this set, in an asymptotic sense. The precise formulation appears in Theorem~\ref{thm:asy-Phi}. Lastly, we essentially use the resonance set as a substitute for a Wiener-type criterion to show that, again in an asymptotic sense, the porous interfaces cannot be avoided by a random walk. \medskip

Compared with the uniformly elliptic case, the presence of local degeneracies introduces substantial difficulties in implementing this strategy. 
For instance, in the uniformly elliptic setting the local density of $U_1$ (crucial to the estimate of hitting probabilities of the resonance set) is deterministic and enjoys good regularity properties similar to the continuous case, while in our case, it is intrinsically dependent on the percolation configuration $\omega \in \{0,1\}^{\bbZ^d}$. 
This is addressed by only considering scales above a certain threshold, utilizing the large-scale regularity properties of the infinite cluster $\cS_\infty$ from Section~\ref{sec:large-scale}.
Our analysis in this section shows that, for almost every $\omega$, there exists a sufficient ``number of scales'' to that can be used to show a decay of the probability to avoid the resonance set. Here, the smallest scale corresponds to the choice of $R_{\min}$ as in the discussion below~\eqref{eq:Omega_reg}, which must be sufficiently ``large'' to guarantee the aforementioned good regularity properties, and the maximal, ``very large'' scale exceeds $r_{\alpha,R_{\min}}$. 

\medskip

Throughout this section, we assume $d\geq 3$, and require the family of probability measures $\{\bbP^u\,:\,u\in (a,b)\}$ on $(\Omega, \cF)$ to satisfy the \refA. We also fix
\begin{equation}\label{eq:ass-bounded}
  \mbox{a compact set $A\subseteq \bbR^d$ with $\mathring{A}\neq\varnothing$.}
\end{equation}
We recall~\eqref{eq:Omega-0} and assume, without loss of generality, throughout this section that $\omega \in \Omega_0$. Consider a bounded, non-empty set $U_0\subseteq \cS_\infty$, its complement $U_1 = \cS_\infty \setminus U_0$ in $\cS_\infty$, and also the boundary $S=\partial_{\cS_\infty} U_0=\partial_{\cS_\infty,\text{in}}U_1$ relative to $\cS_\infty$.
To measure the local density of $U_0$ in balls of dyadic scale, we define for $\omega\in \Omega_0$, a non-negative integer $\ell$, and $x\in \bbZ^d$, the \emph{local density measure} by 
\begin{equation}\label{eq:local-density-measure}
\mu_{x,\ell}^\omega(y) \stackrel{\mathrm{def}}{=} 
\frac{\IND_{B(x,2^\ell)\cap \cS_\infty}(y)}{|B(x,2^\ell)\cap \cS_\infty|}, \qquad \mbox{for $y\in \bbZ^d$}
\end{equation}
(with the convention that $0/0=0$, and we will also use this convention in \eqref{eq:sigma_l^omega}--\eqref{eq:average-ball}). Note that if $x\in \cS_\infty$, $B(x,R)\cap \cS_\infty\neq \varnothing$ for $R\geq 1$.
For $\omega\in \Omega_0$, two \emph{local density functions associated with $U_1$} are defined as (with $U_1$ implicit in the notation)
\begin{align}
\sigma_\ell^{\omega}(x) &\stackrel{\mathrm{def}}{=} \mu_{x,\ell}^\omega(U_1) = \frac{|B(x,2^\ell)\cap U_1|}{|B(x,2^\ell)\cap \cS_\infty|},\quad \mbox{for $x\in \bbZ^d$},\label{eq:sigma_l^omega}\\
\widetilde{\sigma}_\ell^{\omega}(x) &\stackrel{\mathrm{def}}{=} \mu_{x,\ell+2}^\omega(U_1) = \frac{| B(x,4\cdot 2^\ell)\cap U_1|}{|B(x,4\cdot 2^\ell)\cap \cS_\infty|},\quad \mbox{for $x\in \bbZ^d$}\label{eq:sigma_l^omega-2}.
\end{align}
For $\omega\in \Omega_0$ and a given function $f:\bbZ^d \to \bbR$ we write $(f)_{x,\ell}^\omega$ for the average of $f$ on $B(x,2^\ell)\cap \cS_\infty$, namely 
\begin{equation}\label{eq:average-ball}
  (f)_{x,\ell}^\omega \stackrel{\mathrm{def}}{=} \frac{1}{|B(x,2^\ell)\cap \cS_\infty|}\sum_{y\in B(x,2^\ell)\cap \cS_\infty} f(y).
\end{equation}
Given $\omega\in \Omega_0$, and integers $N\geq 1$, $\ell_*\geq 0$, we define (recall~\eqref{eq:a_N-bound} and~\eqref{eq:ass-bounded})
\begin{equation}\label{eq:cU}
  \cU_{\ell_*,N}^\omega \stackrel{\mathrm{def}}{=} \{U_0\subseteq B(0,b_N) \cap \cS_\infty \, : \, \sigma_{\ell}^\omega (x) \leq \tfrac{1}{2}, \mbox{for all $x\in A_N\cap \cS_\infty$, $\ell\leq \ell_*$}\}.
\end{equation}
We furthermore define for fixed $U_0\in\cU_{\ell_*,N}^\omega$, $\epsilon\in \bbN$, and $\chi\in (0,1)$, the class of \emph{porous interfaces associated with $U_0$}
\begin{equation}
\cS_{U_0,\epsilon,\chi}^\omega \stackrel{\mathrm{def}}{=} \{\Sigma\subseteq \cS_\infty \text{ bounded}\, : \, P_x^\omega[H_\Sigma < \tau_\epsilon]\geq \chi, \text{ for all $x\in S$}\}
\end{equation}
(recall the notation $H_\Sigma$ for the entrance time of the random walk into $\Sigma$, and $\tau_\epsilon$ for the first time the random walk moves at sup-distance $\epsilon$ from its starting point, see~\eqref{eq:tau_r}). The parameter $\epsilon$ controls at which distance to $S$ the porous interface is located, while $\chi$ is a measure for how markedly it is felt for a random walk starting from any point in $S$ (see also Figure~\ref{fig:interface} for an illustration of the sets under consideration). \medskip

We now state the uniform controls for the probability that a random walk on $\mathcal{S}_\infty$ is absorbed by porous interfaces. This corresponds to the extension of the \emph{solidification estimates} from~\cite{nitzschner2017solidification} and~\cite{CN2021disconnection} to our case. Recall the growth condition~\eqref{eq:a_N-bound} governing the two sequences $(a_N)_{N \geq 0}$ and $(b_N)_{N \geq 0}$.
\begin{theorem}[$d\geq 3$, \refA, \eqref{eq:ass-bounded}, \eqref{eq:a_N-bound}]\label{thm:solidification}
  Let $\chi\in (0,1)$, and let $(c_N)_{N\geq 0}$ be another sequence of positive real numbers decreasing to $0$.  For $\bbP^u$-almost every $\omega$, it holds that 
  \begin{equation}\label{eq:solid-1}
    \lim_{N\to\infty} \sup_{\epsilon/2^{\ell_*}\leq c_N,a_N\leq 2^{\ell_*}\leq b_N} \sup_{U_0\in \cU_{\ell_*,N}^\omega} \sup_{\Sigma\in \cS_{U_0,\epsilon,\chi}^\omega} \sup_{x\in A_N \cap \cS_\infty} P_x^\omega [H_\Sigma=\infty] =0,
  \end{equation}
  and moreover,
  \begin{equation}\label{eq:solid-2}
    \lim_{N\to\infty}  \sup_{\epsilon/2^{\ell_*}\leq c_N,a_N\leq 2^{\ell_*}\leq b_N} \sup_{U_0\in \cU_{\ell_*,N}^\omega} \sup_{\Sigma\in \cS_{U_0,\epsilon,\chi}^\omega} \sup_{x\in \cS_\infty} (P_x^\omega [H_\Sigma=\infty] - P_x^\omega [H_{A_N}=\infty]) = 0. 
  \end{equation}
\end{theorem}

We immediately obtain the following corollary concerning the capacity:
\begin{corollary}[$d\geq 3$, \refA, \eqref{eq:ass-bounded}, \eqref{eq:a_N-bound}]
\label{cor:Capacity-Lower-Bound}
Under the same assumptions as Theorem~\ref{thm:solidification}, we have the following: For $\mathbb{P}^u$-almost every $\omega$, it holds that
\begin{equation}
\label{eq:Capacity-lower-bound-statement}
\liminf_{N \rightarrow \infty}\inf_{\epsilon/2^{\ell_*}\leq c_N,a_N\leq 2^{\ell_*}\leq b_N} \inf_{U_0\in \cU_{\ell_*,N}^\omega} \inf_{\Sigma\in \cS_{U_0,\epsilon,\chi}^\omega} \frac{\capa^\omega(\Sigma)}{\capa^\omega(A_N \cap \cS_\infty)} \geq 1.
\end{equation}
\end{corollary}
\begin{proof}
The proof of Corollary~\ref{cor:Capacity-Lower-Bound} assuming Theorem~\ref{thm:solidification} follows similarly as that of~\cite[Corollary 4.2]{CN2021disconnection}, and we only sketch it for completeness. Indeed, consider $\omega$ in a set of full $\mathbb{P}^u$-measure such that~\eqref{eq:solid-1} holds. Then, for $N \geq N_0(\omega)$ and $\epsilon/2^{\ell_\ast} \leq c_N$, $a_N \leq 2^{\ell_\ast} \leq b_N$, $U_0 \in \cU_{\ell_\ast, N}^\omega$, and $\Sigma \in \cS^\omega_{U_0, \epsilon, \chi}$, we have
\begin{equation}
\begin{split}
\capa^\omega(\Sigma) & \stackrel{\eqref{eq:capacity}}{\geq } \sum_{z \in \mathcal{S}_\infty} P_{z}^\omega[H_{A_N \cap \cS_\infty} < \infty] e_{\Sigma}^\omega(z) \\ & \stackrel{\eqref{eq:Last-exit}}{\geq } \inf_{x \in A_N \cap \cS_\infty} P^\omega_x[H_\Sigma < \infty]\capa^\omega(A_N \cap \cS_\infty),
\end{split}
\end{equation}
having also used the symmetry of $g^\omega(\cdot,\cdot)$.
\end{proof}

We briefly explain the restriction inherent in Theorem~\ref{thm:solidification} compared with the uniformly elliptic counterpart~\cite[Theorem 4.1]{CN2021disconnection}. In our set-up, we introduce the sequences $(a_N)_{N\geq 0}$ and $(b_N)_{N\geq 0}$ as in~\eqref{eq:a_N-bound} to control the growth rate of $2^{\ell_*}$ in the present theorem. The additional restriction for $(\epsilon,\ell_*)$ over the first supremum requires $2^{\ell_*}$ to fulfill $
a_N \leq 2^{\ell_*} \leq b_N$. One should view $a_N$ heuristically as determining a lower bound on the ``smallest spatial scale'' inspected, and $b_N$ as the corresponding ``largest spatial scale''. Such a restriction is natural, since one cannot expect good regularity properties of the cluster to hold uniformly in small boxes (essentially of size $a_N$) outside a certain box of size $b_N$. Heuristically, the growth condition~\eqref{eq:a_N-bound} is chosen to ensure that such regions are not inspected.

\subsection{Local density functions}\label{subsec:local-density}

In this subsection, we investigate how a random walk enters regions in which certain local densities are well-balanced on multiple scales. Recall from \eqref{eq:sigma_l^omega}--\eqref{eq:sigma_l^omega-2} that the definition of the local density functions depends on a set $U_1\subseteq \cS_\infty$ (with $U_0=\cS_\infty\setminus U_1$), a non-negative integer $\ell$, and a realization $\omega$ of the environment. We assume throughout this subsection that
\begin{equation}\label{eq:assum-4-1}
  \omega \in \Omega_0\quad \text{and}\quad \mbox{$U_0\subseteq \cS_\infty$}
\end{equation}
(recall $\Omega_0$ in \eqref{eq:Omega-0}) and will impose further conditions on $\ell$ and the set of $\omega$ in the statements. \medskip

We begin with two properties of the local density functions, which are analogous to the continuum set-up \cite[Lemma 1.1]{nitzschner2017solidification} and the discrete set-up with uniformly elliptic conductances \cite[Lemma 4.3]{CN2021disconnection}. Recall that $\eta(\cdot)$ is assumed to be strictly positive in \refStwo, $\alpha$ is introduced in Proposition \ref{prop:Quantitative-control-volume} to represent the deviation between the relative volume of the cluster in a box and its expectation, $\Rden(\omega,\alpha)$ is defined in \eqref{eq:minimal-scale-ball} as the minimal scale for spatial volume regularity, $\Omega_{\mathrm{den}}^\alpha$ is defined in \eqref{eq:Omega_den}, and the sequence $(r_{\alpha, N})_{N\geq 1}$ is defined in \eqref{eq:r_N} and refers to the radius of a sup-norm ball in which one has uniformly good regularity properties (see Propositions \ref{prop:uni-den} and \ref{prop:uni-HK}). 

\begin{lemma}[$d\geq 3$, \refA, \eqref{eq:assum-4-1}]\label{lem:sigma-regular} Let $\alpha\in (0,\frac{1}{3})$, $\omega\in\Omega_{\mathrm{den}}^\alpha$. 
Assume $2^{\ell} > 2^{\ell'} \geq \Rden(\omega,\alpha)$ and $r_{\alpha,2^{\ell'}}-2^{\ell}-2^{\ell'}>0$.
\begin{enumerate}[label=(\roman*)]
\item For any $x,x+y\in \cS_\infty \cap B(0,r_{\alpha,2^{\ell}})$ which are connected by a path in $\cS_\infty \cap B(0,r_{\alpha, 2^{\ell}})$, the local density functions are Lipschitz-continuous in the sense that
\begin{equation}\label{eq:Lip-cts}
|\sigma_\ell^\omega(x) - \sigma_\ell^\omega(x+y)| \leq c_{\mathrm{lip}} 
\cdot \widehat{d}_\omega(x,x+y),\quad c_{\rm{lip}} \stackrel{\mathrm{def}}{=} \frac{6\cdot 2^{-\ell}}{\eta(u)},
\end{equation}
where $\widehat{d}_\omega$ denotes the graph distance on $\cS_\infty \cap B(0,r_{\alpha, 2^{\ell}})$.

\item 
For $x\in \cS_\infty \cap B(0,r_{\alpha, 2^{\ell'}}-2^{\ell}-2^{\ell'})$,
we have 
\begin{equation}\label{eq:c0'cont}
  (\sigma_{\ell'}^\omega)_{x,\ell}^\omega \in 
  \left[
  \frac{1-\alpha}{1+\alpha} \sigma_\ell^\omega(x) - c_0 2^{\ell'-\ell},
  \frac{1+\alpha}{1-\alpha} \sigma_\ell^\omega(x) + c_0 2^{\ell'-\ell}
  \right],\quad c_0 \stackrel{\mathrm{def}}{=} \frac{3d\cdot 2^{d-1}}{\eta(u)}
\end{equation}
(see \eqref{eq:average-ball} for notation).
\end{enumerate}
\end{lemma}

Before we state the proof of the lemma above, we briefly comment on the choices of the scales $\ell, \ell'$, as well as the of coefficients $c_{\mathrm{lip}}$ and $c_0$. 

\begin{remark}
\label{rem:Further-remark-choices}
  \begin{enumerate}[label=(\arabic*),leftmargin=*]
    \item We have $R_{\mathrm{den}}(\omega, \alpha)<\infty$ since $\omega \in \Omega_{\mathrm{den}}^\alpha$. There exists $\ell>\ell'$ such that $r_{\alpha, 2^{\ell'}}-2^{\ell}-2^{\ell'}>0$, because $r_{\alpha, n}= \exp(\kappa_{\text{reg}}(\alpha)(\log n)^{1+\DeltaS})\gg n$ as $n\to\infty$, so that we can first choose a large enough $\ell'$ leaving enough room for the choice of $\ell$. 
    \item We highlight that for a given $\alpha\in (0,\frac{1}{3})$, $c_{\mathrm{lip}}$ and $c_0$ are not optimal and could be chosen slightly smaller (depending on $\alpha$), but since this plays no role for our analysis, we work with the choices in~\eqref{eq:Lip-cts} and~\eqref{eq:c0'cont}.
  \end{enumerate}
\end{remark}

\begin{proof}[Proof of Lemma $\ref{lem:sigma-regular}$]
Let $\ell,\ell'$ satisfy the conditions in the statement of the lemma. We begin with the proof of (i). Let $x,x+y\in \cS_\infty\cap B(0,r_{\alpha, 2^\ell})$ and assume that there is a nearest-neighbor path $\pi$ in $\cS_\infty\cap B(0,r_{\alpha, 2^\ell})$ of length $n= \widehat{d}_\omega(x,x+y)$ connecting the two points (otherwise, $\widehat{d}_\omega(x,x+y) = \infty$ and the claim is immediate). By the definition of $\sigma_\ell^\omega(\cdot)$ in \eqref{eq:sigma_l^omega}, for $0\leq i \leq n-1$,
\begin{equation}\label{eq:sigma-1-1}
\bigr|\sigma_\ell^\omega(\pi_{i}) - \sigma_\ell^\omega(\pi_{i+1})\bigl| = \left|\frac{| B(\pi_i,2^\ell)\cap U_1 |}{|B(\pi_i,2^\ell)\cap \cS_\infty|} - \frac{|B(\pi_{i+1},2^\ell)\cap U_1|}{|B(\pi_{i+1},2^\ell)\cap \cS_\infty|}\right|.
\end{equation}
Note that we have the elementary inequality for $a_1,a_2,b_1,b_2\in \bbR_+$, 
\begin{equation}\label{eq:elementary}
\left|\frac{b_1}{a_1} - \frac{b_2}{a_2}\right| 
\leq \frac{|b_1a_2-b_1a_1|+|b_1a_1-a_1b_2|}{a_1a_2} 
= \frac{b_1|a_2-a_1|}{a_1 a_2} + \frac{|b_1-b_2|}{a_2}.
\end{equation}
We then apply~\eqref{eq:elementary} to~\eqref{eq:sigma-1-1}, and use $\frac{| B(\pi_i,2^\ell)\cap U_1|}{|B(\pi_i,2^\ell)\cap \cS_\infty|} \leq 1$, obtaining that 
\begin{equation}\label{eq:sigma-1-2}
\begin{aligned}
&\bigr|\sigma_\ell^\omega(\pi_{i}) - \sigma_\ell^\omega(\pi_{i+1})\bigl|   \\
\leq &~\frac{\left||B(\pi_{i+1},2^\ell)\cap \cS_\infty|-|B(\pi_i,2^\ell)\cap \cS_\infty|\right|}{|B(\pi_{i+1},2^\ell)\cap \cS_\infty|} + \frac{\left||B(\pi_i,2^\ell)\cap U_1|-| B(\pi_{{i+1}},2^\ell)\cap U_1|\right|}{|B(\pi_{i+1},2^\ell)\cap \cS_\infty|}.
\end{aligned}
\end{equation}
Both of the numerators in the above expression can be bounded by
\begin{equation}\label{eq:sym-diff-bound}
  \begin{aligned}
  |(B(\pi_i,2^\ell)\Delta B(\pi_{i+1},2^\ell)) \cap \cS_\infty | &\leq |B(\pi_i,2^\ell)\Delta B(\pi_{i+1},2^\ell) | = 4(2^{\ell+1}+1)^{d-1}
\end{aligned}
\end{equation}
(with $\Delta$ denoting the symmetric difference between sets). Applying the latter to \eqref{eq:sigma-1-2} yields
\begin{equation}\label{eq:successive-1}
  |\sigma_\ell^\omega(\pi_i) - \sigma_\ell^\omega(\pi_{i+1})| \stackrel{\eqref{eq:sym-diff-bound},\eqref{eq:volume-concentration}}{\leq} \frac{8 (2^{\ell+1}+1)^{d-1}}{(1-\alpha)\eta(u)(2^{\ell+1}+1)^d} \leq \frac{4\cdot 2^{-\ell}}{(1-\alpha)\eta(u)}
\end{equation}
(the application of \eqref{eq:volume-concentration} is possible since $2^{\ell}\geq \Rden(\omega,\alpha)$ and $\pi_i\in\cS_\infty\cap B(0,r_{\alpha, 2^\ell})$, for every $0 \leq i \leq n-1$). 
Lastly, 
\[
  |\sigma_\ell^\omega(x) - \sigma_\ell^\omega(x+y)| \leq \sum_{i=0}^{n-1}|\sigma_\ell^\omega(\pi_i) - \sigma_\ell^\omega(\pi_{i+1})| \stackrel{\eqref{eq:successive-1}}{\leq} \frac{4\cdot 2^{-\ell}}{(1-\alpha)\eta(u)} \widehat{d}_\omega(x,x+y).
\]
The Lipschitz continuity \eqref{eq:Lip-cts} then follows since $\alpha<\frac{1}{3}$.
\medskip

We turn to the proof of (ii). Let $x\in \cS_\infty\cap B(0,r_{\alpha, 2^{\ell'}}-2^{\ell}-2^{\ell'})$. Note that 
\begin{align}
  (\sigma_{\ell'}^\omega)_{x,\ell}^\omega & = 
\frac{1}{|B(x,2^\ell)\cap\cS_\infty|} \sum_y \sum_z \frac{1}{|B(y,2^{\ell'})\cap\cS_\infty|} \IND_{B(x,2^\ell)\cap \cS_\infty}(y) \IND_{B(y,2^{\ell'})}(z) \IND_{U_1}(z) \notag \\
&= \frac{1}{|B(x,2^\ell)\cap\cS_\infty|} \sum_{z} \psi_{\ell,\ell'}^{\omega, x}(z) \IND_{U_1}(z) , \label{eq:sigma-1}
\end{align}
where 
\begin{equation}\label{eq:psi-1}
  \psi_{\ell,\ell'}^{\omega,x}(z) \stackrel{\mathrm{def}}{=} \sum_y \frac{1}{|B(y,2^{\ell'})\cap\cS_\infty|} \IND_{B(x,2^\ell)\cap B(z,2^{\ell'})\cap \cS_\infty}(y)
\end{equation}
(with the convention $0/0=0$). It is clear by definition of $\psi_{\ell,\ell'}^{\omega,x}$ that 
\begin{equation}\label{eq:psi-2}
\psi_{\ell,\ell'}^{\omega,x}(z) = 
\begin{cases}
  \sum_y \frac{1}{|B(y,2^{\ell'})\cap\cS_\infty|} \IND_{B(z,2^{\ell'})\cap \cS_\infty}(y) & \text{if $z\in B(x,2^{\ell}-2^{\ell'})$,}\\
   0 & \text{if $z\in \bbZ^d\setminus B(x,2^{\ell}+2^{\ell'})$.}
\end{cases}
\end{equation}
Moreover, we have for $z\in B(x,2^{\ell}-2^{\ell'})$,
\begin{equation}\label{eq:psi-LB}
  \psi_{\ell,\ell'}^{\omega,x}(z) \stackrel{\eqref{eq:psi-2},\eqref{eq:volume-concentration}}{\geq} \frac{|B(z,2^{\ell'})\cap \cS_\infty|}{(1+\alpha)\eta(u)(2^{\ell'+1}+1)^d}  \stackrel{\eqref{eq:volume-concentration}}{\geq} \frac{1-\alpha}{1+\alpha}
\end{equation}
(in both inequalities the use of \eqref{eq:volume-concentration} is possible since $y\in B(z,2^{\ell'})\subseteq B(x,2^\ell)\subseteq B(0,r_{\alpha, 2^{\ell'}}-2^{\ell'})$ and $z\in B(x,2^\ell-2^{\ell'})\subseteq B(0,r_{\alpha, 2^{\ell'}}-2\cdot 2^{\ell'})$). Note that by \eqref{eq:psi-2},  $\psi_{\ell,\ell'}^{\omega,x}(z)=0$ for $z\notin B(x,2^{\ell}+2^{\ell'})$, while for $z\in B(x,2^{\ell}+2^{\ell'})$, we have
\begin{equation}\label{eq:psi-UB}
  \psi_{\ell,\ell'}^{\omega,x}(z) \stackrel{\eqref{eq:psi-1},\eqref{eq:volume-concentration}}{\leq} \frac{|B(z,2^{\ell'})\cap \cS_\infty|}{(1-\alpha)\eta(u)(2^{\ell'+1}+1)^d} 
  \stackrel{\eqref{eq:volume-concentration}}{\leq} \frac{1+\alpha}{1-\alpha}  
\end{equation}
(noting again that the use of \eqref{eq:volume-concentration} is possible since $y\in B(x,2^\ell)\subseteq B(0,r_{\alpha,2^{\ell'}}-2^{\ell'})$ and $z\in B(x,2^{\ell}+2^{\ell'})\subseteq B(0,r_{\alpha, 2^{\ell'}})$), which shows that $\psi_{\ell,\ell'}^{\omega,x}(z)\leq \tfrac{1+\alpha}{1-\alpha}$ for all $z\in\bbZ^d$. 
We will use~\eqref{eq:psi-LB} to bound $(\sigma_{\ell'}^\omega)_{x,\ell}^\omega$ from below by restricting $\psi_{\ell,\ell'}^x$ on $B(x,2^{\ell}-2^{\ell'})$, and bound the same quantity from above by using the upper bound \eqref{eq:psi-UB} of $\psi_{\ell,\ell'}^x$ on all its support $B(x,2^{\ell}+2^{\ell'})$. As for the lower bound, we recall~\eqref{eq:sigma-1} and see that
\begin{align*}
(\sigma_{\ell'}^\omega)_{x,\ell}^\omega &\stackrel{\eqref{eq:psi-LB}}{\geq} \frac{1-\alpha}{1+\alpha} \frac{1}{|B(x,2^\ell)\cap\cS_\infty|} \sum_{z\in B(x,2^{\ell}-2^{\ell'})} \IND_{U_1}(z) \\
& \geq \frac{1-\alpha}{1+\alpha} \frac{|B(x,2^{\ell}-2^{\ell'})\cap U_1|}{|B(x,2^\ell)\cap\cS_\infty|}\\
& \geq \frac{1-\alpha}{1+\alpha} \left[\sigma_\ell^\omega(x) - \frac{|(B(x,2^\ell)\setminus B(x,2^{\ell}-2^{\ell'}))\cap U_1|}{|B(x,2^{\ell})\cap \cS_\infty|}\right],
\end{align*}
using a simple union bound in the last step. As we have 
\begin{align*}
  |(B(x,2^\ell)&\setminus B(x,2^{\ell}-2^{\ell'}))\cap U_1|\leq |B(x,2^\ell)\setminus B(x,2^{\ell}-2^{\ell'})| \\ 
  &= |B(x,2^\ell)| - |B(x,2^{\ell}-2^{\ell'})| = (2^{\ell+1}+1)^d - (2^{\ell+1}-2^{\ell'+1}+1)^d\\
  &= \int_{2^{\ell+1}-2^{\ell'+1}+1}^{2^{\ell+1}+1} d u^{d-1} \De u\leq d (2^{\ell+1}+1)^{d-1}\cdot 2^{\ell'+1},
\end{align*}
it follows that 
\begin{equation}\label{eq:sigma-2}
  \begin{aligned}
  (\sigma_{\ell'}^\omega)_{x,\ell}^\omega & \stackrel{\eqref{eq:volume-concentration}}{\geq} \frac{1-\alpha}{1+\alpha} \sigma_\ell^\omega(x) - \frac{1-\alpha}{1+\alpha} \cdot \frac{d (2^{\ell+1}+1)^{d-1}\cdot 2^{\ell'+1}}{(1-\alpha) \eta(u) (2^{\ell+1}+1)^d}\\
  & \geq \frac{1-\alpha}{1+\alpha} \sigma_\ell^\omega(x) - \frac{d\cdot 2^{\ell'-\ell}}{(1+\alpha)\eta(u)}.
\end{aligned}
\end{equation}
Similarly, for the upper bound, we obtain 
\begin{equation}\label{eq:sigma-3}
  (\sigma_{\ell'}^\omega)_{x,\ell}^\omega \leq \frac{1+\alpha}{1-\alpha} \sigma_\ell^\omega(x) + \frac{1+\alpha}{(1-\alpha)^2} \cdot \frac{d (2^{\ell+1}+2^{\ell'+1}+1)^{d-1}\cdot 2^{\ell'+1}}{\eta(u)(2^{\ell+1}+1)^d}.
\end{equation}
Using inequality $(a+b)^p \leq 2^{p-1}(a^p+b^p)$ for $a,b\in \bbR_+$ and $p\geq 1$, one obtains
\begin{align}
  (\sigma_{\ell'}^\omega)_{x,\ell}^\omega &  \leq \frac{1+\alpha}{1-\alpha} \sigma_\ell^\omega(x) + \frac{(1+\alpha)d\cdot 2^{d-1} \cdot 2^{\ell'-\ell}}{(1-\alpha)^2\eta(u)} \label{eq:sigma-4}.
\end{align}
Finally, we note that $\frac{1+\alpha}{(1-\alpha)^2}\leq 3$ for $\alpha<\frac{1}{3}$. Combining \eqref{eq:sigma-2} and \eqref{eq:sigma-4} yields the claim (ii).
\end{proof}

\begin{lemma}
Let $(\Gamma,\cA,\mathrm{P})$ be a probability space and $Y$ be a $[0,1]$-valued random variable with finite expectation $\mu= \mathrm{E}[Y]$. Then for $0\leq \delta \leq \mu \wedge (1-\mu)$, at least one of the equations below holds:
\begin{subequations}
  \begin{empheq}[left=\empheqlbrace]{align}
&\mathrm{P}(Y>\mu+\delta) \geq \frac{\delta}{2} \quad\text{and}\quad \mathrm{P}(Y<\mu-\delta) \geq \frac{\delta}{2}\\
&\mathrm{P}(\mu-\delta\leq Y\leq \mu+\delta)\geq \frac{1}{4}-\frac{\delta}{2}.
\end{empheq}
\end{subequations}
\end{lemma}

We refer to \cite[Lemma 1.2]{nitzschner2017solidification} for the proof of the previous lemma. Applying this result with $\mathrm{P}=  \mu_{x,\ell}^\omega$,  $Y=  \sigma_{\ell'}^\omega$, and
\[
\mu = (\sigma_{\ell'}^\omega)_{x,\ell}^\omega = \frac{1}{|B(x,2^\ell)\cap \cS_\infty|} \sum_{y\in B(x,2^\ell)\cap \cS_\infty} \sigma_{\ell'}^\omega(y),
\]
we obtain that for $0\leq \delta \leq (\sigma_{\ell'}^\omega)_{x,\ell}^\omega \wedge (1-(\sigma_{\ell'}^\omega)_{x,\ell}^\omega)$, at least one of the equations below holds:
\begin{subequations}\label{eq:alter2}
  \begin{empheq}[left=\empheqlbrace]{align}
    & \mu_{x,\ell}^\omega\big[\sigma_{\ell'}^\omega > (\sigma_{\ell'}^\omega)_{x,\ell}^\omega + \delta\big] \geq \frac{\delta}{2}
    \quad \text{and} \quad
    \mu_{x,\ell}^\omega\big[\sigma_{\ell'}^\omega < (\sigma_{\ell'}^\omega)_{x,\ell}^\omega - \delta\big] \geq \frac{\delta}{2}, \label{eq:alter2-i}\\
    &
    \mu_{x,\ell}^\omega\big[(\sigma_{\ell'}^\omega)_{x,\ell}^\omega - \delta \leq \sigma_{\ell'}^\omega \leq (\sigma_{\ell'}^\omega)_{x,\ell}^\omega + \delta\big] \geq \frac{1}{4} - \frac{\delta}{2}.\label{eq:alter2-ii}
  \end{empheq}
  \end{subequations}
We will use the observation to prove the following proposition, which is analogous to \cite[Proposition 1.3]{nitzschner2017solidification} and \cite[Proposition 4.5]{CN2021disconnection}, and shows that a random walk has a non-degenerate probability of entering the set where the local density $\sigma_{\ell'}^\omega$ takes values close to its average over a ball before exiting the latter.
\medskip

In addition to the condition concerning regular volumes that appeared in Lemma~\ref{lem:sigma-regular}, we will impose further conditions on $\ell$ in the next proposition for the regularity of random walks on the percolation cluster. These, together with other technical conditions, are encapsulated in the minimal scale $\ell_{\min}(\omega, \alpha, \delta)$ (see \eqref{eq:l_min}). A similar notion of a minimal scale for such regularity properties was previously introduced in \cite[Proposition 4.5]{CN2021disconnection}; however, our definition additionally depends on the particular realization $\omega$ and the degree of volume irregularity $\alpha$, reflecting the inherent degeneracies present in the percolation cluster.
Recall the scales $\Rhk(\omega, \alpha)$ and $\Rkhk(\omega, \alpha, \vartheta)$ defined in Proposition \ref{prop:uni-HK}, the set of regular configurations, $\Omega_{\mathrm{reg}}^{\alpha, \vartheta}$, defined in \eqref{eq:Omega_reg}, and the stopping time $\tau_r$ defined in \eqref{eq:tau_r}.

\begin{prop}[$d\geq 3$, \refA, \eqref{eq:assum-4-1}]\label{prop:positive-probability-2}
  Let $\alpha\in (0,\frac{1}{3})$.
  For $0<\delta \leq (\sigma_{\ell'}^\omega)_{x,\ell}^\omega \wedge (1-(\sigma_{\ell'}^\omega)_{x,\ell}^\omega)\wedge \frac{1}{4}$, 
  there exists a minimal scale $\ell_{\min}(\omega,\delta,\alpha)$ (defined in \eqref{eq:l_min}), a positive constant $\delta'(\delta)\in (0,1)$ (defined in \eqref{eq:delta'}) such that for scales $\ell>\ell'\geq \ell_{\min}(\omega,\delta,\alpha) \in \bbN$ and, $\omega \in \Omega_{\mathrm{reg}}^{\alpha, \delta'(\delta)}$, we have 
  \begin{equation}\label{eq:nondeg-prob}
  P_x^\omega \left[H_{ \{ \sigma_{\ell'}^\omega \in[(\sigma_{\ell'}^\omega)_{x,\ell}^\omega-\delta,(\sigma_{\ell'}^\omega)_{x,\ell}^\omega+\delta]\}} < \tau_{2^\ell}\right] \geq c_1(\delta),
  \end{equation}
  for all $x\in \cS_\infty \cap B(0,r_{\alpha, 2^{\ell'}})$, where here and in the following we use the abbreviation $\{\sigma_{\ell'}^\omega\in[(\sigma_{\ell'}^\omega)_{x,\ell}^\omega-\delta,(\sigma_{\ell'}^\omega)_{x,\ell}^\omega+\delta]\}$ for $\{y\in \bbZ^d \,:\, \sigma_{\ell'}^\omega(y)\in[(\sigma_{\ell'}^\omega)_{x,\ell}^\omega-\delta,(\sigma_{\ell'}^\omega)_{x,\ell}^\omega+\delta]\}$.
\end{prop}

\begin{proof}
Let $x\in \cS_\infty \cap B(0,r_{\alpha, 2^{\ell'}})$ and $0<\delta\leq (\sigma_{\ell'}^\omega)_{x,\ell}^\omega \wedge (1-(\sigma_{\ell'}^\omega)_{x,\ell}^\omega)\wedge \frac{1}{4}$. By the definition of $\mu_{x,\ell}^\omega$ (see \eqref{eq:local-density-measure}), we can take \begin{equation}\label{eq:delta'}
  \delta'(\delta) \stackrel{\mathrm{def}}{=} \frac{\eta(u)\delta}{6d},
\end{equation}
so that
\begin{equation}\label{eq:con-deltap}
  \mu_{x,\ell}^\omega [\{y\in \bbZ^d \, : \,  |y-x|_\infty > 2^\ell(1-\delta')\}] \leq \frac{|B(x,2^\ell)\setminus B(x,2^{\ell}(1-\delta'))|}{(1-\alpha)\eta(u)(2^{\ell+1}+1)^d} \leq \frac{3d\delta'}{2\eta(u)} = \frac{\delta}{4}
\end{equation}
holds for all $2^\ell\geq \Rden(\omega,\alpha)$.

We will apply Proposition \ref{prop:uni-HK} (with $\vartheta = \delta'(\delta)$),  Proposition \ref{prop:uni-den}, and Lemma \ref{lem:sigma-regular}. To this end, we define 
\begin{equation}\label{eq:l_min}
  \begin{aligned}
    \ell_{\min} (\omega,\delta,\alpha) & \stackrel{\mathrm{def}}{=}  \big\lceil \log_2 \Rhk(\omega,\alpha) \big\rceil \vee \big\lceil \log_2 \Rkhk(\omega,\alpha, \delta'(\delta)) \big\rceil \vee \big\lceil \log_2 \Rden(\omega, \alpha) \big\rceil \\
& 
\vee \min \{\ell\in \bbN: 2^{-\ell}\leq \tfrac{\delta}{8} \wedge \tfrac{\eta(u)}{6}\delta\}.
  \end{aligned}
\end{equation}

We now let $\ell>\ell'\geq \ell_{\min}(\omega,\delta,\alpha)$ and show that \eqref{eq:nondeg-prob} holds. Assume we are in the situation where~\eqref{eq:alter2-i} is fulfilled. 
By \eqref{eq:con-deltap}, it follows that
\begin{empheq}[left=\empheqlbrace]{align}\label{eq:size-LB}
\mu_{x,\ell}^\omega[\{\sigma_{\ell'}^\omega>(\sigma_{\ell'}^\omega)_{x,\ell}^\omega+\delta\}\cap B(x,2^\ell(1-\delta'))] \geq \frac{\delta}{4},\\
\mu_{x,\ell}^\omega[\{\sigma_{\ell'}^\omega<(\sigma_{\ell'}^\omega)_{x,\ell}^\omega-\delta\}\cap B(x,2^\ell(1-\delta'))] \geq \frac{\delta}{4}.
\end{empheq}
We recall the notation related to random walks introduced at the start of Subsection \ref{subsec:random-walks}, up to~\eqref{eq:KilledHeatKernel}. We have
\begin{align*}
P_x^\omega &\left[
H_{\{\sigma_{\ell'}^\omega>(\sigma_{\ell'}^\omega)_{x,\ell}^\omega+\delta\}\cap B(x,2^\ell(1-\delta'))} < T_{B(x,2^\ell)}
\right] \\
& \geq  ~  P_x^\omega \left[ 
X_{4^\ell} \in \{\sigma_{\ell'}^\omega>(\sigma_{\ell'}^\omega)_{x,\ell}^\omega+\delta\}\cap B(x,2^\ell(1-\delta')), 4^\ell < T_{B(x,2^\ell)}
\right]\\
& =  \sum_{z\in \{\sigma_{\ell'}^\omega>(\sigma_{\ell'}^\omega)_{x,\ell}^\omega+\delta\}\cap B(x,2^\ell(1-\delta'))\cap \cS_\infty} \mu_z q_{4^\ell,B(x,2^\ell)}^\omega(x,z)\\
& \stackrel{\eqref{eq:killed-heat-kernel},\mu_z\geq 1}{\geq}  |\{\sigma_{\ell'}^\omega>(\sigma_{\ell'}^\omega)_{x,\ell}^\omega+\delta\}\cap B(x,2^\ell(1-\delta'))\cap \cS_\infty|\cdot c_{\text{khk}1}(\delta') 2^{-\ell d}\\
& \stackrel{\eqref{eq:size-LB}}{\geq}  |B(x,2^\ell)\cap \cS_\infty| \cdot \frac{\delta}{4}\cdot c_{\text{khk}1}(\delta') 2^{-\ell d}\\
& \stackrel{\eqref{eq:volume-concentration}}{\geq} (1-\alpha) \eta(u)\cdot (2^{\ell+1}+1)^d \cdot \frac{\delta}{4}\cdot c_{\text{khk}1}(\delta') 2^{-\ell d}  \stackrel{\alpha < \frac{1}{3}}{\geq} C(\delta),
\end{align*}
where the application of \eqref{eq:killed-heat-kernel} is possible since $x\in \cS_\infty \cap B(0,r_{\alpha, 2^{\ell'}})\subseteq \cS_\infty\cap B(0,r_{\alpha, 2^{\ell}})$ and $z\in \cS_\infty \cap B(x,(1-\delta')2^\ell)$, and the use of \eqref{eq:volume-concentration} is valid since $x\in \cS_\infty \cap B(0,r_{\alpha, 2^{\ell'}})\subseteq \cS_\infty\cap B(0,r_{\alpha, 2^{\ell}})$.
Similarly, for all $y\in B(x,2^\ell(1-\delta'))\cap \cS_\infty$, we also have
\[
P_y^\omega \big[ 
H_{\{\sigma_{\ell'}^\omega<(\sigma_{\ell'}^\omega)_{x,\ell}^\omega-\delta\}\cap B(x,2^\ell(1-\delta'))} < T_{B(x,2^\ell)}
\big] \geq C(\delta).
\]
It then follows from the strong Markov property that the random walk starting from $x$ enters $\{\sigma_{\ell'}^\omega>(\sigma_{\ell'}^\omega)_{x,\ell}^\omega+\delta\}$ and then $\{\sigma_{\ell'}^\omega<(\sigma_{\ell'}^\omega)_{x,\ell}^\omega-\delta\}$ before exiting $B(x,2^\ell)$ with probability at least $C(\delta)^2$. Note that the trajectory of this walk must lie in $B(x,2^\ell)\cap\cS_\infty$, so by the Lipschitz continuity (using $\ell'\geq \ell_{\min}$) of $\sigma_{\ell'}^\omega$ in the subgraph $\cS_\infty\cap B(0,r_{\alpha, 2^\ell})$ with Lipschitz constant smaller than $\delta$, it must enter $\{(\sigma_{\ell'}^\omega)_{x,\ell}^\omega-\delta<\sigma_{\ell'}^\omega<(\sigma_{\ell'}^\omega)_{x,\ell}^\omega+\delta\}$ before exiting $B(x,2^\ell)$. \smallskip

It remains to consider the situation \eqref{eq:alter2-ii}. Note that 
\[
  \mu_{x,\ell}^\omega\big[(\sigma_{\ell'}^\omega)_{x,\ell}^\omega-\delta\leq \sigma_{\ell'}^\omega \leq (\sigma_{\ell'}^\omega)_{x,\ell}^\omega+\delta\big]\geq \frac{1}{4}-\frac{\delta}{2} \geq \frac{1}{8}.
\]
We can argue using the killed heat kernel bound again to show 
\[
P_x^\omega 
\big[ H_{\{\sigma_{\ell'}^\omega\in[(\sigma_{\ell'}^\omega)_{x,\ell}^\omega-\delta,(\sigma_{\ell'}^\omega)_{x,\ell}^\omega+\delta] \}} < T_{B(x,2^\ell)}\big] \geq C(\delta)
.
\]
This completes the proof.
\end{proof}

Based on the previous propositions, we are now in a position to introduce ``properly separated'' scales $2^{\ell_0}>\ldots>2^{\ell_J}$ that will enter the definition of the resonance sets in the next subsection. Our goal is to show that starting from a point $x$ of well-balanced relative volume (see \eqref{eq:near-1/2}), there is positive probability for a random walk to enter the region where all local densities $\widetilde\sigma_{\ell_j}^\omega(\,\cdot\,)$, $0\leq j\leq J$, lie in some interval $[\widetilde{\alpha},1-\widetilde{\alpha}]$ (with $\widetilde{\alpha}$ a dimension-dependent constant, see~\eqref{eq:happen2-C}), before travelling to a point at sup-distance $\frac{3}{2}\cdot 2^{\ell_0}$ from $x$. This crucially enters the definition of the resonance set in the next subsection. \medskip

To this end, we fix an integer $J\geq 1$ and define
\begin{equation}\label{eq:alpha_J}
\alpha(J) \stackrel{\mathrm{def}}{=} \frac{1}{2} \left[1-\frac{2}{(\frac{10}{9})^{1/J}+1}\right]
\quad \mbox{and} \quad 
\delta(J)\stackrel{\mathrm{def}}{=} \frac{\alpha(J)}{4} = \frac{1}{8}\left[1-\frac{2}{(\frac{10}{9})^{1/J}+1}\right].
\end{equation}
We intentionally reuse the symbols $\alpha$ and $\delta$ to denote quantities that now depend explicitly on the integer $J$, as defined above. This dependence will be fixed throughout the remainder of the article. The reason for these specific choices will become clear in~\eqref{eq:alpha}.
We say the scales $\ell_0>\ldots>\ell_J$ are \emph{properly separated} if for all $j\in \{0,\ldots,J-1\}$, it holds that
\begin{subequations}\label{eq:WellS-1}
  \begin{empheq}[left=\empheqlbrace]{align}
    \ell_j &\geq \ell_{j+1} + L(J), \label{eq:WellS-1-i} \\
    2^{\ell_j} &< r_{\alpha(J), 2^{\ell_{j+1}}} - 2^{\ell_{j+1}}, \label{eq:WellS-1-ii}
  \end{empheq}
  \end{subequations}
where 
\begin{equation}\label{eq:WellS-2}
  \mbox{$L(J)$ is the smallest integer $L\geq 5$ such that $c_0 \cdot 2^{-L}\leq \delta(J)$}
\end{equation}
(recall $c_0$ in \eqref{eq:c0'cont}). As a consequence of \eqref{eq:WellS-1-i}, we have 
\begin{equation}\label{eq:cons-WellS}
  c_0 2^{\ell_{j+1}-\ell_j} \leq c_0 2^{-L(J)} \leq \delta(J).
\end{equation}
We will also consider a sequence of intervals for $j=0,\ldots,J$, given by
\begin{equation}\label{eq:I_j}
I_j \stackrel{\mathrm{def}}{=} \left[ \frac{3}{4}\Big(\frac{1-\alpha(J)}{1+\alpha(J)}\Big)^j - \frac{1+\alpha(J)}{4}, ~ 
\frac{3}{4}\Big(\frac{1+\alpha(J)}{1-\alpha(J)}\Big)^j - \frac{1-\alpha(J)}{4}\right],
\end{equation}
and introduce the stopping times 
\begin{equation}\label{eq:stopping-times}
  \gamma_0 = H_{\{\sigma_{\ell_0}^\omega\in I_0 \}},\quad \gamma_{j+1} = \gamma_j + H_{\{\sigma_{\ell_{j+1}}^\omega\in I_{j+1} \}}\circ \theta_{\gamma_j}.
\end{equation}
We recall the function $\delta'(\,\cdot\,)$ and the minimal scale $\ell_{\min}(\omega,\delta,\alpha)$ from Proposition~\ref{prop:positive-probability-2}, given by~\eqref{eq:delta'} and~\eqref{eq:l_min}, respectively.
\begin{prop}[$d\geq 3$, \refA, \eqref{eq:assum-4-1}]\label{prop:positive-probability-adv}
  Let $J\geq 1$, $\omega \in \Omega_{\mathrm{reg}}^{\alpha(J), \delta'(\delta(J))}$ be fixed.
  Assume $(\ell_j)_{j=0}^J$ to be properly separated in the sense of \eqref{eq:WellS-1}, and furthermore that $\tfrac{2}{3}r_{\alpha(J), 2^{\ell_J}} > 2^{\ell_0}>\ldots>2^{\ell_J} \geq 2^{\ell_{\min}(\omega,\delta(J),\alpha(J))}$.
  Define
  \begin{equation}
  \cC \stackrel{\mathrm{def}}{=} \{\gamma_0=0\} \cap \bigcap_{0\leq j <J} \{\gamma_{j+1}< \gamma_{j} + \tau_{2^{\ell_j}}\circ \theta_{\gamma_j}\}.
  \end{equation}
  Then $P_x^\omega [\cC]\geq c_2(J)$ holds for $x\in \cS_\infty \cap B(0,r_{\alpha(J), 2^{\ell_J}}-\frac{3}{2}2^{\ell_0})$ satisfying 
  \begin{equation}\label{eq:near-1/2}
    \sigma_{\ell_0}^\omega(x) \in \left[\frac{1}{2} - \frac{1}{2^{\ell_{\min}}},~\frac{1}{2} +\frac{1}{2^{\ell_{\min}}}\right], \quad \ell_{\min}=\ell_{\min}(\omega, \delta(J),\alpha(J)).
  \end{equation}
  Moreover, on the set of trajectories constituting the event $\cC$, one has  
  \begin{equation}\label{eq:happen1-C}
  \sup\{|X_s-X_{\gamma_j}|_\infty \, : \,\gamma_j \leq s \leq \gamma_J\} \leq \frac{3}{2}\cdot 2^{\ell_j}\quad \mbox{for $j=0,\ldots,J$,}
  \end{equation}
  and 
  \begin{equation}\label{eq:happen2-C}
  \widetilde{\sigma}^\omega_{\ell_j} (X_{\gamma_J}) \in [\widetilde{\alpha},1-\widetilde{\alpha}]\quad \mbox{for $j=0,\ldots,J$ with $\widetilde\alpha=\dfrac{3}{10}4^{-d}$}.
  \end{equation}
\end{prop}

\begin{proof}
We first record the following observations on the parameters $\alpha(J)$ and $\delta(J)$ defined in~\eqref{eq:alpha_J}, and the intervals $I_j$ from~\eqref{eq:I_j}. 
\begin{enumerate}[label=(\roman*)]
\item As $\alpha(J)$ is a decreasing function of $J$, $\alpha(J) \leq \alpha(1)=\frac{1}{38}(<\frac{1}{3})$ and $\delta \leq \frac{1}{38\cdot 4} (<\frac{1}{4})$. 

\item It is immediate from the expression of $\alpha(J)$ that
$
\alpha(J) < 1-\frac{2}{(\frac{10}{9})^{1/J}+1},
$
which implies 
\begin{equation}\label{eq:alpha}
  \left(\frac{1+\alpha(J)}{1-\alpha(J)}\right)^J = \left(\frac{2}{1-\alpha(J)} -1\right)^J < \frac{10}{9}.
\end{equation}

\item  Denote the left and right endpoint of $I_j$ by $L_j$ and $R_j$ respectively, for $j\in \{0,\ldots,J\}$. As $L_j$ is decreasing and $R_j$ is increasing in $j$, using observation (ii) and $\alpha(J) \leq \frac{1}{38}$ in observation (i), we know that for $j\in \{0,\ldots,J\}$
\begin{equation}\label{eq:I_low}
  L_j \geq L_J \geq \frac{3}{4}\cdot \frac{9}{10} - \frac{1}{4} - \frac{\alpha(J)}{4} \geq \frac{159}{380} \geq \frac{1}{3}
\end{equation}
and 
\begin{equation}
  R_j \leq R_J \leq \frac{3}{4}\cdot \frac{10}{9} - \frac{1}{4} + \frac{\alpha(J)}{4} \leq \frac{269}{456} \leq \frac{2}{3},
\end{equation}
which shows $I_j \subseteq [\frac{1}{3},\frac{2}{3}]$.
\end{enumerate}

\medskip

Consider $x\in \cS_\infty \cap B(0,r_{\alpha(J), 2^{\ell_J}}-\frac{3}{2}2^{\ell_0} )$ satisfying \eqref{eq:near-1/2}.
We will repeatedly apply Proposition \ref{prop:positive-probability-2} and show by induction on $j\in \{0,\ldots,J\}$ that 
\begin{equation}\label{eq:induction-C_j}
  \begin{split}
    &P_x^\omega[\cC_j] \geq c_1(\delta(J))^j \mbox{ with}\\
    & \cC_j \stackrel{\mathrm{def}}{=} \{\gamma_0=0, \gamma_1< \gamma_0+ \tau_{2^{\ell_0}}\circ \theta_{\gamma_0},\ldots, \gamma_j< \gamma_{j-1}+ \tau_{2^{\ell_{j-1}}}\circ \theta_{\gamma_{j-1}}\}
  \end{split}
\end{equation}
(note that $\cC=\cC_J$), which will conclude the first part.

When $j=0$, $I_0=[\frac{1}{2}-\delta(J), ~ \frac{1}{2}+\delta(J)]$. Since $2^{-\ell_{\min}(\omega,\delta(J),\alpha(J))}\leq \frac{\delta(J)}{8}<\delta(J)$, we find that $P_x^\omega[\gamma_0=0]=1$. 
Assume now that \eqref{eq:induction-C_j} is satisfied for some $0\leq j <J$. Firstly, we want to investigate the range of $X_{\gamma_j}$ to make sure Lemma \ref{lem:sigma-regular}(ii) is applicable with $X_{\gamma_j}$ in place of $x$ therein, for all $j\in\{1,\ldots,J\}$. On $\cC_j$, by the definition of $\gamma_j$, we have
\begin{equation}\label{eq:range-Xgamma-1}
  X_{\gamma_j} \in \cS_\infty \cap B(0, r_{\alpha(J), 2^{\ell_J}}-\frac{3}{2}2^{\ell_0} + 2^{\ell_0}+\ldots+2^{\ell_{j-1}}). 
\end{equation}
Note that by \eqref{eq:WellS-1-i}, 
\begin{equation}\label{eq:calcu-stopping}
  2^{\ell_0}+\ldots+ 2^{\ell_j}+2^{\ell_{j+1}}\leq 2^{\ell_0} \sum_{m=0}^{j+1} 2^{-mL(J)} \leq \frac{3}{2} 2^{\ell_0}.
\end{equation}
It then follows from \eqref{eq:range-Xgamma-1} that 
\begin{equation}
  X_{\gamma_j} \in \cS_\infty \cap  B(0, r_{\alpha(J), 2^{\ell_{j+1}}}-2^{\ell_j}-2^{\ell_{j+1}}).
\end{equation}
Together with \eqref{eq:WellS-1-ii}, we know that Lemma \ref{lem:sigma-regular}(ii) is applicable with $X_{\gamma_j}$ for all $j\in\{1,\ldots,J\}$. 
Denote $(\sigma_{\ell_{j+1}}^\omega)_{X_{\gamma_j},\ell_j}^\omega$ by $(\beta_{j+1}^\omega)'$. On $\cC_j$, by Lemma \ref{lem:sigma-regular}(ii) and \eqref{eq:cons-WellS}, we have
\[
(\beta_{j+1}^\omega)' \in \left[\frac{1-\alpha}{1+\alpha} \sigma_{\ell_j}^\omega(X_{\gamma_j})-\delta, ~ \frac{1+\alpha}{1-\alpha} \sigma_{\ell_j}^\omega(X_{\gamma_j})+\delta\right].
\]
As $\sigma^\omega_{\ell_j}(X_{\gamma_j})\in I_j$ by the definition of $\gamma_j$, it is clear by our choice of $\delta=\frac{\alpha}{4}$ that
\[
[(\beta_{j+1}^\omega)'-\delta, (\beta_{j+1}^\omega)'+\delta] \subseteq
\left[\frac{1-\alpha}{1+\alpha}L_j-2\delta, ~ \frac{1+\alpha}{1-\alpha}R_j+2\delta\right] = I_{j+1}
\]
(recall that $L_j$ and $R_j$ are the left and right endpoint of $I_j$ respectively).
It then holds on $\cC_j$ that 
\begin{equation}\label{eq:hit-compare}
  H_{\left\{ \sigma_{\ell_{j+1}}^\omega\in \left[(\beta_{j+1}^\omega)'-\delta,(\beta_{j+1}^\omega)'+\delta\right] \right\}} \geq H_{\{\sigma_{\ell_{j+1}}^\omega\in I_{j+1} \}}.
\end{equation}
We can now apply the strong Markov property at time $\gamma_j$ and find that 
\begin{align*}
  &P_x^\omega \left[\cC_{j+1}\right]  =P_x^\omega \left[\cC_j \cap\left\{\gamma_{j+1}< \gamma_{j}+ \tau_{2^{\ell_{j}}}\circ \theta_{\gamma_{j}}\right\}\right] =E_x^\omega \left[\cC_j, P_{X_{\gamma_j}}\left[H_{\{\sigma^\omega_{\ell_{j+1}} \in I_{j+1}\}}<\tau_{2^{\ell_j}}\right]\right]\\
  & \stackrel{\eqref{eq:hit-compare}}{\geq} E_x^\omega \left[\cC_j, P_{X_{\gamma_j}}\left[H_{\{ \sigma_{\ell_{j+1}}^\omega\in \left[(\beta_{j+1}^\omega)'-\delta,(\beta_{j+1}^\omega)'+\delta\right] \}}<\tau_{2^{\ell_j}}\right]\right] \underset{\text{induction}}{\stackrel{\text{Prop.\ref{prop:positive-probability-2}}}{\geq}} c_1(\delta(J))^{j+1} .
\end{align*}
The first part of the proposition is concluded. 

\medskip

We now turn to~\eqref{eq:happen1-C} and~\eqref{eq:happen2-C}. First, \eqref{eq:happen1-C} follows by the definition of stopping times \eqref{eq:stopping-times} and a calculation using triangle inequalities similar to \eqref{eq:calcu-stopping}, and we refer to \cite[(1.40)]{nitzschner2017solidification} for details.
For the lower bound in \eqref{eq:happen2-C}, 
\begin{align*}
  \widetilde{\sigma}_{\ell_j}(X_{\gamma_J}) &= \frac{|B(X_{\gamma_J},4\cdot 2^{\ell_j}) \cap U_1|}{|B(X_{\gamma_J},4\cdot 2^{\ell_j})\cap \cS_\infty|} \geq \frac{|B(X_{\gamma_j}, 2^{\ell_j}) \cap U_1|}{|B(X_{\gamma_j}, 2^{\ell_j})\cap \cS_\infty|} \cdot \frac{|B(X_{\gamma_j}, 2^{\ell_j})\cap \cS_\infty|}{|B(X_{\gamma_J},4\cdot 2^{\ell_j})\cap \cS_\infty|} \\
  & \geq \sigma_{\ell_j} (X_{\gamma_j}) \cdot \frac{1-\alpha}{1+\alpha} \cdot \frac{(2^{\ell_j+1}+1)^d}{(4\cdot 2^{\ell_j+1}+1)^d }
  > \sigma_{\ell_j} (X_{\gamma_j}) \cdot \frac{1-\alpha}{1+\alpha} \cdot 4^{-d} \\
  & \underset{\eqref{eq:alpha}}{\stackrel{\eqref{eq:stopping-times},\eqref{eq:I_low}}{\geq}}
  \frac{1}{3} \cdot \frac{9}{10}\cdot 4^{-d}=\widetilde{\alpha}.
\end{align*}
Replacing $U_1$ by $U_0$ in the above calculations then yields the upper bound (see \eqref{eq:sigma_l^omega-2}). The proof is therefore complete.
\end{proof}

\subsection{Resonance set}

In this subsection, our goal is to show that a simple random walk on the percolation cluster starting from $x\in A_N\cap \cS_\infty$ is very unlikely to avoid a certain resonance set, defined in~\eqref{eq:Resonance-set}. The main result is stated in Theorem~\ref{thm:asy-Phi}. \medskip

We begin by providing the definition of the resonance set. To describe it, we first introduce several parameters: the integer $\ell_*\geq 0$ controls from above the scales under consideration; the integer $J\geq 1$ represents the ``strength'' of resonance, the integer $L\geq L(J)$ (with $L(J)$ as in \eqref{eq:WellS-2}) determines the separation between successive scales, and the integer $I$ denotes the number of scales examined. The resonance set also depends crucially on the realization of the environment $\omega$, which we choose in a set of $\mathbb{P}^u$-measure one, and which is characterized by good regularity properties holding for all possible choices of the parameter $J$. Let 
\begin{align}
\label{eq:ell-0-def}
  \ell_0 &\stackrel{\mathrm{def}}{=} \sup\{\ell\in (J+1)L\bbN \, : \, \ell\leq \ell_*\};\\
  \label{eq:A-ast_def}
  \cA_* &\stackrel{\mathrm{def}}{=} \{\ell\in L\bbN \, : \, \ell_0 \geq \ell > \ell_0 - I(J+1)L\}, \quad |\cA_*|=(J+1)I; \\
  \cA &\stackrel{\mathrm{def}}{=} \{\ell\in (J+1)L\bbN \, : \, \ell_0 \geq \ell > \ell_0 - I(J+1)L\}, \quad |\cA|=I.
\end{align}
Recall the set $\Omega_{\mathrm{reg}}^{\alpha, \vartheta}$ in \eqref{eq:Omega_reg} and let
\begin{equation}\label{eq:Omega_reg-2}
  \Omega_{\mathrm{reg}} \stackrel{\mathrm{def}}{=} \bigcap_{J=1}^\infty \Omega_{\mathrm{reg}}^{\alpha(J), \delta'(\delta(J))} \stackrel{\eqref{eq:Omega_reg}}{\in} \cF.
\end{equation}
It is clear that this is a set of full $\bbP^u$-measure.
For $\omega\in \Omega_{\mathrm{reg}}$, we say that $\ell_*$ is \emph{$(I,J,L,N)$-compatible}, if with $\ell_0$ as in~\eqref{eq:ell-0-def}, one has 
\begin{subequations}\label{eq:compatible}
  \begin{empheq}[left=\empheqlbrace]{align}
    &\ell_0 - (I+1)(J+1)L > \ell_{\min}(\omega, \delta(J), \alpha(J)), \label{eq:compatible-i} \\
    &r_{\alpha(J), 2^{\ell_0 - (I+1)(J+1)L}} - 4 \cdot 2^{\ell_0} \geq b_N \label{eq:compatible-ii}
  \end{empheq}
  \end{subequations}
(with $b_N$ as in~\eqref{eq:a_N-bound}). \medskip

The resonance set can be seen informally as the set of points such that there are at least $J$ members among the collection (indexed by $\cA_*$) of $(J+1)I$ local densities $\widetilde{\sigma}^\omega_\ell$ (see \eqref{eq:sigma_l^omega}) attaining values in $[\widetilde{\alpha},1-\widetilde{\alpha}]$; more precisely, we set
\begin{equation}
\label{eq:Resonance-set}
  \Res = \Res^\omega (U_0,I,J,L,\ell_*) \stackrel{\mathrm{def}}{=} \Big\{x\in \cS_\infty\, : \, \sum_{\ell\in \cA_*} \IND_{\{\widetilde{\sigma}^\omega_\ell (x) \in [\widetilde{\alpha},1-\widetilde{\alpha}]\}} \geq J\Big\}.
\end{equation}

For $N,J,I\geq 1$, $L\geq L(J)$, we define 
  \begin{equation}
  \label{eq:Phi-Def}
    \Phi^\omega_{J,I,L,N} \stackrel{\mathrm{def}}{=} \sup_{\ell_*} \sup_{U_0\in \cU_{\ell_*,N}^\omega} \sup_{x\in A_N \cap \cS_\infty} P_x^\omega \left[H_{\Res}=\infty\right]
  \end{equation}
  (with the convention $\sup\varnothing=-\infty$) where the first supremum is over all $\ell_*$ which are $(I,J,L,N)$-compatible.

\begin{theorem}[$d\geq 3$, \refA, \eqref{eq:ass-bounded}]\label{thm:asy-Phi}
  Let $\omega \in \Omega_{\mathrm{reg}}$. 
  For any sequence of positive integers $(J_N)_{N \geq 0}$ with $J_N \nearrow \infty$, we have
  \begin{equation}\label{eq:asy-Phi-result}
    \lim_{N\to\infty} \Phi^\omega_{J_N,I(J_N),L(J_N),N} = 0,
  \end{equation}
  where $I(\cdot)$ is an explicit function defined in \eqref{eq:tilde_I_0} and $L(\cdot)$ is as in \eqref{eq:WellS-2}.
\end{theorem}

\begin{proof}
  We begin by considering $J \geq 2$, $L \geq L(J)$, $I,N\geq 1$ fixed. We will later allow them to depend on $N$ in a subsequent part of the proof. Note that the set of $\ell_*$ which are $(I,J,L,N)$-compatible is always non-empty, by an argument analogous to the one in Remark~\ref{rem:Further-remark-choices}~(1). Suppose $\ell_*\geq 0$ is $(I,J,L,N)$-compatible, $U_0\in \cU_{\ell_*,N}^\omega$, and $x\in A_N\cap \cS_\infty$. 

  We first note that for $\ell\leq \ell_0$, if $U_0$ has positive relative volume in $B(z,2^\ell)$, i.e. $\sigma_\ell^\omega (z)<1$ for some $z \in \cS_\infty$, then $B(z,2^\ell)\cap U_0\neq \varnothing$, so that $z$ must be in $B(0,b_N+2^\ell)$ as $U_0\subseteq B(0,b_N)$ (see \eqref{eq:cU}). The compatibility condition \eqref{eq:compatible-ii} yields 
  \begin{equation}\label{eq:u0-rel-positive-reason}
  B(0,b_N+2^\ell)\subseteq B(0,r_{\alpha(J), 2^{\ell_0 - (I+1)(J+1)L}} -4 \cdot 2^{\ell_0}+2^\ell)\subseteq B(0,r_{\alpha(J), 2^{\ell_0 - (I+1)(J+1)L}} -\tfrac{3}{2}2^{\ell_0}).
  \end{equation}
  This shows that for $\ell\leq \ell_0$, 
  \begin{equation}\label{eq:u0-rel-positive}
    \{\sigma_\ell^\omega<1\}\subseteq B(0,r_{\alpha(J), 2^{\ell_0 - (I+1)(J+1)L}} -\tfrac{3}{2}2^{\ell_0}).
  \end{equation}

  Now we introduce the concept of \emph{$I'$-families} and related definitions. For an integer $1\leq I'\leq I$, an $I'$-family consists of the initial point $x$, stopping times $(S_i)_{i=0}^{I'}$ with respect to the canonical filtration $(\mathcal{F}_t)_{t \geq 0}$, a random subset $\cL$ of $\cA$, and integer-valued random variables $(\widehat{\ell}_i)_{i=1}^{I'}$, satisfying
  \begin{subequations}\label{eq:I-family}
    \begin{empheq}[left=\empheqlbrace]{align}
      & 0 \leq S_0 \leq S_1 \leq \ldots \leq S_{I'} \text { are } P^\omega_x \text {-a.s.~finite, } \label{eq:I-family-i} \\
      & \cL \text { is an } \cF_{S_0} \text {-measurable finite subset of $\cA$, } |\cL|\geq I', \label{eq:I-family-ii}\\
      & \widehat{\ell}_i \text { are } \cF_{S_i} \text {-measurable for } 1 \leq i \leq I' \text {, distinct and } \cL \text {-valued, } \label{eq:I-family-iii}\\
      & P^\omega_x \text {-a.s., } \sigma^\omega_{\widehat{\ell}_i}\left(X_{S_i}\right)\in \left[\frac{1}{2}-\frac{1}{2^{\ell_{\min}}}, \frac{1}{2}+\frac{1}{2^{\ell_{\min}}}\right]\text { for } 1 \leq i \leq I' \label{eq:I-family-iv}.
  \end{empheq}
  \end{subequations}
  The ``canonical'' $I'$-family as defined in \cite[(2.12)]{nitzschner2017solidification} also exists in our case, if we replace the local density function $\widehat{\sigma}_\ell$ by $\sigma^\omega_\ell$, and the conditions $\widehat{\sigma}_\ell(X_{S_i})=\frac{1}{2}$ by $\sigma_{\ell}^\omega(X_{S_i})\in [\frac{1}{2}-\frac{1}{2^{\ell_{\min}}},\frac{1}{2}+\frac{1}{2^{\ell_{\min}}}]$. Given a general $I'$-family as above, we also define for $1\leq i \leq I'$ the stopping times 
  \begin{equation}
    T_i = \inf \{s\geq S_i \, : \, |X_s-X_{S_i}|_\infty \geq 2\cdot 2^{\widehat{\ell}_i}\},
  \end{equation}
  ``intermediate labels'' and ``labels'' 
  \begin{equation}
    \cL_{\text{int}} =\{\ell-jL \, : \, \ell\in \cL, 1\leq j \leq J\},\quad \cL_* = \cL \cup  \cL_{\text{int}}.
  \end{equation}
  Finally, we will need for $1\leq k\leq J$ the $(\cL_*,k)$-resonance set 
  \begin{equation}
    \Res_{(\cL_*,k)}^\omega = \left\{ x\in \bbZ^d \, : \, \sum_{\ell\in \cL_*} \IND_{\{\widetilde{\sigma}_\ell^\omega (x)\in [\widetilde{\alpha},1-\widetilde{\alpha}]\}}\geq k \right\}
  \end{equation}
  and the quantity (which also depends on $N,J,L,\omega$, but we omit the dependence since they are all fixed during the discussion of this quantity) for $1\leq k\leq J$, $I'\geq 1$
  \begin{align*}
      \Gamma_k(I') \stackrel{\mathrm{def}}{=} \sup P^\omega_x\left[\inf \left\{s \geq S_0 \, : \, X_s \in \operatorname{Res}_{(\cL_*, k)}\right\}>\max _{1 \leq i \leq I'} T_i\right],
  \end{align*}
  where the first supremum is taken over all $I'$-families and $\Gamma_k(I') \stackrel{\mathrm{def}}{=} 1$ whenever $I'\leq 0$. The following analogue of \cite[Lemma A.1]{CN2021disconnection} is the main ingredient of the proof.

  \begin{lemma}[fix $N,J,I\geq 1$, $L\geq L(J)$, $\ell_*\geq 0$ $(I,J,L,N)$-compatible]\label{lem:Gamma}
    It holds that 
    \begin{equation}\label{eq:Gamma-1}
      \Gamma_1(I')=0,\quad \mbox{for all $1\leq I'\leq I$}
    \end{equation}
    and for $1\leq k < J$, $1 \leq I' \leq I,N\geq 1$, $\Delta=\lfloor \sqrt{I'} \rfloor$,
    \begin{equation}\label{eq:Gamma-2}
      \Gamma_{k+1}(I') \leq (1-c_2(J))^{\sqrt{I'}-1} + (I')^{1+\frac{k-1}{2}} \Gamma_k(\Delta-k+1).
    \end{equation}
  \end{lemma}

  \begin{proof}
    We only sketch the proof. As \eqref{eq:Gamma-1} follows by exactly the same observation as in \cite[Lemma A.1]{CN2021disconnection}, we briefly explain the derivation of~\eqref{eq:Gamma-2}. It is important to note that 
    \begin{equation}
      X_{S_i}\in \{\sigma_{\widehat{\ell}_i}^\omega <1\}\subseteq B(0,r_{\alpha(J), 2^{\ell_0 - (I+1)(J+1)L}} -\tfrac{3}{2}2^{\ell_0})
    \end{equation}
    by \eqref{eq:I-family-iv} and \eqref{eq:u0-rel-positive}. In addition, the compatibility condition \eqref{eq:compatible-ii} implies \eqref{eq:WellS-1-ii} for all $\ell_0-(I+1)(J+1)L<\ell\leq \ell_0$ (with $\ell_{j+1}$ as $\ell$). Therefore, together with \eqref{eq:I-family-iv}, Proposition \ref{prop:positive-probability-adv} is applicable to $X_{S_i}$ for all $1\leq i \leq I'$. The remaining part of the proof of \eqref{eq:Gamma-2} is then similar to the proof of \cite[Lemma A.1]{CN2021disconnection}.
  \end{proof}

  We set
  \begin{equation}
    \widetilde{\Gamma}_k(I') = 
    \left\{
    \begin{aligned}
      &\sup_{\ell_*} \sup_{U_0\in \cU_{\ell_*,N}^\omega} \Gamma_k(I'), \quad &\mbox{for $1\leq k\leq J$ and $I'\geq 1$},\\
      &1, &\mbox{for $1\leq k\leq J$ and $I'\leq 0$},
    \end{aligned}
    \right.
  \end{equation}
  where in the first case, the supremum in $\ell_*$ is over all $(I,J,L,N)$-compatible $\ell_*\geq 0$. Using the ``canonical'' $I$-family, one has that 
  \begin{equation}\label{eq:asy-Phi-3}
    \Phi^\omega_{I,J,L,N} \leq \widetilde{\Gamma}_J(I).
  \end{equation}

  We want to show that the right-hand side of \eqref{eq:asy-Phi-3} is sufficiently small as long as $I$ is much larger than $J$ and it will be convenient to find an appropriate functional term characterizing the dependence of $I$ on $J$. We can first bound from above the right hand side of \eqref{eq:Gamma-2} by its supremum over $U_0\in \cU_{\ell_*,N}^\omega$ and $(I,J,L,N)$-compatible $\ell_*\geq 0$ and then take the same suprema on the left hand side to obtain a recursive relation for $\widetilde{\Gamma}_{\cdot}(\,\cdot\,)$:
  \begin{equation}\label{eq:Gamma_tilde}
    \widetilde{\Gamma}_{k+1}(I') \leq (1-c_2(J))^{\sqrt{I'}-1} + (I')^{1+\frac{k-1}{2}} \widetilde{\Gamma}_k(\Delta-k+1).
  \end{equation}

  \begin{lemma}[$J\geq 2$]\label{lem:tilde_Gamma}
    Let $2\leq k \leq J$. Define
    \begin{equation}\label{eq:I_0}
      I_0(\epsilon, k)= \left[c_3 \cdot \frac{k 2^{k}}{\epsilon} \log\left(\frac{k2^{k}}{\epsilon}\right)\right]^{2^{k-1}}, 
    \end{equation}
    where $c_3\stackrel{\mathrm{def}}{=} 2 \vee \big(-\log(1-c_2(1))\big)$. 
    For any $\epsilon \in (0,1)$, if $I\geq I_0(\epsilon, k)$, then 
    \begin{equation}
      I^{-1/2^{k-1}} \log \widetilde{\Gamma}_k (I) \leq \log(1-c_2(J)) + \epsilon.
    \end{equation}
    ($c_2(J)$ is the constant in Proposition \ref{prop:positive-probability-adv}).
  \end{lemma}

  \begin{proof}
    We prove the lemma by induction. For $k=2$, the claim follows from~\eqref{eq:Gamma_tilde} because $\widetilde{\Gamma}_1 (I)=0$. From now on we assume the claim holds for some $2\leq k\leq J$ and we want to show that for any $\epsilon>0$, if $I\geq I_0(\epsilon, k+1)$, then 
    \[
    I^{-1/2^{k}} \log \widetilde{\Gamma}_{k+1} (I) \leq \log(1-c_2(J)) + \epsilon.
    \]
    Due to \eqref{eq:Gamma_tilde}, we have 
    \begin{equation}\label{eq:induction-1}
      \begin{split}
        &I^{-1/2^{k}} \log \widetilde{\Gamma}_{k+1} (I)   \\
        &\begin{split}
          \leq \log 2\cdot I^{-1/2^{k}}
         + I^{-1/2^{k}} \max \Big\{&(\sqrt{I}-1)\log(1-c_2(J)),\\
         & \Big(1+\frac{k-1}{2}\Big)\log I + \log \widetilde{\Gamma}_k(\lfloor \sqrt{I} \rfloor -k+1)\Big\}.
      \end{split}
      \end{split}
    \end{equation}
    To bound the right hand side of \eqref{eq:induction-1} by $\log(1-c_2(J))+\epsilon$, it suffices to have the following four estimates when $I \geq I_0(\epsilon, k+1)$: 
    \begin{subequations}\label{eq:induction-2}
      \begin{empheq}[left=\empheqlbrace]{align}
        &\log 2\cdot I^{-1/2^{k}} \leq \frac{\epsilon}{4},\label{eq:induction-2-i}\\
        &I^{-1/2^{k}} (\sqrt{I}-1)\log(1-c_2(J)) \leq \log(1-c_2(J)) + \frac{3}{4}\epsilon,\label{eq:induction-2-ii}\\
        & I^{-1/2^{k}}\cdot \Big(1+\frac{k-1}{2}\Big)\log I \leq \frac{\epsilon}{4}, \label{eq:induction-2-iii}\\
        & I^{-1/2^{k}} \log \widetilde{\Gamma}_k(\lfloor \sqrt{I} \rfloor -k+1) \leq \log(1-c_2(J)) + \frac{\epsilon}{2}.\label{eq:induction-2-iv}
      \end{empheq}
    \end{subequations}
    {As $c_3\geq 2$, $\epsilon \in (0,1)$, and $(k+1)2^{k+1}\geq 2$, using the definition~\eqref{eq:I_0}, one has
    \[
    \log 2 \cdot I_0(\epsilon,k+1)^{-1/2^k} = \frac{\epsilon}{c_3 \cdot (k+1) 2^{k+1} } \cdot \frac{\log 2}{\log\left(\frac{(k+1)2^{k+1}}{\epsilon}\right)} \leq \frac{\epsilon}{4}.
    \]
    Hence \eqref{eq:induction-2-i} holds for $I \geq {I_0(\epsilon, k+1)}$. As for \eqref{eq:induction-2-ii}, we note that $\log(1-c_2(J)) < 0$ and $I^{-1/2^{k}} (\sqrt{I}-1)\geq 1$, since $I \geq I_0(\epsilon,2) \geq 256$, using $k \geq 2$.}
    We now focus on \eqref{eq:induction-2-iii} and \eqref{eq:induction-2-iv}. 

    Now note that \eqref{eq:induction-2-iii} holds if and only if 
    \begin{equation}\label{eq:2-c-1}
     \exp\left(-\frac{\epsilon/4}{(1+\frac{k-1}{2})2^k} I^{1/2^k}\right) \cdot \left(-\frac{\epsilon/4}{(1+\frac{k-1}{2})2^k} I^{1/2^k}\right) \geq -\frac{\epsilon/4}{(1+\frac{k-1}{2})2^k}. 
    \end{equation}
    If 
    \[
    -\frac{\epsilon/4}{(1+\frac{k-1}{2})2^k} \leq - \frac{1}{e},
    \]
    then \eqref{eq:2-c-1} holds for all $I\geq 1$ as $-xe^{-x}\geq -e^{-1}$ for $x\geq 0$. We can therefore assume without loss of generality that  
    \[
    - \frac{1}{e}<-\frac{\epsilon/4}{(1+\frac{k-1}{2})2^k} < 0,
    \]
    and thus \eqref{eq:2-c-1} holds if 
    \begin{equation}\label{eq:2-c-2}
      I \geq \left[- \bigg(\frac{\epsilon/4}{(1+\frac{k-1}{2})2^k}\bigg)^{-1} W_{-1}\Big(-\frac{\epsilon/4}{(1+\frac{k-1}{2})2^k}\Big)\right]^{2^k},
    \end{equation}
    where $W_{-1}(\cdot)$ is a real branch of the Lambert $W$-function (see, e.g., \cite{chatzigeorgiou2013bounds} for its definition). It is proved in~\cite[Theorem 1]{chatzigeorgiou2013bounds} that $W_{-1}(-e^{-u-1})>-1-\sqrt{2u}-u$ for $u>0$, so
    \begin{equation}
    \label{eq:Lambert-ineq}
      W_{-1}(-e^{-u-1}) > {2(-u-1),}
    \end{equation}
    as $u-\sqrt{2u}+1>0$ for $u>0$. 
    Applying the inequality~\eqref{eq:Lambert-ineq} with $-u= \log \frac{\epsilon/4}{(1+\frac{k-1}{2})2^k}+1$, we obtain 
    \begin{align*}
      &\left[- \bigg(\frac{\epsilon/4}{(1+\frac{k-1}{2})2^k}\bigg)^{-1} W_{-1}\bigg(-\frac{\epsilon/4}{(1+\frac{k-1}{2})2^k}\bigg)\right]^{2^k}\leq \left[ {2} \left(\frac{(1+\frac{k-1}{2})2^k}{\epsilon/4}\right) \log \frac{(1+\frac{k-1}{2})2^k}{\epsilon/4}\right]^{2^k} \\
      & =  \left[ {2}\left(\frac{(k+1)2^{k+1}}{\epsilon}\right) \log \frac{(k+1)2^{k+1}}{\epsilon}\right]^{2^k} \stackrel{{c_3 \geq 2}}{\leq} I_0(\epsilon, k+1)
      .
    \end{align*}
    Hence \eqref{eq:2-c-2} holds when $I \geq I_0(\epsilon, k+1)$. Combined with \eqref{eq:2-c-1}, \eqref{eq:induction-2-iii} holds when $I \geq I_0(\epsilon, k+1)$.

    To show that \eqref{eq:induction-2-iv} holds when $I \geq I_0(\epsilon,k+1)$, it suffices to show that 
    \begin{equation}\label{eq:iv-sufficient}
      I_0(\epsilon,k+1) \geq (I_0(\epsilon/2, k)+k)^2.
    \end{equation}
    Indeed, suppose~\eqref{eq:iv-sufficient} holds, then $I \geq I_0(\epsilon,k+1)$ implies $I \geq (I_0(\epsilon/2, k)+k)^2$ so that $\lfloor \sqrt{I} \rfloor -k+1 \geq I_0(\epsilon/2,k)$. By the induction hypothesis, 
    \[
      (\lfloor \sqrt{I} \rfloor -k+1)^{-1/2^{k-1}} \log \widetilde{\Gamma}_{k}(\lfloor \sqrt{I} \rfloor - k+1)  \leq \log(1-c_2(J)) + \epsilon/2.
    \]
    Then \eqref{eq:induction-2-iv} holds because
    \[
    I^{-1/2^k}/ (\lfloor \sqrt{I} \rfloor -k+1)^{-1/2^{k-1}} = (\lfloor \sqrt{I} \rfloor -k+1)^{1/2^{k-1}}/ I^{1/2^k} \leq 1.
    \] 
    We now show \eqref{eq:iv-sufficient}. To that end, observe that 
    \begin{align*}
      I_0(\epsilon, k+1) &= \left[c_3 \cdot \frac{(k+1)2^{k+1}}{\epsilon} \log \left(\frac{(k+1)2^{k+1}}{\epsilon}\right)\right]^{2^k}\\
      &= \left[c_3 \cdot \frac{k2^{k+1}}{\epsilon}\log \left(\frac{(k+1)2^{k+1}}{\epsilon}\right) + c_3\cdot \frac{2^{k+1}}{\epsilon}\log \left(\frac{(k+1)2^{k+1}}{\epsilon}\right)\right]^{2^k}\\
      & \geq \left\{\left[c_3 \cdot \frac{k2^{k+1}}{\epsilon}\log \left(\frac{k2^{k+1}}{\epsilon}\right) + k \right]^{2^{k-1}}\right\}^2 \geq (I_0(\epsilon/2, k)+k)^2.
    \end{align*}
  This completes the induction.
  \end{proof}

  For the following result, we let $I$ depend on $J$ in the following explicit way:  
  \begin{equation}\label{eq:tilde_I_0}
    I(J) \stackrel{\mathrm{def}}{=} \left\lceil I_0\left(-\frac{1}{2} \log(1-c_2(J)), J \right)\right\rceil \vee \left\lceil\frac{J}{c_2(J)}\right\rceil^{2^{J-1}},
  \end{equation}
  where $I_0$ is defined in \eqref{eq:I_0}.

  \begin{corollary}\label{coro:tilde-Gamma}
    $\widetilde{\Gamma}_J(I(J))\to 0$ as $J \to \infty$.
  \end{corollary}

  \begin{proof}
    Let $\epsilon=-\frac{1}{2}\log(1-c_2(J))$ ($<1$ for $J$ large enough). On one hand, by Lemma \ref{lem:tilde_Gamma}, we have 
    \begin{equation}\label{eq:Gamma_JI}
      \log \widetilde{\Gamma}_J (I) \leq \frac{1}{2} \log(1-c_2(J)) \cdot I^{1/2^{J-1}}
    \end{equation}
    for $I\geq I_0(-\frac{1}{2}\log(1-c_2(J)), J)$. On the other hand, we have
    \begin{equation}\label{eq:Ic_2}
      I^{1/2^{J-1}} c_2(J) \geq J
    \end{equation}
    provided that $I\geq (\frac{J}{c_2(J)})^{2^{J-1}}$. Combining the two cases and recalling the definition of $I(J)$ in \eqref{eq:tilde_I_0}, we have
    \[
    \lim_{J\to\infty} \log (1-c_2(J)) I(J)^{-1/2^{J-1}} = - \lim_{J\to\infty} c_2(J) I(J)^{-1/2^{J-1}} \stackrel{\eqref{eq:Ic_2}}{\leq} -\lim_{J\to\infty} J = -\infty.
    \]
    The result then follows by \eqref{eq:Gamma_JI}.
  \end{proof}

With this, we conclude the proof of~\eqref{eq:asy-Phi-result} as follows. Let $(J_N)_{N \geq 0}$ be a positive sequence with $J_N \nearrow \infty$. 
Inserting the choices $I= I(J_N),J= J_N, L= L(J_N)$ into \eqref{eq:asy-Phi-3} and letting $N$ tend to infinity, we observe
  \[
  \limsup_{N\to\infty} \Phi^\omega_{I(J_N), J_N, L(J_N), N} \leq \lim_{N\to\infty} \widetilde{\Gamma}_{J_N}(I(J_N)) \stackrel{\text{Cor.\ref{coro:tilde-Gamma}}}{=} 0.
  \]
  This proves \eqref{eq:asy-Phi-result} and completes the proof of the theorem.
\end{proof}

\subsection{Proof of Theorem \ref{thm:solidification}}

We start by introducing an analogue of \cite[Lemma 4.9]{CN2021disconnection}.
\begin{lemma}[$d\geq 3$, \refA]\label{lem:hit-sigma}
  Let $\omega\in\Omega_{\mathrm{reg}}^{1/3, 1/9}$ and $U_0 \in \mathcal{U}_{\ell_\ast,N}^\omega$.
  Let $\epsilon\in \bbN$, $\chi\in(0,1)$, and $\Sigma\in \cS_{U_0,\epsilon,\chi}^\omega$. There exists a sufficiently large constant $c_4(\omega)$ such that if $\ell\geq c_4(\omega)$ with $\epsilon \leq \frac{1}{4}\cdot 2^{\ell}$, $x_0\in \cS_\infty\cap B(0,r_{\frac{1}{3}, 4\cdot 2^\ell})$ with $\widetilde{\sigma}_\ell^\omega(x_0)\in [\widetilde{\alpha},1-\widetilde{\alpha}]$, and $y\in \cS_\infty$ with $|y-x_0|\leq \frac{1}{4}\cdot 2^\ell$, then 
  \begin{equation}
    P_y^\omega [H_\Sigma < T_{B(x_0,5\cdot 2^\ell)}] \geq c_5(\chi).
  \end{equation}
\end{lemma}

\begin{proof}
  Define $\widetilde{U}_0\stackrel{\mathrm{def}}{=} U_0\cap B(x_0, 4\cdot 2^\ell)$ and $\widetilde{U}_1\stackrel{\mathrm{def}}{=} U_1\cap B(x_0, 4\cdot 2^\ell)$, and note that 
  \begin{equation}\label{eq:hit-sigma-1}
    \begin{aligned}
      P_y^\omega &[\mbox{$X_{\frac{81}{4}\cdot 4^\ell}\in U_0$, $X_{\frac{81}{2}\cdot 4^\ell}\in U_1$ and $\frac{81}{2}\cdot 4^\ell < T_{B(x_0,\frac{9}{2}\cdot 2^\ell)}$}]\\
    &\geq \sum_{x\in \widetilde{U}_0} \sum_{z\in \widetilde{U}_1} q^\omega_{\frac{81}{4}\cdot 4^\ell, B(x_0,\frac{9}{2}\cdot 2^\ell)} (y,x) q^\omega_{\frac{81}{4}\cdot 4^\ell, B(x_0,\frac{9}{2}\cdot 2^\ell)} (x,z) \geq C, 
    \end{aligned}
  \end{equation}
  where we used that both $\widetilde{U}_0$ and $\widetilde{U}_1$ have at least cardinality $\frac{1}{3}\widetilde{\alpha}\eta(u)(2^{\ell+3}+1)^d$, combined with the killed heat kernel bound \eqref{eq:killed-heat-kernel} (and setting $\vartheta=\frac{1}{9}$, $R=\frac{9}{2}\cdot 2^\ell$ so that $(1-\vartheta)R=4\cdot 2^\ell$, and $t=R^2$), and $c_4(\omega)$ large enough. On the event under the probability on the left-hand side of \eqref{eq:hit-sigma-1}, $X_{\cdot}$ enters both $U_0$ and $U_1$ before exiting $B(x_0,\frac{9}{2}\cdot 2^\ell)\cap \cS_\infty$, which means
  \begin{equation}
    P_y^\omega [H_S < T_{B(x_0,\frac{9}{2}\cdot 2^\ell)}]\geq C,
  \end{equation}
  where we recall that $S = \partial_{\mathcal{S}_\infty} U_0$.  Then, upon using the strong Markov property, we obtain 
  \begin{equation}
    \begin{aligned}
 P_y^\omega [H_\Sigma < T_{B(x_0,5\cdot 2^\ell)}] & \geq      P_y^\omega [H_\Sigma \circ \theta_S + H_S < T_{B(x_0,5\cdot 2^\ell)}]\\
      &\geq E_y^\omega \left[H_S < T_{B(x_0,\frac{9}{2}\cdot 2^\ell)}, P_{X_{H_S}}^\omega [H_\Sigma < \tau_\epsilon]\right] \geq C\cdot \chi = C(\chi),
    \end{aligned}
  \end{equation}
  since $\epsilon \leq \frac{1}{4} \cdot 2^\ell$.
\end{proof}

With these preparations in place, we now turn to the proof of our main result.

\begin{proof}[Proof of Theorem \ref{thm:solidification}]
Let $(a_N)_{N\geq 0}$, $(b_N)_{N\geq 0}$ be sequences fulfilling \eqref{eq:a_N-bound} and $(c_N)_{N \geq 0}$ decreasing to zero, as in the statement of the theorem. Then $a_N= s(b_N) \cdot u_N$ for an increasing sequence $(u_N)_{N \geq 0}$  with $u_N \nearrow \infty$, where 
\begin{equation}\label{eq:s_n}
  s(n) = \exp((\log n)^{\frac{1}{1+\DeltaS/2}}).
\end{equation}

Let $\omega \in \Omega_{\mathrm{reg}} \cap \Omega_{\mathrm{reg}}^{1/3,1/9}$. We claim that we can choose a sequence of positive integers $(J_N^\omega)_{N \geq 0}$, which depends on the realization $\omega$, with $J_N^\omega \nearrow \infty$ and a natural number $N_0(\omega)$ large enough such that for each $N\geq N_0(\omega)$, we have
\begin{subequations}\label{eq:N0-1}
  \begin{empheq}[left=\empheqlbrace]{align}
& 2^{-(I_N^\omega+2)(J_N^\omega+1)L_N^\omega} \geq 1/u_N, \label{eq:N0-1-i}\\
& \frac{1}{4} 2^{-(I_N^\omega+1)(J_N^\omega+1)L_N^\omega} \geq c_N,\label{eq:N0-1-ii}\\
& \text{if $a_N \leq 2^{\ell_*}\leq b_N$, then} \notag\\
& \qquad\text{$\ell_*$ is $(I_N^\omega,J_N^\omega,L_N^\omega,N)$-compatible, and}\label{eq:N0-1-iii-1}\\
 & \qquad \ell_*-(I_N^\omega+1)(J_N^\omega+1)L_N^\omega\geq c_4(\omega) \label{eq:N0-1-iii-2},
\end{empheq}
\end{subequations}
where for simplicity, we write $I_N^\omega \stackrel{\mathrm{def}}{=} I(J_N^\omega)$ and $L_N^\omega \stackrel{\mathrm{def}}{=} L(J_N^\omega)$, with $I(\,\cdot\,)$ and $L(\,\cdot\,)$ defined in \eqref{eq:tilde_I_0} and \eqref{eq:WellS-2}, respectively, and we recall that, compatibility is defined in \eqref{eq:compatible}, and $c_4(\omega)$ is the constant defined in Lemma~\ref{lem:hit-sigma}. 
\medskip

We first show how the proof is finished assuming the existence of $(J_N^\omega)_{N \geq 0}$. To that end, consider $\epsilon/2^{\ell_*}\leq c_N $, $2^{\ell_\ast} \in [a_N,b_N]$, $\Sigma \in \cS_{U_0, \epsilon,\chi}^\omega$ with $U_0\in \cU_{\ell_*,N}^\omega$ for $N\geq N_0$, and let $x_0\in \Res = \Res^\omega(U_0,I^\omega_N,J^\omega_N,
L^\omega_N,\ell_\ast)$. Then 
\begin{equation}\label{eq:epsilon-cond}
  \epsilon \leq c_N 2^{\ell_*} \stackrel{\eqref{eq:N0-1-ii}}{\leq} \frac{1}{4} 2^{\ell_*-(I_N^\omega+1)(J_N^\omega+1)L_N^\omega} \leq \frac{1}{4} 2^{\min \cA_*}.
\end{equation}
We can apply the strong Markov property at successive exit times of the sup-norm ball $B(x_0,5\cdot 2^\ell)$, $\ell\in \cA_*$ and find (because $5\cdot 2^{\ell'}\leq \frac{1}{4} 2^\ell$ for $\ell'<\ell\in \cA_*$) upon repeated use of Lemma \ref{lem:hit-sigma} that
\begin{equation}
  P_{x_0}^\omega [H_\Sigma > T_{B(x_0,5\cdot 2^{\max \cA_*})}] 
  \leq (1-c_5(\chi))^{\sum_{\ell\in \cA_*} \IND_{\{\widetilde{\sigma}^\omega_\ell (x_0)\in [\widetilde{\alpha},1-\widetilde{\alpha}]\}} } \leq (1-c_5(\chi))^{J_N^\omega},
\end{equation}
where we used that $x_0\in \Res$ in the last step. Indeed, for $\ell'<\ell\in \cA_*$, 
\begin{align*}
  P_{x_0}^\omega &[H_\Sigma > T_{B(x_0,5\cdot 2^{\ell})}] \\
  & \leq E_{x_0}^\omega \Big[H_\Sigma > T_{B(x_0,5\cdot 2^{\ell'})}, P_{X_{T_{B(x_0,5\cdot 2^{\ell'})}}}^\omega [H_\Sigma > T_{B(x_0,5\cdot 2^{\ell})}] \Big]\\
  & \leq P_{x_0}^\omega [H_\Sigma > T_{B(x_0,5\cdot 2^{\ell'})}] {(1-c_5(\chi))}^{\IND_{\{\widetilde{\sigma}_\ell^\omega(x_0)\in [\widetilde{\alpha},1-\widetilde{\alpha}]\}}}.
\end{align*}
In the last step, Lemma~\ref{lem:hit-sigma} is applicable for any $\ell \in \cA_\ast$ with $\widetilde{\sigma}^\omega_\ell(x_0) \in [\widetilde{\alpha},1-\widetilde{\alpha}]$ by \eqref{eq:epsilon-cond},~\eqref{eq:N0-1-iii-2}, and 
\begin{equation}
\begin{split}
  x_0\in \{\widetilde{\sigma}_\ell^\omega \in [\widetilde{\alpha},1-\widetilde{\alpha}]\} & \subseteq \{\widetilde{\sigma}_\ell^\omega<1\} \subseteq B(0,b_N +4 \cdot 2^{\ell}) \\
  & \hspace{-0.32cm} \stackrel{\eqref{eq:N0-1-iii-1}}{\subseteq} B(0,r_{\alpha(J^\omega_N),2^{\ell_0 - (I^\omega_N+1)(J^\omega_N+1)L^\omega_N} } {- 4 \cdot 2^{\ell_0} + 4 \cdot 2^{\ell}}) \\
  & \subseteq B(0,r_{\alpha(J^\omega_N),2^\ell }) \subseteq B(0,r_{\frac{1}{3},4 \cdot 2^\ell }), 
  \end{split}
\end{equation}
where the second inclusion follows by a similar argument as in~\eqref{eq:u0-rel-positive-reason}, and in the last line we used that $\ell \geq \ell_0 - (I^\omega_N+1)(J^\omega_N+1)L^\omega_N$, the fact that $r_{\alpha, N}$ is increasing in $\alpha$ (see~\eqref{eq:r_N}), and the observation that $\alpha(J) \leq \alpha(1) < \frac{1}{3}$ (see (i) in the proof of Proposition~\ref{prop:positive-probability-adv}).

Therefore, for $x\in A_N\cap\cS_\infty$ and $N\geq N_0(\omega)$, 
\begin{align*}
  P_x^\omega [H_\Sigma=\infty] &\leq P_x^\omega [H_{\Res} =\infty] + E_x^\omega [H_{\Res}<\infty, P_{X_{\Res}}^\omega [H_\Sigma=\infty]]\\
  & \leq P_x^\omega [H_{\Res} =\infty] + (1-c_5(\chi))^{J_N^\omega} \\
  & \stackrel{\eqref{eq:Phi-Def},\eqref{eq:N0-1-iii-1}}{\leq} \Phi_{J_N^\omega,I_N^\omega,L_N^\omega,N}^\omega + (1-c_5(\chi))^{J_N^\omega}.
\end{align*}
We then take supremum over $x\in A_N\cap\cS_\infty, \Sigma\in \cS_{U_0,\epsilon,\chi}^\omega$, $U_0\in \cU_{\ell_*,N}^\omega$, $\epsilon/2^{\ell_*}\leq c_N$, $a_N\leq 2^{\ell_*}\leq b_N$, 
and let $N\to\infty$ to obtain 
\begin{equation}
  \begin{split}
    \lim_{N\to\infty} \sup_{\epsilon/2^{\ell_*}\leq c_N,a_N\leq 2^{\ell_*}\leq b_N} & \sup_{U_0\in \cU_{\ell_*,N}^\omega} \sup_{\Sigma\in \cS_{U_0,\epsilon,\chi}^\omega} \sup_{x\in A_N\cap\cS_\infty} P_x^\omega [H_\Sigma=\infty]\\
    & \leq \lim_{N\to\infty}  \Phi_{J_N^\omega,I_N^\omega,L_N^\omega,N}^\omega + \lim_{N\to\infty}(1-c_5(\chi))^{J_N^\omega} = 0,
  \end{split}
\end{equation}
where we have used the result of Theorem \ref{thm:asy-Phi} in the last step. Finally, \eqref{eq:solid-2} follows in the same way as \cite[Lemma 2.1]{chiarini2020GFF}.

\medskip 
We are left with constructing $(J_N^\omega)_{N \geq 0}$ fulfilling~\eqref{eq:N0-1-i}--\eqref{eq:N0-1-iii-2}. 
The following elementary lemma will be helpful for that purpose. 

\begin{lemma}\label{lem:sequence}
  Let $(x_n)_{n \geq 1}$ be a sequence of real numbers with $x_n \nearrow \infty$, and $f : \mathbb{Z}_+ \to \mathbb{R}$. Then there exists a sequence of positive integers $(y_n)_{n \geq 1}$ with $y_n \nearrow \infty$ and an $n_0$ such that $f(y_n)\leq x_n$ for $n\geq n_0$. Moreover, if $(y'_n)_{n \geq 1}$  is another sequence of positive integers with $y_n'\leq y_n$, then $f(y_n') \leq x_n$ for $n\geq n_0$.
\end{lemma}

\begin{proof}
   Let $\tilde{f}(m) = \max_{1\leq j \leq m} f(j)$. We can assume without loss of generality that $\tilde{f}(m)\to \infty$ as $m \to \infty$ since otherwise the claim is obvious. The first part of the lemma follows by considering $y_n \stackrel{\mathrm{def}}{=} \max \{m \in \mathbb{Z}_+ :\tilde f(m)\leq x_n\}$
   (with the convention $\max \varnothing = 1$, but note that since $x_n \nearrow \infty$, we see that there exists $n_0 \in \mathbb{Z}_+$ such that $\tilde{f}(1) \leq x_n$ for $n \geq n_0$). 
    The second claim follows from $f(y'_n) \leq \tilde{f}(y'_n) \leq \tilde{f}(y_n)$, for $n \geq n_0$, since $\tilde{f}$ is increasing.
\end{proof}

In what follows, we define $J_N^{j,\omega}$ with $1 \leq j \leq 4$ and set $J_N^\omega \stackrel{\mathrm{def}}{=} J_N^{1,\omega} \wedge J_N^{2,\omega} \wedge J_N^{3,\omega} \wedge J_N^{4,\omega}$, where each of the sequences $(J_N^{j,\omega})_{N \geq 0}$ is chosen to fulfill one of the conditions in~\eqref{eq:N0-1-i}--\eqref{eq:N0-1-iii-2}. 
By the second part of Lemma~\ref{lem:sequence}, the sequence $(J_N^\omega)_{N \geq 0}$ then satisfies~\eqref{eq:N0-1-i}--\eqref{eq:N0-1-iii-2}.
\medskip

For \eqref{eq:N0-1-i}, note that $2^{(I(\cdot)+2)(\cdot+1)L(\cdot)}$ is a function defined on positive integers. By Lemma \ref{lem:sequence}, we can choose $(J_N^{1,\omega})_{N \geq 0}$ as a sequence of integers such that $J_N^{1,\omega} \nearrow \infty$ and \eqref{eq:N0-1-i} holds for $(J_N^{1,\omega})_{N \geq 0}$. We can argue similarly for \eqref{eq:N0-1-ii} and \eqref{eq:N0-1-iii-2}, where in \eqref{eq:N0-1-iii-2} we require
\begin{equation}
(I(J_N^{4,\omega})+1)(J_N^{4,\omega}+1)L(J_N^{4,\omega}) \leq \log_2 a_N - c_4(\omega).
\end{equation}
It remains to construct a sequence $(J_N^{3,\omega})_{N \geq 0}$ fulfilling~\eqref{eq:N0-1-iii-1} (see again~\eqref{eq:compatible} for the definition of compatibility). For the first compatibility condition \eqref{eq:compatible-i}, we first choose a sequence $(\widetilde{J}^{3,\omega}_N)_{N \geq 0}$, similarly as before by Lemma \ref{lem:sequence}, with 
\begin{equation}
\ell_{\min}(\omega, \delta(\widetilde{J}_N^{3,\omega}), \alpha(\widetilde{J}_N^{3,\omega})) + (I(\widetilde{J}_N^{3,\omega})+1)(\widetilde{J}_N^{3,\omega}+1)L(\widetilde{J}_N^{3,\omega}) < \log_2 a_N,
\end{equation}
(see \eqref{eq:l_min} for the definition of $\ell_{\min}$). 
To fulfill the second compatibility condition \eqref{eq:compatible-ii}, we choose another sequence $(\widehat{J}_N^{3,\omega})_{N \geq 0}$ with $\widehat{J}_N^{3,\omega} \leq J^{1,\omega}_N$ such that (recall that $\kappa_{\text{reg}}(\alpha)$ enters the definition of $r_{\alpha, R}$; see \eqref{eq:r_N})
\begin{equation}\label{eq:alpha_N}
  1/ \kappa_{\text{reg}}(\alpha(\widehat{J}_N^{3,\omega})) \leq \frac{1}{2}\cdot \log( b_N)^{\frac{\DeltaS/2}{1+\DeltaS/2}},
\end{equation}
and assume that we have chosen $N_0'(\omega)$ large enough such that for $N\geq N_0'(\omega)$, \eqref{eq:N0-1-i} holds. Then for $N\geq N_0'(\omega)$, if $a_N \leq 2^{\ell_*}\leq b_N$, we have (with $\widehat{I}^{3,\omega}_N = I(\widehat{J}_N^{3,\omega})$ and $\widehat{L}^{3,\omega}_N = L(\widehat{J}_N^{3,\omega})$)
\begin{equation}\label{eq:min-scale-est-1}
  \begin{aligned}
    2^{\ell_0-(\widehat{I}_N^{3,\omega}+1)(\widehat{J}_N^{3,\omega}+1)\widehat{L}_N^{3,\omega}} &\geq 2^{\ell_*-(\widehat{I}_N^{3,\omega}+2)(\widehat{J}_N^{3,\omega}+1)\widehat{L}_N^{3,\omega}} \\
    &\geq a_N \cdot 2^{-(\widehat{I}_N^{3,\omega}+2)(\widehat{J}_N^{3,\omega}+1)\widehat{L}_N^{3,\omega}} \stackrel{\text{\eqref{eq:N0-1-i}}}{\geq} \frac{a_N}{u_N} = s(b_N)
  \end{aligned}
\end{equation}
(with $s(\cdot)$ defined in~\eqref{eq:s_n}). This implies 
\begin{align*}
  r_{\alpha(\widehat{J}_N^{3,\omega}), 2^{\ell_0-(\widehat{I}_N^{3,\omega}+1)(\widehat{J}_N^{3,\omega}+1)\widehat{L}_N^{3,\omega}}} &\geq  
  \exp(\kappa_{\text{reg}}(\alpha(\widehat{J}_N^{3,\omega})) (\log s( b_N))^{1+\DeltaS})\\
  & \stackrel{\eqref{eq:alpha_N},\eqref{eq:s_n}}{\geq} 
  \exp\left( 2 \log( b_N)^{-\frac{\DeltaS/2}{1+\DeltaS/2}} \cdot\log( b_N)^{\frac{1+\DeltaS}{1+\DeltaS/2}} \right)= b_N^2.
\end{align*}
Hence, 
\begin{equation}
  \frac{1}{b_N} \Big(r_{\alpha(\widehat{J}_N^{3,\omega}), 2^{\ell_0-(\widehat{I}_N^{3,\omega}+1)(\widehat{J}_N^{3,\omega}+1)\widehat{L}_N^{3,\omega}}} - 4\cdot 2^{\ell_0} \Big) \geq \frac{1}{b_N}\Big(b_N^2 - 4 b_N\Big)
\end{equation}
for $N\geq N_0'(\omega)$, which tends to infinity as $N\to\infty$. In other words, for $N$ large enough (depending on $\omega$), we have
\begin{equation}
r_{\alpha(\widehat{J}_N^{3,\omega}), 2^{\ell_0-(\widehat{I}_N^{3,\omega}+1)(\widehat{J}_N^{3,\omega}+1)\widehat{L}_N^{3,\omega}}} - 4\cdot 2^{\ell_0} \geq b_N,
\end{equation}
 so~\eqref{eq:compatible-ii} holds. We finally set $J^{3,\omega}_N = \widetilde{J}^{3,\omega}_N \wedge  \widehat{J}^{3,\omega}_N$, and $(J_N^{3,\omega})_{N \geq 0}$ satisfies~\eqref{eq:N0-1-iii-1}. 
\end{proof}
We conclude this section with a discussion of generalizations of our main result and pertinent applications.

\begin{remark}
\label{rem:Applications-generalizations}
\begin{enumerate}[label=(\arabic*),leftmargin=*]
\item \textit{(Bond percolation models)} We have chosen in the~\refA a set-up of site percolation models fulfilling \refPone--\refPtwo, \refD, and \refSone--\refStwo to include a broad class of examples of percolation clusters with strong, algebraically decaying correlations, see Remark~\ref{rmk:examples}. One can alternatively consider bond percolation models by setting $\Omega_{\mathrm{bond}} = \{0,1\}^{\mathbb{E}^d}$, $\mathcal{F}_{\mathrm{bond}} = \sigma(\omega_b \mapsto \omega_b(e) \, : \, e \in \mathbb{E}^d )$ and considering a class of probability measures $\{\mathbb{P}^u \, : \, u \in (a,b) \}$ with real numbers $a < b$ on $(\Omega_{\mathrm{bond}},\mathcal{F}_{\mathrm{bond}})$. One then studies the subgraph $(\mathbb{Z}^d, \mathcal{E})$ with $\mathcal{E} = \{e \in \mathbb{E}^d \, : \, \omega_b(e) = 1 \}$ and defines 
\begin{equation}
\mathcal{S}_\infty^{\mathrm{bond}} = \text{ the set of vertices in infinite connected components of $(\mathbb{Z}^d, \mathcal{E})$}.
\end{equation}
A particularly important example is supercritical Bernoulli bond percolation on $(\mathbb{Z}^d,\mathbb{E}^d)$, $d \geq 3$, where $\mathbb{P}^u$ is a product measure of a Bernoulli distribution with success parameter $u$ on the edges, and $u > p_c^{\mathrm{bond}}$. A simple inspection of the proof of Theorem~\ref{thm:solidification} shows that our results remain valid if the quenched heat kernel bounds~\eqref{eq:HK} (together with an integrability condition as in~\eqref{eq:HK-time}) and the large-deviation estimates for the volume~\eqref{eq:QCV} hold for the corresponding infinite connected component $\mathcal{S}_\infty^{\mathrm{bond}}$. The analogues of~\eqref{eq:HK} and~\eqref{eq:HK-time} are the main results in~\cite{barlow2004RWpercolation} (see {Theorem 1} and (0.5), p.~3026 therein), and~\eqref{eq:QCV} is found for instance in~\cite[Theorem 1.2]{pisztora1996large} (or~\cite[Proposition 11]{dario-gu2021percolation} for an even more quantitative result). We can therefore obtain the equivalent of~\eqref{eq:solid-1} in this situation, namely that for $\chi \in (0,1)$, $(c_N)_{N \geq 0}$ a sequence of positive real numbers decreasing to $0$, and $\mathbb{P}^u$-almost every $\omega_b \in \{0,1\}^{\mathbb{E}^d}$, one has
  \begin{equation}\label{eq:solid-bond}
    \lim_{N\to\infty} \sup_{\epsilon/2^{\ell_*}\leq c_N,a_N\leq 2^{\ell_*}\leq b_N} \sup_{U_0\in \cU_{\ell_*,N}^\omega} \sup_{\Sigma\in \cS_{U_0,\epsilon,\chi}^\omega} \sup_{x\in A_N \cap \cS_\infty^{\mathrm{bond}}} P_x^{\omega_b} [H_\Sigma=\infty] =0,
  \end{equation}
  with $(a_N)_{N \geq 0}$ and $(b_N)_{N \geq 0}$ as in~\eqref{eq:a_N-bound} with \textit{any} choice of $\DeltaS > 0$. In fact, the integrability for the random variables $\mathcal{T}(z)$ is stretched exponential in~\cite{barlow2004RWpercolation} (as compared to the superpolynomial rate in~\eqref{eq:HK-time}), which should also straightforwardly lead to an improvement in condition in~\eqref{eq:a_N-bound}, but we do not pursue this extension since~\eqref{eq:a_N-bound} seems strong enough for most applications we have in mind.
\item \textit{(Stability under perturbations with uniformly elliptic edge weights)} As we remarked earlier, in~\cite{CN2021disconnection} solidification estimates were obtained for random walks on $(\mathbb{Z}^d,\mathbb{E}^d,\lambda)$, where $\lambda \in [\lambda_{\min},1]^{\mathbb{E}^d}$ are uniformly elliptic edge weights, and $\lambda_{\min} \in (0,1)$. As we briefly explain, our results can be adapted in a direct way to include uniformly elliptic (potentially non-random) edge weights. 

To that end, let $\lambda_{\min} \in (0,1)$ and $\lambda \in [\lambda_{\min},1]^{\mathbb{E}^d}$ be fixed. Under either our standing~\refA for site percolation models or for the supercritical Bernoulli bond percolation of the previous item, we define the induced \textit{weighted percolation configuration} $(\mathcal{S}_\infty,E(\mathcal{S}_\infty),(\lambda_e)_{e  \in E(\mathcal{S}_\infty)})$ or $(\mathcal{S}^{\mathrm{bond}}_\infty,\mathcal{E},(\lambda_e)_{e \in \mathcal{E}})$, governed by $\mathbb{P}^u$, $u \in (a,b)$ (where in the Bernoulli bond percolation case, $(a,b) = (p_c^{\mathrm{bond}},1)$). We then consider the simple random walk on $(\mathcal{S}_\infty,E(\mathcal{S}_\infty),(\lambda_e)_{e  \in E(\mathcal{S}_\infty)})$ or $(\mathcal{S}^{\mathrm{bond}}_\infty,\mathcal{E},(\lambda_e)_{e \in \mathcal{E}})$, governed by $(P^{\omega,\lambda}_x)_{x \in \mathcal{S}_\infty}$ with generator
\begin{equation}
\mathcal{L}_{\mathcal{S}_\infty}^\lambda f(x) = \frac{1}{\mu_x^{\lambda}} \sum_{x \in \mathcal{S}_\infty} \mu_{\{x,y\}}^\lambda (f(y)-f(x)),
\end{equation} 
with 
 \begin{equation}
  \label{eq:Weights-given-mu-lambda}
      \mu_{\{x,y\}}^\lambda = \begin{cases}
          \lambda_{\{x,y\}}, & \{x,y\}\in E(\cS_\infty),\\
          0, & \text{otherwise},
      \end{cases}
      \qquad
      \mu_{x}^\lambda = \sum_{y \in \mathcal{S}_\infty} \mu_{\{x,y\}}^\lambda,  \ x \in \cS_\infty,
  \end{equation}
  in the site percolation case, and similarly for the case of Bernoulli bond percolation. We also define the heat kernel for the random walk, $q^{\omega,\lambda}_t$, and heat kernel for the walk killed upon leaving $U$, $q^{\omega,\lambda}_{t,U}$, as in~\eqref{eq:HeatKernelDef} and~\eqref{eq:KilledHeatKernel} with $\mu^\lambda$ replacing $\mu$ and $P^{\omega,\lambda}_x$ replacing $P^\omega_x$ (again, similarly in the bond percolation case). By uniform ellipticity, we see that 
 \begin{equation}
 \lambda_{\min} \mu_x \leq  \mu_{x}^{\lambda} \leq \mu_x
\end{equation}  
holds for every $x \in \mathcal{S}_\infty$. Using this observation one can see that the proof of the functional inequalities in~\cite[Section 4]{sapozhnikov2017long-range} hold by potentially adjusting the constants in the definition of (very) regular balls (see~\cite[p.~1874]{sapozhnikov2017long-range}), showing that~\eqref{eq:HK} and~\eqref{eq:HK-time} remain valid (again after possibly after adjusting the constants therein) with  $q^{\omega,\lambda}_{t}$ replacing $q^\omega_t$.
  Leaving the control on the volumes in Proposition~\ref{prop:uni-den} unchanged, one can see that Theorem~\ref{thm:solidification} and Corollary~\ref{cor:Capacity-Lower-Bound} remain valid for $P_x^{\omega,\lambda}$ replacing $P_x^\omega$. As a particular application (choosing the $(\lambda_{e})_{e \in \mathbb{E}^d}$ i.i.d.~random variables under some environment measure $\widehat{\mathbb{Q}}$ and considering $\widehat{\mathbb{Q}} \otimes \mathbb{P}^u$ with $\mathbb{P}^u$ governing Bernoulli bond percolation and $u \in (p_c^{\mathrm{bond}},1)$), we obtain solidification estimates and capacity controls for i.i.d.~edge weights $\mu_{\{x,y\}}^\lambda \in \{ 0 \} \cup [ \lambda_{\min},1] $. This is for instance the framework considered in~\cite{armstrong2018elliptic}, ~\cite{dario-gu2021percolation}, or~\cite{schweiger2024maximum}. 
\item \textit{(Prospective applications)} The solidification estimates and related capacity controls of~\cite{nitzschner2017solidification} and~\cite{CN2021disconnection} have been indispensable for the proofs of certain upper bounds on largely deviant disconnection or isolation-type events in the (strongly) supercritical phases of various long-range percolation models, as well as in the investigation of the picture emerging when conditioning on such events. As one case in point (but see also below for further examples), we consider as in Remark~\ref{rmk:examples} (4) the level set $E^{\geq h}_{\mathrm{GFF}}$ of the Gaussian free field $(\varphi_x)_{x \in \mathbb{Z}^d}$ on $\mathbb{Z}^d$, $d \geq 3$ when $h < h_\ast(d)$ and introduce for $N \geq 1$ the \textit{disconnection event} 
\begin{equation}
\begin{split}
\mathcal{D}_N^h &= \Big\lbrace A_N \stackrel{\geq h}{\centernot \longleftrightarrow}   S_N \Big\rbrace \\
& = \{\text{no path in $E^{\geq h}_{\mathrm{GFF}}$ starting in $A_N$ and ending in $S_N$ exists} \},
\end{split}
\end{equation}
where $S_N = \partial_{\mathbb{Z}^d} B(0,MN)$ with $M > 0$, and $A \subseteq \mathring{B}_{\mathbb{R}^d}(0,M)$ is a compact set with non-empty interior fulfilling $\capa_{\mathbb{R}^d}(A) = \capa_{\mathbb{R}^d}(\mathring{A})$ (and $\capa_{\mathbb{R}^d}$ denoting the Brownian capacity). By~\cite{sznitman2015disconnection} and~\cite{nitzschner2018disconnection} one knows in combination with the sharpness of the phase transition of $E^{\geq h}_{\mathrm{GFF}}$ from~\cite{hugo2023equality} that
\begin{equation}
\label{eq:Disconnection-Statement}
\lim_{N \rightarrow \infty} \frac{1}{N^{d-2}} \log \mathbb{P}[\mathcal{D}_N^h] = - \frac{1}{2d}(h_\ast - h)^2 \capa_{\mathbb{R}^d}(A).
\end{equation}
Furthermore by~\cite{chiarini2020GFF} (again in combination with~\cite{hugo2023equality}) one has that conditionally on $\mathcal{D}_N^h$, the empirical mean $\frac{1}{N^d} \sum_{x \in \mathbb{Z}^d}\varphi_x \delta_{x/N}$ tends to follow a local profile given by $-(h_\ast - h)\mathscr{h}_A$ (with $\mathscr{h}_A$ the harmonic potential of $A$). While upper and lower bounds for~\eqref{eq:Disconnection-Statement} were developed in~\cite{sznitman2015disconnection} in the case where $A = [-1,1]^d$, the upper bound implicit in~\eqref{eq:Disconnection-Statement} requires a capacity bound obtained from solidification estimates in the case where $A$ is non-convex. For the case of the Gaussian free field on $(\mathbb{Z}^d,\mathbb{E}^d,(\lambda_e)_{e \in \mathbb{E}^d})$ with uniformly elliptic random edge weights $\lambda_e \in [\lambda_{\min},1]$, stationary and ergodic under an environment measure $\widehat{\mathbb{Q}}$, quenched upper and lower bounds for the disconnection event in~\eqref{eq:Disconnection-Statement} have been established in~\cite{CN2021disconnection} with the help of a version of solidification estimates for uniformly elliptic edge weights. \medskip

We briefly explain how capacity lower bounds are used to obtain asymptotic upper bounds for the probability of disconnection events, and refer to~\cite[Section 7]{CN2021disconnection} for details. One considers a certain ``effective disconnection event'' $\widetilde{\mathcal{D}}^{h,\lambda}_N$, on which a coarse graining of low enough combinatorial complexity is performed. This results in a decomposition
\begin{equation}
\widetilde{\mathcal{D}}^{h,\lambda}_N = \bigcup_{\kappa \in \mathcal{K}_N } \mathcal{D}_{N,\kappa}^{\lambda},
\end{equation}
with $|\mathcal{K}_N| = o(\exp(N^{d-2}))$ and for any choice of $\kappa \in \mathcal{K}_N$ (corresponding to the coarse graining in effect) on the event $\mathcal{D}_{N,\kappa}^{\lambda}$ an interface $\Sigma_\kappa$ of ``blocking boxes'' is present (see~\cite[(7.44)--(7.45)]{CN2021disconnection}) that corresponds to the porous interfaces considered in Section~\ref{sec:solidification}.  The challenge is then to bound \textit{uniformly} over all possible sets of blocking boxes encoded by $\kappa \in \mathcal{K}_N$ the capacity $\mathrm{\capa}^{\lambda}(\Sigma_\kappa)$ from below by $\mathrm{\capa}^{\lambda}(A'_N)$, where $A' \subseteq \mathring{A}$ is compact. Together with a segmentation $U_1$ (depending also on the coarse-grained configuration $\kappa \in \mathcal{K}_N$) with $d_\infty(A'_N,U_0) \geq cN$, one can then verify that $\Sigma \in \mathcal{S}^{\lambda}_{U_0,10\widehat{L}_0,c}$, where $\widehat{L}_0 = o(N)$ is a ``near macroscopic'' scale (see~\cite[(7.31)]{CN2021disconnection} for its precise definition). This is where the capacity control 
\begin{equation}
\label{eq:Capacity-uniform-elliptic}
\liminf_{N \rightarrow \infty}\inf_{\epsilon/2^{\ell_*}\leq c_N} \inf_{U_0\in \cU_{\ell_*,N}} \inf_{\lambda \in [\lambda_{\min},1]^{\mathbb{E}^d}} \inf_{\Sigma\in \cS_{U_0,\epsilon,\chi}^{\lambda}} \frac{\capa^{\lambda}(\Sigma)}{\capa^{\lambda}(A'_N)} \geq 1,
\end{equation}
see~\cite[(4.8)]{CN2021disconnection} is deployed, with the choices $\epsilon = 10\widehat{L}_0$, $2^{\ell_\ast} \geq c N$, and $c_N = \frac{2^{\ell_\ast}}{N} \to \infty$. \medskip

It is instructive to compare this to our set-up. The capacity bound derived in~\eqref{eq:Capacity-lower-bound-statement} which ought to replace~\eqref{eq:Capacity-uniform-elliptic} in our context holds for $\mathbb{P}^u$-a.e.~$\omega \in \Omega_0$, and involves the additional constraint $a_N \leq 2^{\ell_\ast} \leq b_N$, with $(a_N)_{N \geq 0}$ and $(b_N)_{N \geq 0}$ fulfilling~\eqref{eq:a_N-bound}. Since the application we have in mind requires $\widehat{c} N \leq  2^{\ell_\ast} \leq \widehat{C}N$, one can for instance choose $a_N = \widehat{c}N$ and $b_N= \widehat{C}N^{1+\nu}$ for any $\nu > 0$ and see that~\eqref{eq:a_N-bound} is fulfilled for \textit{any} $\DeltaS > 0$. \medskip

It is therefore quite plausible that the bounds in~\eqref{eq:Capacity-lower-bound-statement} are pertinent to the investigation of disconnection-type events for the level-sets of the Gaussian free field on various percolation clusters, and we hope to return to this elsewhere~\cite{chiarini2025-in-prep}. We also highlight that similar disconnection-type questions have been investigated for the vacant set of random interlacements and the random walk (see~\cite{li2014} and~\cite{li2017} for lower bounds, and~\cite{sznitman2017},
~\cite{nitzschner2017solidification}, and~\cite{chiarini2020entropic} for corresponding upper bounds). In a similar direction, in~\cite{sznitman2019macroscopic} large-deviation type upper and lower bounds were developed for the probability that the adequately thickened connected component of the boundary of a large box centered at the origin in either the vacant set of the random walk or random interlacements, or the upper level set of the GFF leaves in the box a macroscopic volume in its complement. In the derivation of the corresponding upper bounds, capacity-type controls are brought into play in much the same vein as for the previously mentioned disconnection-type estimates. Removing the thickening inherent in the above construction is highly non-trivial, and in fact relies on a more sophisticated control on a certain type of ``bubble set'' present when observing an excess of points disconnected from the boundary of an enclosing box. In the case of random interlacements on $\mathbb{Z}^d$, $d \geq 3$, results of this type have been developed in~\cite{sznitman2023bubble} (see also references therein). Remarkably, the construction of the bubble set is achieved in part by relying on a notion of resonance sets (on $\mathbb{Z}^d$, $d \geq 3$). Methods in this direction developed in the present article may thus also be useful to provide some insight into natural questions pertaining to the excess behavior of random interlacements or the GFF on the infinite cluster of Bernoulli (bond- or site-)percolation.

\end{enumerate} 
\end{remark}

\appendix

\section{Proof of Proposition~\ref{prop:Quantitative-control-volume}}\label{sec:proof-of-QCV}

In this appendix, we prove a large-deviation type estimate for the volume density as stated in Proposition \ref{prop:Quantitative-control-volume}. The proof of this result is a slight generalization of~\cite[Lemma 3.3]{sapozhnikov2017long-range}, and follows along the same lines. \medskip

We first introduce some further notation following~\cite{sapozhnikov2017long-range}. Let $(\ell_n)_{n\geq 0}$, $(r_n)_{n\geq 0}$, and $(L_n)_{n\geq 0}$ be sequences of positive integers, and  consider the rescaled lattice $\bbG_n = L_n \bbZ^d$ as a graph with an edge between each pair of $\ell^1$-nearest neighbors in the latter. For $x\in \bbZ^d$ and $K,s \in \bbZ_+$, we define $Q_{K,s}(x)= x+ \bbZ^d \cap [0,KL_s)^d$ and let $\cC_{K,s,r}(x)$ be the largest connected component in $\cS_r \cap Q_{K,s}(x)$ for $r \in \bbZ_+$ (with ties broken arbitrarily), where we recall that $\cS_r$ is the set of vertices of $\cS$ which are in connected components of $\cS$ with $\ell^1$-diameter greater or equal to $r$. We write $Q_{K,s}$ and $\cC_{K,s,r}$ as a {shorthand} if $x=0$. We recall below the definition of $(\overline{D},\overline{I},n)$-bad events with respect to two generic sequences of events $\overline{D} = \{\overline{D}_{x,L_0} \, : \, L_0 \geq 1, x \in \mathbb{G}_0 \}$ and $\overline{I} = \{\overline{I}_{x,L_0} \, : \, L_0 \geq 1, x \in \mathbb{G}_0 \}$ from~\cite[Definition 2.1]{sapozhnikov2017long-range}. For our purposes, it will be sufficient to consider the specific events in~\eqref{eq:event-D} and~\eqref{eq:event-I}. \smallskip

We fix $\alpha\in (0,1)$ as in the statement of the proposition. Let $\eta_1(\alpha)=\sqrt{1-\frac{\alpha}{2}}\eta(u)$ and $\eta_2(\alpha)=(1+\frac{\alpha}{4})\eta(u)$ (with $\eta(u)$ as in~\eqref{eq:eta}). We define the ``seed events'' similarly as in~\cite[Section 3.1]{sapozhnikov2017long-range}
\begin{align}
  \overline{D}_{x,L_0}^\alpha & \stackrel{\mathrm{def}}{=} \left\{
  \begin{minipage}{0.6\linewidth}
    for $y\in \bbG_0$ with $|y-x|_1\leq L_0$, $\cS_{L_0} \cap (y+[0,L_0)^d$ contains a connected component $\cC_y$ with $|\cC_y|\geq \eta_1(\alpha) L_0^d$ and $\cC_y$ and $\cC_x$ are connected in $\cS\cap (x+[0,L_0)^d) \cup (y+[0,L_0)^d)$
  \end{minipage}
  \right\}^c, \text{ and } \label{eq:event-D}\\
  \overline{I}_{x,L_0}^\alpha & \stackrel{\mathrm{def}}{=} \{|\cS_{L_0} \cap (x+[0,L_0)^d)|> \eta_2(\alpha)L_0^d \},\label{eq:event-I}
\end{align}
for $L_0 \geq 1$ and $x \in \mathbb{G}_0$. We then recursively define for $\overline{E} \in \{\overline{D},\overline{I}\}$ the events $\overline{G}_{x,n,L_0}(\overline{E}^\alpha)$ by $\overline{G}_{x,0,L_0}(\overline{E}^\alpha) = \overline{E}_{x,L_0}^\alpha$ and 
\begin{equation}
  \overline{G}_{x,n,L_0}(\overline{E}^\alpha) = \bigcup_{\substack{x_1, x_2 \in \mathbb{G}_{n-1} \cap\left(x+\left[0, L_n\right)^d\right)\\
  \left|x_1-x_2\right|_{\infty} \geq r_{n-1} \cdot L_{n-1}}}\overline{G}_{x_1,n-1,L_0}(\overline{E}^\alpha) \cap \overline{G}_{x_2,n-1,L_0}(\overline{E}^\alpha), \qquad n \geq 1.
\end{equation}
We say that for $n\geq 0$, $x\in \bbG_n$ is \emph{$n$-bad} if the event $\overline{G}_{x,n,L_0}(\overline{D}^\alpha)\cup\overline{G}_{x,n,L_0}(\overline{I}^\alpha)$ occurs, and \emph{$n$-good} otherwise. \medskip

Heuristically, if all vertices {in} $\bbG_s\cap Q_{K,s}$ are $s$-good, we can iteratively remove ``small'' cubic regions at each scale, resulting in a subgraph $\cQ_{K,s,0}$ of $\bbG_0$ in which all remaining vertices are $0$-good, see \cite[Fig.~2]{sapozhnikov2017long-range} for a visualization. Since only small portions of the domain are removed at each step, the connectivity of $\cQ_{K,s,0}$ is maintained, and its volume remains comparable to that of $Q_{K,s}$. 
Hence, we obtain a globally connected component that satisfies the two-sided volume constraints at small scales. This allows us to conclude that the largest connected component in $Q_{K,s}$ meets the announced volume bounds. \medskip

The difference of our construction of the seed events \eqref{eq:event-D}--\eqref{eq:event-I} compared to \cite{sapozhnikov2017long-range} is the dependence on the additional fixed constant $\alpha \in (0,1)$. In the latter reference, $\eta_1= \frac{3}{4}\eta(u)$ and $\eta_2=\frac{5}{4}\eta(u)$ (see~Subsection 3.1 and (35) therein). We have the following analogue of \cite[Lemma 3.3]{sapozhnikov2017long-range}.

\begin{lemma}\label{lem:vol}
  Assume that the sequences $(r_n)_{n \geq 0}$ and $(\ell_n)_{n \geq 0}$ satisfy
  \begin{equation}\label{eq:prod-con}
    \prod_{i=0}^{\infty} \left[1- \left(\frac{4r_i}{\ell_i}\right)^d \right] > \frac{1+\eta_2(\alpha)}{1+2\eta_1(\alpha)} \vee \sqrt{1-\frac{\alpha}{2}}\vee \left(1-\frac{\alpha}{4}\eta(u)\right).
  \end{equation}
  Let $K,s \in \bbZ_+$ and $x_s\in \bbG_s$. If all the vertices in $\bbG_s\cap Q_{K,s}(x_s)$ are $s$-good, then $\cC_{K,s,L_0}(x_s)$ is uniquely defined and 
  \begin{equation}\label{eq:esti-largest-clust}
     \left(1-\frac{\alpha}{2}\right) \eta(u) \cdot |Q_{K,s}| \leq |\cC_{K,s,L_0}(x_s)|\leq |\cS_{L_0}\cap Q_{K,s}(x_s)|\leq \left(1+\frac{\alpha}{2}\right) \eta(u)\cdot |Q_{K,s}|.
  \end{equation}
\end{lemma}
\begin{proof}
  Without loss of generality, we assume that $x_s=0$. Suppose that all the vertices in $\bbG_s\cap Q_{K,s}$ are $s$-good. According to \cite[Lemma 3.3]{sapozhnikov2017long-range} and its proof  as well as \cite[Corollary 3.4]{sapozhnikov2017long-range}, the cluster $\cC_{K,s,L_0}$ is uniquely defined and fulfills 
  \begin{equation}
  |\cC_{K,s,L_0}| \stackrel{}{\geq} \eta_1 \cdot \prod_{i=0}^{\infty} \left[1- \left(\frac{4r_i}{\ell_i}\right)^d \right] \cdot |Q_{K,s}| \stackrel{\eqref{eq:prod-con}}{\geq} \eta_1 \cdot \sqrt{1-\frac{\alpha}{2}} \cdot |Q_{K,s}|= \left(1-\frac{\alpha}{2}\right)\eta(u)\cdot |Q_{K,s}|,
    \end{equation}
  where we used \cite[(24)]{sapozhnikov2017long-range} in the first inequality and abbreviated $\eta_1(\alpha)$ by $\eta_1$. \smallskip

  On the other hand, abbreviating $\eta_2(\alpha)$ by $\eta_2$, we have
  \begin{align}
    |\cS_{L_0}\cap Q_{K,s}| \leq\left(\eta_2+1-\prod_{i=0}^{\infty}\left(1-\left(\frac{4 r_i}{\ell_i}\right)^d\right)\right) \cdot\left|Q_{K, s}\right| \stackrel{\eqref{eq:prod-con}}{\leq} \left(1+\frac{\alpha}{2}\right)\eta(u)\cdot |Q_{K,s}|,
  \end{align}
  where we used \cite[(25)]{sapozhnikov2017long-range} in the first inequality.
\end{proof}

Note that $\eta_1(\alpha)<\eta(u)<\eta_2(\alpha)$ for $\alpha\in (0,1)$. To conclude, it remains to prove that the probabilistic controls used in~\cite{sapozhnikov2017long-range} or~\cite{alves2019} remain valid when $(\frac{3}{4}\eta(u),\frac{5}{4}\eta(u))$ is replaced by $(\eta_1(\alpha), \eta_2(\alpha))$, where $\alpha \in (0,1)$ is fixed. \medskip

More precisely, we define $\cH_{K,s}^\alpha(x_s)\in \cF$ as in \cite[Definition~3.7]{sapozhnikov2017long-range}: 
\begin{equation}
\begin{split}
  &\cH_{K,s}^\alpha(x_s) \stackrel{\mathrm{def}}{=} \left\{
  \mbox{all the vertices in $\bbG_s \cap (x_s+[-2L_s,((K+2)L_s)^d))$ are $s$-good}
  \right\} \\
  \cap &
  \left\{
  \begin{minipage}{0.85\linewidth}
    all $x,y\in \cS_{L_s} \cap Q_{K,s}(x_s)$ with $|x-y|\leq L_s$ are connected in $\cS\cap B(x,2L_s)$
  \end{minipage}
  \right\}.
  \end{split}
\end{equation}
If $\cH_{K,s}^\alpha(0)\cap \{0\in \cS_\infty\}$ occurs, then by definition $0$ and $\cC_{K,s,L_0}$ are connected, so that $\cC_{K,s,L_0}\subseteq \cS_\infty \cap Q_{K,s}$. From Lemma \ref{lem:vol}, it follows that on $\cH_{K,s}^\alpha(x_s)\cap \{0\in \cS_\infty\}$,
\begin{equation}
  \left(1-\frac{\alpha}{2}\right) \eta(u)\cdot |Q_{K,s}|\leq |Q_{K,s}(x_s) \cap \cS_\infty| \leq  \left(1+\frac{\alpha}{2}\right) \eta(u)\cdot |Q_{K,s}|, 
\end{equation}
provided that \eqref{eq:prod-con} holds. Let $R\geq 1$. We can assume without loss of generality by enlarging $\kappa_{\mathrm{d}1} (\alpha)$ that $R\geq R_0(\alpha)$ is sufficiently large. We choose 
\begin{equation}\label{eq:K-s}
  s = \max\{s' \,:\, L_{s'}^d \leq R\} \quad \text{and} \quad K= \min\{k \,:\, kL_s \geq 2R+1\}+1
\end{equation}
and $x_s\in \bbG_s$ such that $B(0,R)\subseteq Q_{K,s}(x_s)$. Note that 
\[
\frac{|B(0,R)|}{|Q_{K,s}|} \geq \frac{|Q_{K-1,s}|}{|Q_{K,s}|} = \left(1-\frac{1}{K}\right)^d \stackrel{\eqref{eq:K-s}}{\geq} \left(1-\frac{L_s}{2R+1}\right)^d \stackrel{\eqref{eq:K-s}}{\geq} \left(1-\frac{R^{1/d}}{2R+1}\right)^d.
\]
It is therefore possible to choose $R_0(\alpha)$ big enough such that 
\begin{equation}
  (1-\alpha) \eta(u) \cdot |B(0,R)|\leq |\cS_\infty \cap B(0,R)| \leq (1+\alpha) \eta(u) \cdot |B(0,R)|.
\end{equation}
For instance, a possible choice is 
\begin{equation}
  R_0(\alpha) \stackrel{\mathrm{def}}{=} \min \left\{ R_0'(\alpha)\geq 1 : \forall R\geq R_0'(\alpha), \left(1-\frac{R^{1/d}}{2R+1}\right)^d \geq \left(1-\frac{\alpha}{2} \eta(u)\right) \vee \frac{2+\alpha}{2+2\alpha}\right\}.
\end{equation}
It remains to show that
\begin{equation}\label{eq:H_K-s-0}
  \bbP^u_0[\cH_{K,s}^\alpha (x_s)^c] \leq C(\alpha)e^{-c(\alpha) (\log R)^{1+\DeltaS}}
\end{equation}
(recall $\bbP^u_0[\,\cdot\,]= \bbP^u[\,\cdot \, \vert \,  {0}\in \cS_\infty]$). To this end, we let $(\ell_n)_{n \geq 0}$, $(r_n)_{n \geq 0}$, $(L_n)_{n \geq 0}$ be defined as in~\cite[Theorem 6.4]{alves2019}, and then apply it to prove the adjusted version of Lemmas 5.2 and 5.4 in \cite{drewitz2014chemical}, with $\frac{3}{4} \eta(u)$ and $\frac{5}{4} \eta(u)$ replaced by $\eta_1(\alpha)$ and $\eta_2(\alpha)$ respectively. The adjustment is valid since only the relation $\frac{3}{4}\eta(u)<\eta(u)<\frac{5}{4}\eta(u)$ has been used in the proofs to apply the ergodic theorem (see \cite[(5.1)]{drewitz2014chemical}). Hence there exists $C_1(\alpha)$ and $C_2(\alpha,\ell_0(\alpha))$ (the dependence on $u$ is omitted) such that for all $\ell_0, r_0 \geq C_1(\alpha)$, $L_0 \geq C_2(\alpha,\ell_0(\alpha))$, and $n\geq 0$, 
\begin{equation}
  \bbP^u[\mbox{$0$ is $n$-bad}] \leq 2\cdot 2^{-2^n}.
\end{equation}
It then follows by the above inequality, the definition of $\cH_{K,s}^\alpha(x_s)$, and \refSone, that 
\begin{equation}\label{eq:H_K-s}
  \bbP^u_0[\cH_{K,s}^\alpha (x_s)^c] \leq (2K+2)^d \cdot 2\cdot 2^{-2^n} + (KL_s)^d \cdot C \cdot e^{-f_{\mathrm{S}}(u, 2L_s)}.
\end{equation}
To conclude \eqref{eq:H_K-s-0}, it remains to show that the right-hand side of the above equation is bounded from above by $C(\alpha)e^{-c(\alpha) (\log R)^{1+\DeltaS}}$. Comparing \eqref{eq:H_K-s} with the inequality below (39) in \cite{sapozhnikov2017long-range}, the claim can be proved following the arguments presented there. This completes the proof of Proposition \ref{prop:Quantitative-control-volume}. \hfill \qed

\bigskip
\noindent
\textbf{Acknowledgements} While this work was written, AC was associated with INdAM (Istituto Nazionale di Alta Matematica ``Francesco Severi'') and the group GNAMPA. MN was partially supported by Hong Kong RGC (Research Grants Council) grants ECS 26301824 and GRF 16303825. The authors wish to thank the referee for the thorough review of the article and for valuable suggestions. \\

\bibliographystyle{plain}
\bibliography{literature.bib}

@article {abe2015effective,
    AUTHOR = {Abe, Y.},
     TITLE = {Effective resistances for supercritical percolation clusters
              in boxes},
   JOURNAL = {Ann. Inst. Henri Poincar\'e{} Probab. Stat.},
  FJOURNAL = {Annales de l'Institut Henri Poincar\'e{} Probabilit\'es et
              Statistiques},
    VOLUME = {51},
      YEAR = {2015},
    NUMBER = {3},
     PAGES = {935--946},
      ISSN = {0246-0203,1778-7017},
   MRCLASS = {60J45 (60K35)},
  MRNUMBER = {3365968},
MRREVIEWER = {Anatoly\ Yambartsev},
       DOI = {10.1214/14-AIHP604},
       URL = {https://doi.org/10.1214/14-AIHP604},
}

@article {alves2019,
    AUTHOR = {Alves, C. and Sapozhnikov, A.},
     TITLE = {Decoupling inequalities and supercritical percolation for the
              vacant set of random walk loop soup},
   JOURNAL = {Electron. J. Probab.},
  FJOURNAL = {Electronic Journal of Probability},
    VOLUME = {24},
      YEAR = {2019},
     PAGES = {Paper No. 110, 34},
      ISSN = {1083-6489},
   MRCLASS = {60K35 (60G55 60J55)},
  MRNUMBER = {4017128},
       DOI = {10.1214/19-ejp360},
       URL = {https://doi.org/10.1214/19-ejp360},
}

@article {andres2024first,
    AUTHOR = {Andres, S. and Pr\'evost, A.},
     TITLE = {First passage percolation with long-range correlations and
              applications to random {S}chr\"odinger operators},
   JOURNAL = {Ann. Appl. Probab.},
  FJOURNAL = {The Annals of Applied Probability},
    VOLUME = {34},
      YEAR = {2024},
    NUMBER = {2},
     PAGES = {1846--1895},
      ISSN = {1050-5164,2168-8737},
   MRCLASS = {60K35 (60J35 60K37 82B43 82C41)},
  MRNUMBER = {4728159},
MRREVIEWER = {Malin\ P.\ Forsstr\"om},
       DOI = {10.1214/23-aap2008},
       URL = {https://doi.org/10.1214/23-aap2008},
}

@article{andres2025scaling,
  title={Scaling limit of the discrete {G}aussian free field with degenerate random conductances},
  author={Andres, S. and Slowik, M. and Sokol, A.-L.},
  year={2025},
  eprint={2508.17369},
  archivePrefix={arXiv},
  primaryClass={math.PR},
  url={https://arxiv.org/abs/2508.17369},
  note = {Preprint, \href{https://arxiv.org/abs/2508.17369}{arXiv2508.17369}},
}

@article {antal1996chemical,
    AUTHOR = {Antal, P. and Pisztora, A.},
     TITLE = {On the chemical distance for supercritical {B}ernoulli
              percolation},
   JOURNAL = {Ann. Probab.},
  FJOURNAL = {The Annals of Probability},
    VOLUME = {24},
      YEAR = {1996},
    NUMBER = {2},
     PAGES = {1036--1048},
      ISSN = {0091-1798,2168-894X},
   MRCLASS = {60K35 (82B43)},
  MRNUMBER = {1404543},
MRREVIEWER = {Mathew\ D.\ Penrose},
       DOI = {10.1214/aop/1039639377},
       URL = {https://doi.org/10.1214/aop/1039639377},
}

@article {armstrong2018elliptic,
    AUTHOR = {Armstrong, S. and Dario, P.},
     TITLE = {Elliptic regularity and quantitative homogenization on
              percolation clusters},
   JOURNAL = {Comm. Pure Appl. Math.},
  FJOURNAL = {Communications on Pure and Applied Mathematics},
    VOLUME = {71},
      YEAR = {2018},
    NUMBER = {9},
     PAGES = {1717--1849},
      ISSN = {0010-3640,1097-0312},
   MRCLASS = {35B27 (35J25 60K35 82B43)},
  MRNUMBER = {3847767},
MRREVIEWER = {Denis\ I.\ Borisov},
       DOI = {10.1002/cpa.21726},
       URL = {https://doi.org/10.1002/cpa.21726},
}

@article {barlow2004RWpercolation,
    AUTHOR = {Barlow, M. T.},
     TITLE = {Random walks on supercritical percolation clusters},
   JOURNAL = {Ann. Probab.},
  FJOURNAL = {The Annals of Probability},
    VOLUME = {32},
      YEAR = {2004},
    NUMBER = {4},
     PAGES = {3024--3084},
      ISSN = {0091-1798,2168-894X},
   MRCLASS = {60K37 (60K35)},
  MRNUMBER = {2094438},
MRREVIEWER = {Nikita\ Y.\ Ratanov},
       DOI = {10.1214/009117904000000748},
       URL = {https://doi.org/10.1214/009117904000000748},
}

@book {barlow2017heatkernel,
    AUTHOR = {Barlow, M. T.},
     TITLE = {Random walks and heat kernels on graphs},
    SERIES = {London Mathematical Society Lecture Note Series},
    VOLUME = {438},
 PUBLISHER = {Cambridge University Press, Cambridge},
      YEAR = {2017},
     PAGES = {xi+226},
      ISBN = {978-1-107-67442-4},
   MRCLASS = {60-02 (05C63 05C81 60J10 60J45)},
  MRNUMBER = {3616731},
MRREVIEWER = {Nicolas\ Curien},
       DOI = {10.1017/9781107415690},
       URL = {https://doi.org/10.1017/9781107415690},
}

@article {berger2007quenched,
    AUTHOR = {Berger, N. and Biskup, M.},
     TITLE = {Quenched invariance principle for simple random walk on
              percolation clusters},
   JOURNAL = {Probab. Theory Related Fields},
  FJOURNAL = {Probability Theory and Related Fields},
    VOLUME = {137},
      YEAR = {2007},
    NUMBER = {1-2},
     PAGES = {83--120},
      ISSN = {0178-8051,1432-2064},
   MRCLASS = {60F17 (60G50 60K35 82B41 82B43)},
  MRNUMBER = {2278453},
MRREVIEWER = {Nicoletta\ Cancrini},
       DOI = {10.1007/s00440-006-0498-z},
       URL = {https://doi.org/10.1007/s00440-006-0498-z},
}

@article {bolthausen1995entropic,
    AUTHOR = {Bolthausen, E. and Deuschel, J.-D. and Zeitouni,
              O.},
     TITLE = {Entropic repulsion of the lattice free field},
   JOURNAL = {Comm. Math. Phys.},
  FJOURNAL = {Communications in Mathematical Physics},
    VOLUME = {170},
      YEAR = {1995},
    NUMBER = {2},
     PAGES = {417--443},
      ISSN = {0010-3616,1432-0916},
   MRCLASS = {82B20 (60F15 60G15 81T25)},
  MRNUMBER = {1334403},
MRREVIEWER = {Aernout\ C. D. van Enter},
       URL = {http://projecteuclid.org/euclid.cmp/1104273128},
}

@article {bricmont1987percolation,
    AUTHOR = {Bricmont, J. and Lebowitz, J. L. and Maes, C.},
     TITLE = {Percolation in strongly correlated systems: the massless
              {G}aussian field},
   JOURNAL = {J. Statist. Phys.},
  FJOURNAL = {Journal of Statistical Physics},
    VOLUME = {48},
      YEAR = {1987},
    NUMBER = {5-6},
     PAGES = {1249--1268},
      ISSN = {0022-4715,1572-9613},
   MRCLASS = {82A43 (60K35)},
  MRNUMBER = {914444},
MRREVIEWER = {H.\ Kesten},
       DOI = {10.1007/BF01009544},
       URL = {https://doi.org/10.1007/BF01009544},
}

@article {chang2017,
    AUTHOR = {Chang, Y.},
     TITLE = {Supercritical loop percolation on {$\mathbb Z^d$} for {$d\geq
              3$}},
   JOURNAL = {Stochastic Process. Appl.},
  FJOURNAL = {Stochastic Processes and their Applications},
    VOLUME = {127},
      YEAR = {2017},
    NUMBER = {10},
     PAGES = {3159--3186},
      ISSN = {0304-4149,1879-209X},
   MRCLASS = {60K35 (60J10 82B43)},
  MRNUMBER = {3692311},
       DOI = {10.1016/j.spa.2017.02.003},
       URL = {https://doi.org/10.1016/j.spa.2017.02.003},
}

@article {chiarini2016extremes,
    AUTHOR = {Chiarini, A. and Cipriani, A. and Hazra, R. S.},
     TITLE = {Extremes of some {G}aussian random interfaces},
   JOURNAL = {J. Stat. Phys.},
  FJOURNAL = {Journal of Statistical Physics},
    VOLUME = {165},
      YEAR = {2016},
    NUMBER = {3},
     PAGES = {521--544},
      ISSN = {0022-4715,1572-9613},
   MRCLASS = {82C24 (60K35 82C20)},
  MRNUMBER = {3562423},
MRREVIEWER = {B.\ L.\ Granovsky},
       DOI = {10.1007/s10955-016-1634-5},
       URL = {https://doi.org/10.1007/s10955-016-1634-5},
}

@article {CN2021disconnection,
    AUTHOR = {Chiarini, A. and Nitzschner, M.},
     TITLE = {Disconnection and entropic repulsion for the harmonic crystal
              with random conductances},
   JOURNAL = {Comm. Math. Phys.},
  FJOURNAL = {Communications in Mathematical Physics},
    VOLUME = {386},
      YEAR = {2021},
    NUMBER = {3},
     PAGES = {1685--1745},
      ISSN = {0010-3616,1432-0916},
   MRCLASS = {60K37 (60G15 60G60 82B43)},
  MRNUMBER = {4299133},
       DOI = {10.1007/s00220-021-04153-4},
       URL = {https://doi.org/10.1007/s00220-021-04153-4},
}

@article {cosco2021directed,
    AUTHOR = {Cosco, C. and Seroussi, I. and Zeitouni, O.},
     TITLE = {Directed polymers on infinite graphs},
   JOURNAL = {Comm. Math. Phys.},
  FJOURNAL = {Communications in Mathematical Physics},
    VOLUME = {386},
      YEAR = {2021},
    NUMBER = {1},
     PAGES = {395--432},
      ISSN = {0010-3616,1432-0916},
   MRCLASS = {05C63 (60K37 82D60)},
  MRNUMBER = {4287190},
MRREVIEWER = {Adam\ Matthew\ Bowditch},
       DOI = {10.1007/s00220-021-04034-w},
       URL = {https://doi.org/10.1007/s00220-021-04034-w},
}

@article {chiarini2020GFF,
    AUTHOR = {Chiarini, A. and Nitzschner, M.},
     TITLE = {Entropic repulsion for the {G}aussian free field conditioned
              on disconnection by level-sets},
   JOURNAL = {Probab. Theory Related Fields},
  FJOURNAL = {Probability Theory and Related Fields},
    VOLUME = {177},
      YEAR = {2020},
    NUMBER = {1-2},
     PAGES = {525--575},
      ISSN = {0178-8051,1432-2064},
   MRCLASS = {60F10 (60G15 60K35 82B43)},
  MRNUMBER = {4095021},
       DOI = {10.1007/s00440-019-00957-7},
       URL = {https://doi.org/10.1007/s00440-019-00957-7},
}

@article {chiarini2020entropic,
    AUTHOR = {Chiarini, A. and Nitzschner, M.},
     TITLE = {Entropic repulsion for the occupation-time field of random
              interlacements conditioned on disconnection},
   JOURNAL = {Ann. Probab.},
  FJOURNAL = {The Annals of Probability},
    VOLUME = {48},
      YEAR = {2020},
    NUMBER = {3},
     PAGES = {1317--1351},
      ISSN = {0091-1798,2168-894X},
   MRCLASS = {60J27 (60F10 60G60 60K35 82B43)},
  MRNUMBER = {4112716},
       DOI = {10.1214/19-AOP1393},
       URL = {https://doi.org/10.1214/19-AOP1393},
}

@article{chiarini2025-in-prep,
  title={In preparation.},
  author={Chiarini, A. and Liu, Z. and Nitzschner, M.},
}

@article {dario-gu2021percolation,
    AUTHOR = {Dario, P. and Gu, C.},
     TITLE = {Quantitative homogenization of the parabolic and elliptic {G}reen's functions on percolation clusters},
   JOURNAL = {Ann. Probab.},
  FJOURNAL = {The Annals of Probability},
    VOLUME = {49},
      YEAR = {2021},
    NUMBER = {2},
     PAGES = {556--636},
      ISSN = {0091-1798,2168-894X},
   MRCLASS = {60K37 (35J08 60K35)},
  MRNUMBER = {4255127},
       DOI = {10.1214/20-aop1456},
       URL = {https://doi.org/10.1214/20-aop1456},
}

@article {sapozhnikov2017long-range,
    AUTHOR = {Sapozhnikov, A.},
     TITLE = {Random walks on infinite percolation clusters in models with
              long-range correlations},
   JOURNAL = {Ann. Probab.},
  FJOURNAL = {The Annals of Probability},
    VOLUME = {45},
      YEAR = {2017},
    NUMBER = {3},
     PAGES = {1842--1898},
      ISSN = {0091-1798,2168-894X},
   MRCLASS = {60K37 (58J35 60G15 60K35)},
  MRNUMBER = {3650417},
MRREVIEWER = {Ji\v r\'i\ \v Cern\'y},
       DOI = {10.1214/16-AOP1103},
       URL = {https://doi.org/10.1214/16-AOP1103},
}

@article {drewitz2014chemical,
    AUTHOR = {Drewitz, A. and R\'ath, B. and Sapozhnikov,
              A.},
     TITLE = {On chemical distances and shape theorems in percolation models
              with long-range correlations},
   JOURNAL = {J. Math. Phys.},
  FJOURNAL = {Journal of Mathematical Physics},
    VOLUME = {55},
      YEAR = {2014},
    NUMBER = {8},
     PAGES = {083307, 30},
      ISSN = {0022-2488,1089-7658},
   MRCLASS = {60K35 (60G15)},
  MRNUMBER = {3390739},
       DOI = {10.1063/1.4886515},
       URL = {https://doi.org/10.1063/1.4886515},
}

@article{drewitz2014local,
  title={Local percolative properties of the vacant set of random interlacements with small intensity},
  author={Drewitz, A. and R\'ath, B. and Sapozhnikov,
              A.},
  JOURNAL={Ann. Inst. Henri Poincar{\'e} (B) Probab. Stat.},
  volume={50},
  number={4},
  pages={1165--1197},
  year={2014}
}

@misc{goswami2025stronglocaluniquenessvacant,
  title={Strong local uniqueness for the vacant set of random interlacements},
  author={Goswami, S. and Rodriguez, P.-F. and Shulzhenko, Y.},
  year={2025},
  eprint={2503.14497},
  archivePrefix={arXiv},
  primaryClass={math.PR},
  url={https://arxiv.org/abs/2503.14497},
  note = {Preprint, \href{https://arxiv.org/abs/2503.14497}{arXiv:2503.14497}},
}

@article {hugo2023equality,
    AUTHOR = {Duminil-Copin, H. and Goswami, S. and Rodriguez,
              P.-F. and Severo, F.},
     TITLE = {Equality of critical parameters for percolation of {G}aussian
              free field level sets},
   JOURNAL = {Duke Math. J.},
  FJOURNAL = {Duke Mathematical Journal},
    VOLUME = {172},
      YEAR = {2023},
    NUMBER = {5},
     PAGES = {839--913},
      ISSN = {0012-7094,1547-7398},
   MRCLASS = {60K35 (82B43)},
  MRNUMBER = {4568695},
       DOI = {10.1215/00127094-2022-0017},
       URL = {https://doi.org/10.1215/00127094-2022-0017},
}

@article {CN23GMM,
    AUTHOR = {Chiarini, A. and Nitzschner, M.},
     TITLE = {Phase transition for level-set percolation of the membrane
              model in dimensions {$d \ge 5$}},
   JOURNAL = {J. Stat. Phys.},
  FJOURNAL = {Journal of Statistical Physics},
    VOLUME = {190},
      YEAR = {2023},
    NUMBER = {3},
     PAGES = {Paper No. 59, 30},
      ISSN = {0022-4715,1572-9613},
   MRCLASS = {60K35 (60G15 60G60 82B43)},
  MRNUMBER = {4540766},
MRREVIEWER = {Mu\ Fa\ Chen},
       DOI = {10.1007/s10955-023-03072-z},
       URL = {https://doi.org/10.1007/s10955-023-03072-z},
}

@article {mathieu2007quenched,
    AUTHOR = {Mathieu, P. and Piatnitski, A.},
     TITLE = {Quenched invariance principles for random walks on percolation
              clusters},
   JOURNAL = {Proc. R. Soc. Lond. Ser. A Math. Phys. Eng. Sci.},
  FJOURNAL = {Proceedings of The Royal Society of London. Series A.
              Mathematical, Physical and Engineering Sciences},
    VOLUME = {463},
      YEAR = {2007},
    NUMBER = {2085},
     PAGES = {2287--2307},
      ISSN = {1364-5021,1471-2946},
   MRCLASS = {82B41 (60F17 60K35 82B43)},
  MRNUMBER = {2345229},
MRREVIEWER = {Marina\ Vachkovskaia},
       DOI = {10.1098/rspa.2007.1876},
       URL = {https://doi.org/10.1098/rspa.2007.1876},
}

@article {nitzschner2018disconnection,
    AUTHOR = {Nitzschner, M.},
     TITLE = {Disconnection by level sets of the discrete {G}aussian free field and entropic repulsion},
   JOURNAL = {Electron. J. Probab.},
  FJOURNAL = {Electronic Journal of Probability},
    VOLUME = {23},
      YEAR = {2018},
     PAGES = {Paper No. 105, 21},
      ISSN = {1083-6489},
   MRCLASS = {60F10 (60G15 60K35 82B43)},
  MRNUMBER = {3870448},
       DOI = {10.1214/18-ejp226},
       URL = {https://doi.org/10.1214/18-ejp226},
}

@article {nitzschner2017solidification,
    AUTHOR = {Nitzschner, M. and Sznitman, A.-S.},
     TITLE = {Solidification of porous interfaces and disconnection},
   JOURNAL = {J. Eur. Math. Soc. (JEMS)},
  FJOURNAL = {Journal of the European Mathematical Society (JEMS)},
    VOLUME = {22},
      YEAR = {2020},
    NUMBER = {8},
     PAGES = {2629--2672},
      ISSN = {1435-9855,1435-9863},
   MRCLASS = {60J45 (60F10 60K35 82B24)},
  MRNUMBER = {4118617},
       DOI = {10.4171/JEMS/973},
       URL = {https://doi.org/10.4171/JEMS/973},
}

@article {nitzschner2022polymer,
    AUTHOR = {Nitzschner, M.},
     TITLE = {Absence of weak disorder for directed polymers on
              supercritical percolation clusters},
   JOURNAL = {Ann. Inst. Henri Poincar\'e{} Probab. Stat.},
  FJOURNAL = {Annales de l'Institut Henri Poincar\'e{} Probabilit\'es et
              Statistiques},
    VOLUME = {61},
      YEAR = {2025},
    NUMBER = {2},
     PAGES = {1319--1333},
      ISSN = {0246-0203,1778-7017},
   MRCLASS = {82B44 (60K35 82B43)},
  MRNUMBER = {4901645},
       DOI = {10.1214/24-aihp1463},
       URL = {https://doi.org/10.1214/24-aihp1463},
}

@misc{rodriguez2016nabla,
  title={Decoupling inequalities for the {G}inzburg-{L}andau $\nabla \varphi$ models},
  author={Rodriguez, P.-F.},
  year={2016},
  eprint={1612.02385},
  archivePrefix={arXiv},
  primaryClass={math.PR},
  url={https://arxiv.org/abs/1612.02385},
  note = {Preprint, \href{https://arxiv.org/abs/1612.02385}{arXiv:1612.02385}},
}

@article {procaccia2016,
    AUTHOR = {Procaccia, E. B. and Rosenthal, R. and Sapozhnikov,
              A.},
     TITLE = {Quenched invariance principle for simple random walk on
              clusters in correlated percolation models},
   JOURNAL = {Probab. Theory Related Fields},
  FJOURNAL = {Probability Theory and Related Fields},
    VOLUME = {166},
      YEAR = {2016},
    NUMBER = {3-4},
     PAGES = {619--657},
      ISSN = {0178-8051,1432-2064},
   MRCLASS = {60F17 (60G50 60K35 82B41 82B43)},
  MRNUMBER = {3568036},
       DOI = {10.1007/s00440-015-0668-y},
       URL = {https://doi.org/10.1007/s00440-015-0668-y},
}

@article {rath2011transience,
    AUTHOR = {R\'ath, B. and Sapozhnikov, A.},
     TITLE = {On the transience of random interlacements},
   JOURNAL = {Electron. Commun. Probab.},
  FJOURNAL = {Electronic Communications in Probability},
    VOLUME = {16},
      YEAR = {2011},
     PAGES = {379--391},
      ISSN = {1083-589X},
   MRCLASS = {60K35},
  MRNUMBER = {2819660},
MRREVIEWER = {Jos\'e\ Trashorras},
       DOI = {10.1214/ECP.v16-1637},
       URL = {https://doi.org/10.1214/ECP.v16-1637},
}

@article {rodriguez2013phase,
    AUTHOR = {Rodriguez, P.-F. and Sznitman, A.-S.},
     TITLE = {Phase transition and level-set percolation for the {G}aussian
              free field},
   JOURNAL = {Comm. Math. Phys.},
  FJOURNAL = {Communications in Mathematical Physics},
    VOLUME = {320},
      YEAR = {2013},
    NUMBER = {2},
     PAGES = {571--601},
      ISSN = {0010-3616,1432-0916},
   MRCLASS = {60G15 (60K35 82B26 82B43)},
  MRNUMBER = {3053773},
MRREVIEWER = {N.\ Leonenko},
       DOI = {10.1007/s00220-012-1649-y},
       URL = {https://doi.org/10.1007/s00220-012-1649-y},
}

@article {schweiger2024maximum,
    AUTHOR = {Schweiger, F. and Zeitouni, O.},
     TITLE = {The maximum of log-correlated {G}aussian fields in random
              environment},
   JOURNAL = {Comm. Pure Appl. Math.},
  FJOURNAL = {Communications on Pure and Applied Mathematics},
    VOLUME = {77},
      YEAR = {2024},
    NUMBER = {5},
     PAGES = {2778--2859},
      ISSN = {0010-3640,1097-0312},
   MRCLASS = {60G60 (60G15 60G70 60K37)},
  MRNUMBER = {4720221},
MRREVIEWER = {Peter\ Bernard\ Weichman},
       DOI = {10.1002/cpa.22181},
       URL = {https://doi.org/10.1002/cpa.22181},
}

@article {sidoravicius2004quenched,
    AUTHOR = {Sidoravicius, V. and Sznitman, A.-S.},
     TITLE = {Quenched invariance principles for walks on clusters of
              percolation or among random conductances},
   JOURNAL = {Probab. Theory Related Fields},
  FJOURNAL = {Probability Theory and Related Fields},
    VOLUME = {129},
      YEAR = {2004},
    NUMBER = {2},
     PAGES = {219--244},
      ISSN = {0178-8051,1432-2064},
   MRCLASS = {60K35 (60F05 60F17 60G50)},
  MRNUMBER = {2063376},
MRREVIEWER = {Allan\ Gut},
       DOI = {10.1007/s00440-004-0336-0},
       URL = {https://doi.org/10.1007/s00440-004-0336-0},
}

@article {sznitman2010,
    AUTHOR = {Sznitman, A.-S.},
     TITLE = {Vacant set of random interlacements and percolation},
   JOURNAL = {Ann. of Math. (2)},
  FJOURNAL = {Annals of Mathematics. Second Series},
    VOLUME = {171},
      YEAR = {2010},
    NUMBER = {3},
     PAGES = {2039--2087},
      ISSN = {0003-486X,1939-8980},
   MRCLASS = {60K35},
  MRNUMBER = {2680403},
MRREVIEWER = {Ingemar\ Kaj},
       DOI = {10.4007/annals.2010.171.2039},
       URL = {https://doi.org/10.4007/annals.2010.171.2039},
}

@article {sznitman2012decoupling,
    AUTHOR = {Sznitman, A.-S.},
     TITLE = {Decoupling inequalities and interlacement percolation on
              {$G\times\mathbb Z$}},
   JOURNAL = {Invent. Math.},
  FJOURNAL = {Inventiones Mathematicae},
    VOLUME = {187},
      YEAR = {2012},
    NUMBER = {3},
     PAGES = {645--706},
      ISSN = {0020-9910,1432-1297},
   MRCLASS = {60K35},
  MRNUMBER = {2891880},
MRREVIEWER = {Ingemar\ Kaj},
       DOI = {10.1007/s00222-011-0340-9},
       URL = {https://doi.org/10.1007/s00222-011-0340-9},
}

@article {sznitman2019macroscopic,
    AUTHOR = {Sznitman, A.-S.},
     TITLE = {On macroscopic holes in some supercritical strongly dependent
              percolation models},
   JOURNAL = {Ann. Probab.},
  FJOURNAL = {The Annals of Probability},
    VOLUME = {47},
      YEAR = {2019},
    NUMBER = {4},
     PAGES = {2459--2493},
      ISSN = {0091-1798,2168-894X},
   MRCLASS = {60F10 (60G15 60G50 60K35 82B43)},
  MRNUMBER = {3980925},
       DOI = {10.1214/18-AOP1312},
       URL = {https://doi.org/10.1214/18-AOP1312},
}

@article {muirhead2024sharpness,
    AUTHOR = {Muirhead, S.},
     TITLE = {Percolation of strongly correlated {G}aussian fields {II}.
              {S}harpness of the phase transition},
   JOURNAL = {Ann. Probab.},
  FJOURNAL = {The Annals of Probability},
    VOLUME = {52},
      YEAR = {2024},
    NUMBER = {3},
     PAGES = {838--881},
      ISSN = {0091-1798,2168-894X},
   MRCLASS = {60K35 (60G60)},
  MRNUMBER = {4736694},
       DOI = {10.1214/23-aop1673},
       URL = {https://doi.org/10.1214/23-aop1673},
}

@book {grimmett1999,
    AUTHOR = {Grimmett, G.},
     TITLE = {Percolation},
    SERIES = {Grundlehren der mathematischen Wissenschaften [Fundamental
              Principles of Mathematical Sciences]},
    VOLUME = {321},
   EDITION = {Second},
 PUBLISHER = {Springer-Verlag, Berlin},
      YEAR = {1999},
     PAGES = {xiv+444},
      ISBN = {3-540-64902-6},
   MRCLASS = {60K35 (60-02 82B43)},
  MRNUMBER = {1707339},
MRREVIEWER = {Neal\ Madras},
       DOI = {10.1007/978-3-662-03981-6},
       URL = {https://doi.org/10.1007/978-3-662-03981-6},
}

@misc{duminilcopin2023phasetransitionvacantset,
      title={Phase transition for the vacant set of random walk and random interlacements}, 
      author={H. Duminil-Copin and S. Goswami and P.-F. Rodriguez and F. Severo and A. Teixeira},
      year={2023},
      eprint={2308.07919},
      archivePrefix={arXiv},
      primaryClass={math.PR},
      url={https://arxiv.org/abs/2308.07919}, 
      note= {Preprint, \href{https://arxiv.org/abs/2308.07919}{arXiv:2308.07919}}
}

@article {hugo2024,
    AUTHOR = {Duminil-Copin, H. and Goswami, S. and Rodriguez,
              P.-F. and Severo, F. and Teixeira, A.},
     TITLE = {A characterization of strong percolation via disconnection},
   JOURNAL = {Proc. Lond. Math. Soc. (3)},
  FJOURNAL = {Proceedings of the London Mathematical Society. Third Series},
    VOLUME = {129},
      YEAR = {2024},
    NUMBER = {2},
     PAGES = {Paper No. e12622, 49},
      ISSN = {0024-6115,1460-244X},
   MRCLASS = {60K35 (05C81 60G50 82B43)},
  MRNUMBER = {4784258},
MRREVIEWER = {Thomas\ Polaski},
       DOI = {10.1112/plms.12622},
       URL = {https://doi.org/10.1112/plms.12622},
}

@article {duminilcopin2023finiterangeinterlacementscouplings,
    AUTHOR = {Duminil-Copin, H. and Goswami, S. and Rodriguez,
              P.-F. and Severo, F. and Teixeira, A.},
     TITLE = {Finite range interlacements and couplings},
   JOURNAL = {Ann. Probab.},
  FJOURNAL = {The Annals of Probability},
    VOLUME = {53},
      YEAR = {2025},
    NUMBER = {6},
     PAGES = {1987--2053},
      ISSN = {0091-1798,2168-894X},
   MRCLASS = {60K35 (05C80 60G50 82B43)},
  MRNUMBER = {4988271},
       DOI = {10.1214/24-AOP1725},
       URL = {https://doi.org/10.1214/24-AOP1725},
}

@article{chatzigeorgiou2013bounds,
  title={Bounds on the {L}ambert function and their application to the outage analysis of user cooperation},
  author={Chatzigeorgiou, I.},
  JOURNAL = {IEEE Commun. Lett.},
  FJOURNAL = {IEEE Communications Letters},
  volume={17},
  number={8},
  pages={1505--1508},
  year={2013},
  publisher={IEEE}
}

@article {sznitman2015disconnection,
    AUTHOR = {Sznitman, A.-S.},
     TITLE = {Disconnection and level-set percolation for the {G}aussian
              free field},
   JOURNAL = {J. Math. Soc. Japan},
  FJOURNAL = {Journal of the Mathematical Society of Japan},
    VOLUME = {67},
      YEAR = {2015},
    NUMBER = {4},
     PAGES = {1801--1843},
      ISSN = {0025-5645,1881-1167},
   MRCLASS = {60F10 (60G15 60K35 82B43)},
  MRNUMBER = {3417515},
       DOI = {10.2969/jmsj/06741801},
       URL = {https://doi.org/10.2969/jmsj/06741801},
}

@article {li2017,
    AUTHOR = {Li, X.},
     TITLE = {A lower bound for disconnection by simple random walk},
   JOURNAL = {Ann. Probab.},
  FJOURNAL = {The Annals of Probability},
    VOLUME = {45},
      YEAR = {2017},
    NUMBER = {2},
     PAGES = {879--931},
      ISSN = {0091-1798,2168-894X},
   MRCLASS = {60F10 (60J27 60K35 82B43)},
  MRNUMBER = {3630289},
MRREVIEWER = {Jos\'e\ Trashorras},
       DOI = {10.1214/15-AOP1077},
       URL = {https://doi.org/10.1214/15-AOP1077},
}

@article {li2014,
    AUTHOR = {Li, X. and Sznitman, A.-S.},
     TITLE = {A lower bound for disconnection by random interlacements},
   JOURNAL = {Electron. J. Probab.},
  FJOURNAL = {Electronic Journal of Probability},
    VOLUME = {19},
      YEAR = {2014},
     PAGES = {no. 17, 26},
      ISSN = {1083-6489},
   MRCLASS = {60F10 (60J27 60K35)},
  MRNUMBER = {3164770},
MRREVIEWER = {Boualem\ Djehiche},
       DOI = {10.1214/EJP.v19-3067},
       URL = {https://doi.org/10.1214/EJP.v19-3067},
}

@article {sznitman2023bubble,
    AUTHOR = {Sznitman, A.-S.},
     TITLE = {On the cost of the bubble set for random interlacements},
   JOURNAL = {Invent. Math.},
  FJOURNAL = {Inventiones Mathematicae},
    VOLUME = {233},
      YEAR = {2023},
    NUMBER = {2},
     PAGES = {903--950},
      ISSN = {0020-9910,1432-1297},
   MRCLASS = {60K35 (60F10 82B43)},
  MRNUMBER = {4607724},
MRREVIEWER = {Stephen\ Muirhead},
       DOI = {10.1007/s00222-023-01190-9},
       URL = {https://doi.org/10.1007/s00222-023-01190-9},
}

@article {pisztora1996large,
    AUTHOR = {Pisztora, A.},
     TITLE = {Surface order large deviations for {I}sing, {P}otts and
              percolation models},
   JOURNAL = {Probab. Theory Related Fields},
  FJOURNAL = {Probability Theory and Related Fields},
    VOLUME = {104},
      YEAR = {1996},
    NUMBER = {4},
     PAGES = {427--466},
      ISSN = {0178-8051,1432-2064},
   MRCLASS = {82B20 (60F10 60K35 82B43)},
  MRNUMBER = {1384040},
MRREVIEWER = {Francis\ Comets},
       DOI = {10.1007/BF01198161},
       URL = {https://doi.org/10.1007/BF01198161},
}

@article {sznitman2017,
    AUTHOR = {Sznitman, A.-S.},
     TITLE = {Disconnection, random walks, and random interlacements},
   JOURNAL = {Probab. Theory Related Fields},
  FJOURNAL = {Probability Theory and Related Fields},
    VOLUME = {167},
      YEAR = {2017},
    NUMBER = {1-2},
     PAGES = {1--44},
      ISSN = {0178-8051,1432-2064},
   MRCLASS = {60F10 (60J27 60K35 82B43)},
  MRNUMBER = {3602841},
MRREVIEWER = {Eviatar\ B.\ Procaccia},
       DOI = {10.1007/s00440-015-0676-y},
       URL = {https://doi.org/10.1007/s00440-015-0676-y},
}

@article{teixeira2009uniqueness,
  title={On the uniqueness of the infinite cluster of the vacant set of random interlacements},
  author={Teixeira, A.},
 JOURNAL = {Ann. Appl. Probab.},
 VOLUME = {19},
     PAGES = {454--466},
  year={2009}
}

\end{document}